\documentclass[10pt,a4paper]{amsart}
\usepackage[utf8]{inputenc}
\usepackage[T1]{fontenc}
\usepackage{ mathrsfs }
\usepackage{url}
\usepackage{tikz}
\usepackage{tikz-cd}
\usepackage{amsmath}
\usepackage{amsthm}
\usepackage[none]{hyphenat}
\usepackage{pgfplots}
\usepgflibrary{shapes.misc} 
\pgfplotsset{width=11cm,compat=1.15}
\usepgfplotslibrary{fillbetween}
\usepackage{amsfonts}
\usepackage{amssymb}
\usepackage{graphicx}
\usepackage{xfrac}
\usepackage{mathtools}
\newcommand*{\vertbar}{\rule[-1ex]{0.5pt}{2.5ex}}
\newcommand*{\horzbar}{\rule[.5ex]{2.5ex}{0.5pt}}

\theoremstyle{plain}
\newtheorem{thm}{Theorem}[section]
\newtheorem{lem}[thm]{Lemma}
\newtheorem{prop}[thm]{Proposition}

\newtheorem{rmk}[thm]{Remark}
\theoremstyle{definition}
\newtheorem{dfn}[thm]{Definition}
\newtheorem{exmp}[thm]{Example}
\newtheorem{asmp}[thm]{Assumption}

\begin{document}
\title{Singular Lagrangian torus fibrations on the smoothing of algebraic cones}
\author{Santiago Achig-Andrango}
\address{Instituto de Matem\'atica Pura e Aplicada, Rio de Janeiro, Brazil}
\email{esantiag@impa.br} 
\maketitle
\begin{abstract}
	Given a lattice polytope $Q\subset \mathbb{R}^n$, we can consider the cone $\sigma=C(Q)=\{\lambda(q,1)\in \mathbb{R}^{n+1}|\lambda \in \mathbb{R}_{\geq0}, q\in Q\} \subset \mathbb{R}^{n+1}$, and the affine toric variety $Y_{\sigma}$ associated to $\sigma$. In \cite{altmann1997versal}, Altmann showed that the versal deformation space of $Y_\sigma$ can be described by the Minkowski decomposition of the polytope $Q$. Under some conditions on $Q$, we can obtain a smooth deformation $Y_\epsilon$ of $Y_\sigma$ using Altmann's result. In this article, we consider $Y_\epsilon$ inside some $\mathbb{C}^N$ and construct a complex fibration on $Y_\epsilon$, with general fibre $(\mathbb{C}^*)^n$ and finite singular fibres described using global coordinates related to the components of the Minkowski decomposition. We construct a singular Lagrangian torus fibration out of the complex fibration, as in \cite{auroux2007mirror, auroux2009special, gross2001examples, lau2014open}. This singular fibration admits a convex base diagram representation with cuts as a natural generalization of base diagrams described in \cite{symington71four} for Almost Toric Fibrations ($\dim=4$). In particular, we obtain a convex base diagram whose image is the dual cone of $C(Q)$. There is a 1-parameter family of monotone Lagrangian tori in each of these fibrations. Using the wall-crossing formula \cite{pascaleff2020wall}, we describe the potential associated with this family in terms of the Minkowski decomposition of $Q$, recovering the result of \cite{lau2014open}, and discuss non-displaceability. We also discuss some other consequences of our results.
\end{abstract}
	\section{Introduction}
	The SYZ conjecture \cite{strominger1996mirror} gives a geometric interpretation for the mirror symmetry phenomenon first observed in physics \cite{candelas1990calabi, greene1990duality, candelas1991pair}. It claims that if two Calabi-Yau manifolds are the mirror to each other, then both manifolds admit a special Lagrangian torus fibration over the same base and are dual to each other. Under this premise, several refinements, including the need to consider singular Lagrangian fibrations, were developed over the years in a nonexhaustive and long list of work, \cite{gross2001topological, kontsevich2001homological, fukaya2005multivalued, gross2015mirror, tu2014reconstruction, gross2006mirror, gross2011real, gross2018intrinsic, gross2019intrinsic, abouzaid2017family, abouzaid2021homological, auroux2007mirror, auroux2009special, chan2012syz, lau2014open,  abouzaid2020monotone,  chan2016lagrangian}.\\
	\begin{rmk}
		For the definition of a singular Lagrangian fibration, see Definition \ref{defsinglagfib}.
	\end{rmk}
	
	 In the present article, we describe the smoothing of certain Gorenstein singularities based on \cite{altmann1997versal}. We then construct a special Lagrangian torus fibration in the complement of a divisor as in \cite{auroux2007mirror, auroux2009special}. Some of the results and examples in this article are similar to those in \cite{chan2012syz, lau2014open}, which also studies Altmann's smoothing of cones and Lagrangian torus fibrations. However, the proofs of the main theorems are different. In particular, we use an embedding of the $Y_\epsilon$ in some $\mathbb{C}^N$ in order to describe a complex fibration and the singular Lagrangian fibration using global coordinates related to the components of the Minkowski decomposition. The author hopes that this description will help to study the Fukaya and Wrapped Fukaya category of $Y_\epsilon$ using a Morse-Bott-Lefschetz fibration as in \cite{chan2013homological, chan2016lagrangian,abouzaid2020monotone, keating2023symplectomorphisms}. One of the main contributions of the article is to give a local model for the base of the Lagrangian torus fibration similar to Symington's almost toric base diagrams. This paper provides some illustrative examples.
\subsection{Main Results} Given a lattice polytope $Q\subset \mathbb{R}^n$, we can consider $\sigma=C(Q)=\{\lambda(q,1)\in \mathbb{R}^{n+1}|\lambda \in \mathbb{R}_{\geq0}, q\in Q\}$, the cone of $Q$ in $\mathbb{R}^{n+1}$. Since $\sigma$ is a polyhedral rational cone, we can associate an algebraic toric variety $Y_{\sigma}$, singular in general. In \cite{altmann1997versal}, Altmann showed that the versal deformation space of $Y_\sigma$ can be described in terms of a Minkowski decomposition of the polytope $Q$. In this setup, Gross \cite{gross2001examples} gave a recipe to construct special Lagrangian fibrations, and in \cite{chan2012syz, lau2014open} the authors studied the special torus Lagrangian fibration with singularities resulting from \cite{gross2001examples}. This construction can be understood using a complex fibration $\pi: Y_\epsilon \to \mathbb{C}$, where the preimage of a point is isomorphic to $(\mathbb{C}^*)^n$ except at $k$ singular points, where $k$ is the number of components on the Minkowski decomposition. In Theorem \ref{theoremA}, under certain conditions, we describe the local behavior of the singularities of $\pi$, using global coordinates related to the Minkowski decomposition, which guarantees that $Y_\epsilon$ is smooth.
\begin{thm}\label{theoremA}
	Given some conditions in $Q$, a neighborhood of the singular fibres of $\pi:Y_\epsilon\to \mathbb{C}$ is biholomorphic to the preimage of a neighborhood of $0$ under the map  $x_0\cdots x_m$ defined in $\mathbb{C}^{m+1}\times (\mathbb{C}^*)^{n-m}$, for some $m$. In particular, the conditions imply that $Y_\epsilon$ is smooth.
\end{thm}

In \cite{auroux2007mirror, auroux2009special}, Auroux constructed a special singular Lagrangian fibration in the complement of a divisor of $\mathbb{C}^n$. The mentioned fibration is constructed from the complex fibration $\widetilde{\pi}:\mathbb{C}^k\to \mathbb{C}$ given by $\widetilde{\pi}(x_1,\dots,x_k)=x_1\cdots x_k$. We showed that the local model of the singular fibres of the complex fibration $\pi$ is of the form $\mathbb{C}^k\times(\mathbb{C}^*)^{n+1-k}\to \mathbb{C}^k \xrightarrow{\widetilde{\pi}} \mathbb{C}$. So the analogous approach in \cite{auroux2007mirror, auroux2009special} produces a singular Lagrangian fibration as in \cite{chan2012syz, lau2014open}. To visualize symplectic aspects of this singular Lagrangian fibration, we describe it by a convex base diagram with cuts, which can be thought of as analogous to the moment map of a toric action, and is a generalization of the base diagrams constructed by Symington \cite{symington71four} to describe almost toric fibrations. In particular, we show that the image of the diagram is the polyhedral dual to $\sigma$.
\begin{thm} \label{theoremB}
	With the same assumptions as in Theorem \ref{theoremA}, there is a singular Lagrangian fibration in $Y_\epsilon$ that can be represented by a convex diagram $Y_\epsilon \to \mathbb{R}^{n+1}$ with cuts, such that the image of the convex diagram is $\sigma^\vee$ (the dual cone of $\sigma$). In addition, there is a one-parameter family of monotone Lagrangian torus fibres in $Y_\epsilon$.
\end{thm} 
These results allow us to carry out the following immediate applications:
\begin{itemize}
	\item Using the wall-crossing formula \cite{pascaleff2020wall}, we can describe the potential of the one-parameter family of monotone fibres in terms of the Minkowski decomposition of Q. This potential was also described in \cite{lau2014open} using a different technique.
	\item  We show several examples of families of non-displaceable monotone tori.
	\item When $Q$ is itself a polytope associated with the fan of a Fano algebraic variety, we can immediately describe a singular Lagrangian skeleton for which the smoothing $Y_\epsilon$ retracts to. Moreover, there is a natural compactification of $Y_\epsilon$, and by the work of \cite{diogoprep} the potential of the monotone Lagrangian on the compactification $\overline{Y_\epsilon}$ of $Y_\epsilon$ has a term associated with the relative Gromov Witten invariant between $\overline{Y_\epsilon}$ and the compactifying divisor. This term can be readily identified in the potential for the Lagrangian in $Y_\epsilon$, which is described in terms of the Minkowski decomposition of $Q$.
\end{itemize}
Some future developments arise naturally from this work. In \cite{chan2016lagrangian}, the authors discussed homological mirror symmetry for the conifold beginning with a singular Lagrangian fibration as in this article. Also, we hope to be able to use the fibration to describe a Weinstein structure on $Y_\epsilon$ and work towards describing symplectic homology and Wrapped Floer homology for these spaces. We also want to describe a Gelfand-Cetlin fibration in the sense of \cite{shelukhin2018geometry} originated as a limit of the singular Lagrangian fibrations described in this paper. These limits can be thought of as moving the singular fibres toward the boundary of the convex base diagram with cuts. We intend to provide some applications using these Gelfand-Cetlin fibrations.\\

In Section \ref{Background}, we review the basic concepts of toric geometry in the algebro-geometric framework. We also review some ideas and introduce definitions regarding singular Lagrangian fibrations that are useful for this article. In Section \ref{examples}, we describe a toy example to illustrate how the proofs of Theorem \ref{theoremA} and Theorem \ref{theoremB} work. In Section \ref{thm}, we give the proofs of the main theorems. In Section \ref{secwallc}, we describe the potential function of families of monotone Lagrangians in $Y_{\epsilon}$. In Section \ref{otherex}, we present more examples of the application of the main results. In Section \ref{secompact}, we show how to obtain a compactification of $Y_\epsilon$ compatible with the $\mathbb{C}$-fibration under certain conditions (see Theorem \ref{comp1}), and Section \ref{future} presents future research developments.

\subsection*{Acknowledgments} I am deeply indebted to my co-advisor, Renato Vianna, for presenting to me this research project and for all the fantastic support given at every step of the process, including in the writing of this manuscript. I am also thankful to Vinicius Ramos for all the guidance during my graduate studies and, in particular, for introducing me to topics related to symplectic topology and setting me up to be co-advised by Renato Vianna. Special thanks to Eduardo Alves da Silva for solving my doubts about toric algebraic geometry. Thanks to Jonathan Evans and Lu\'is Diogo for taking an interest in my research and for helpful discussions. This work was supported by the CNPq, Conselho Nacional de Desenvolvimento Científico e Tecnológico, Ministry of Science, Technology and Innovation, Brazil.

	\section{Background}
	\label{Background}
	This section aims to review some concepts of toric varieties from algebraic geometry and discuss base diagrams for singular Lagrangians fibrations.
	\subsection{Toric Algebraic Geometry}
	In this subsection, we discuss Toric Varieties. Most parts of the definitions are classical and taken from \cite{cox}.  
 	\begin{dfn}
		The affine variety $ (\mathbb{C}^*)^n $ is a group under component-wise multiplication.	An \textbf{algebraic torus} $ T $ is an affine variety isomorphic to $ (\mathbb{C}^*)^n $, where $ T $ inherits a group structure from the isomorphism.
	\end{dfn}

There are two essential groups associated with an arbitrary algebraic torus $ T $: the group of characters and the group of one-parameter subgroups.

\begin{dfn}
	
	A \textbf{character} of an algebraic torus T is a morphism $ \chi: T \to \mathbb{C}^*  $, as an algebraic variety,  that is a group homomorphism. 
\end{dfn}

For example, if $ T=(\mathbb{C}^*)^n $, $ m = (a_1,\dots,a_n) \in \mathbb{Z}^n  $ gives a character $ \chi^m: (\mathbb{C}^*)^n \to \mathbb{C}^* $ defined by

\begin{equation}\label{defchar}
	\chi^m(t_1,\dots, t_n) = t_1^{a_1}\dots t_n^{a_n}.
\end{equation}

\begin{dfn}

A \textbf{one-parameter subgroup} of an algebraic torus $ T $ is a morphism $ \lambda:\mathbb{C}^* \to T $, as an algebraic variety, that is a group homomorphism. 
	\end{dfn}

For example, if $ T=(\mathbb{C}^*)^n $, $ u = (b_1,\dots,b_n) \in \mathbb{Z}^n  $ gives a one-parameter subgroup $ \lambda^u : \mathbb{C}^* \to (\mathbb{C}^*)^n $ defined by

\begin{equation}\label{defone}
	\lambda^u(t) = (t^{b_1} ,\dots, t^{b_n} ). 
\end{equation}
	
	For an arbitrary algebraic torus $ T $, its characters and the one-parameter subgroups form free abelian groups of rank	equal to the dimension of $ T $. Denote the group of characters and the group of one-parameter subgroups as $M$ and $N$, respectively. We say that $ m\in M $ gives the character $ \chi^m $ and that $ u \in N $ gives the one-parameter subgroup $ \lambda^u:\mathbb{C}^* \to T $.\\

There is a natural bilinear pairing $ \langle, \rangle: M\times N \to \mathbb{C} $ defined as follows:\\
Given a character $ \chi^m $ and a one-parameter subgroup $ \lambda^u $, the composition $  \chi ^m \circ \lambda^u : \mathbb{C}^* \to \mathbb{C}^* $ is a character of $ \mathbb{C}^* $, which is given by $ t \mapsto t^l $ for
some $ l\in \mathbb{Z} $. Then $ \langle m,u \rangle=l.$\\

By \cite[\textsection 16]{humphreys1975linear}, all characters and one-parameter subgroups of $(\mathbb{C}^*)^n$ arise as in (\ref{defchar}) and (\ref{defone}), respectively. Therefore, we can identify $ M $ and $ N $ with $ \mathbb{Z}^n $. The bilinear pairing obtained is the usual dot product
\[ \langle m,u \rangle  = \sum_{i=1}^{n} a_ib_i. \]

These two groups are important because we can construct affine toric varieties from subsets of them, as we will see below. In what follows, we will denote by $ T $ an algebraic torus isomorphic to $ (\mathbb{C}^*)^n  $, $ M $ its character lattice, and $ N $ its group of one-parameter subgroups. 

	\begin{dfn}[{\cite[Definition~1.1.3]{cox}}]
		An \textbf{affine toric variety} is an irreducible affine variety $ V $ containing
		an algebraic torus $ T  $ as a Zariski open subset such that the action of $ T $ on itself extends to an algebraic action of $ T $ on $ V $.
	\end{dfn}
	
	A set $ \mathscr{A} = \{m_1,\dots,m_s\} \subset M  $ gives characters $ \chi^{m_i}: T \to \mathbb{C}^* $. Consider the map
	\[ \Phi_{\mathscr{A}}: T \to (\mathbb{C}^*)^s \]
	defined by
	\[ \Phi_{\mathscr{A}}(t)= (\chi^{m_1}(t),\dots,\chi^{m_s}(t))\in (\mathbb{C}^*)^s\]
\begin{dfn}[{\cite[Definition~1.1.7]{cox}}] \label{toricfromA}
	Given a finite set $ \mathscr{A} \subseteq M $, the affine toric variety $ Y_{\mathscr{A}} $ is defined to be the Zariski closure in $\mathbb{C}^s$ of the image of the map $ \Phi_{\mathscr{A}} $.
\end{dfn}

The following proposition ensures that $ Y_{\mathscr{A}} $ is an affine toric variety.

\begin{prop}[{\cite[Proposition 1.1.8 ]{cox}}]
	
		 Given $ \mathscr{A} \subset M  $ as above, let $ \mathbb{Z}\mathscr{A} \subset M  $ be the sublattice generated by $ \mathscr{A} $. Then $ Y_{\mathscr{A}} $ is an affine toric variety whose algebraic torus (the image of $ \Phi_{\mathscr{A}} $) has character lattice $  \mathbb{Z}\mathscr{A} $. In particular, the dimension of $ Y_{\mathscr{A}} $ is the rank of $ \mathbb{Z}\mathscr{A} $.

\end{prop}

\begin{exmp} \label{exmpva1}
	Set
	\[ \mathscr{A}= \left\{ \begin{pmatrix}
		-1\\
		0\\
		1
	\end{pmatrix}, \begin{pmatrix}
		0\\
		-1\\
		1
	\end{pmatrix}, \begin{pmatrix}
		0\\1\\0
	\end{pmatrix}, \begin{pmatrix}
		1\\0\\0
	\end{pmatrix}   \right\}. \]
	The Zariski closure in $\mathbb{C}^4_{(x,y,z,w)}$ of the image of the map $\Phi_{\mathscr{A}}(t_1,t_2,t_3)=(t_1^{-1}t_3,t_2^{-1}t_3,t_2,t_1)$ is $Y_{\mathscr{A}}=V(xw-yz) \subseteq \mathbb{C}^4 $.
\end{exmp}

From the finite set $ \mathscr{A} = \{m_1,\dots,m_s\}  \subseteq M$, we can also get a projective variety by considering the homomorphism 
\begin{alignat*}{2}
	\pi:(\mathbb{C}^*)^s&\longrightarrow&&\mathbb{P}^{s-1}\setminus V(x_0\cdots x_{s-1})\\
	(t_1,\dots,t_s)&\longmapsto&& [t_1:\dots:t_s]
\end{alignat*}

\begin{dfn}[{\cite[Definition~2.1.1]{cox}}]
	Given a finite set $ \mathscr{A} = \{m_1,\dots,m_s\} \subseteq M $, the \textbf{projective toric variety} $ X_{\mathscr{A}} $ is the Zariski closure in $ \mathbb{P}^{s-1} $ of the image of the map $ \pi \circ \Phi_{\mathscr{A}} $.
\end{dfn}

\begin{prop} [{\cite[Proposition~2.1.4]{cox}}] \label{hom}
	Given $ Y_{\mathscr{A}}$ and $ X_{\mathscr{A}}  $ as above, the following are equivalent:
	\begin{enumerate}
		\item $ Y_{\mathscr{A}} \subseteq \mathbb{C}^s $ is the affine cone of $ X_{\mathscr{A}} \subseteq \mathbb{P}^{s-1} $.
		\item The ideal of $ Y_{\mathscr{A}} $ is homogeneous.
		\item There is $ u\in N $ and $ k>0 $ in $ \mathbb{N} $ such that $ \langle m_i,u \rangle =k $ for $ i=1\dots,s $.
	\end{enumerate}
\end{prop}

\begin{exmp}
	Following Example \ref{exmpva1}, we can consider:
	\[ \pi \circ \Phi_{\mathscr{A}} (t_1,t_2,t_3)=[t_1^{-1}t_3:t_2^{-1}t_3:t_2:t_1] \subseteq \mathbb{P}^3. \]
	The Zariski closure in $\mathbb{P}^3_{(x:y:z:w)}$ of the image of $\pi \circ \Phi_{\mathscr{A}} $ is $X_{\mathscr{A}}=V(xw-yz)\subseteq \mathbb{P}^3$. In this case, $Y_{\mathscr{A}}$ is the affine cone of $X _{\mathscr{A}} $
\end{exmp}

To describe another way to obtain affine toric varieties, we will need the following definitions:

\begin{dfn}
	A \textbf{semigroup} is a set $ S $ with an associative binary operation
	and an identity element. To be an \textbf{affine semigroup}, we also require that:
	\begin{itemize}
		\item The binary operation on $ S $ is commutative. We will write the operation as $ + $
		and the identity element as $ 0 $. Thus a finite set $ \mathscr{A} \subset S  $ gives
		\[ \mathbb{N} \mathscr{A} = \left. \left\{ \sum_{m\in \mathscr{A}} a_mm \right| a_m\in \mathbb{N}\right\} \subseteq S \]
	\item The semigroup is finitely generated, meaning that there is a finite set $ \mathscr{A} \subset S $
	such that $ \mathbb{N} \mathscr{A} = S $.
	\item The semigroup can be embedded in a lattice $ M $.
	\end{itemize}

\end{dfn}

\begin{dfn}
	Given an affine semigroup $ S \subset M, $ the \textbf{semigroup algebra} $ \mathbb{C}[S] $ is the vector
	space over $ \mathbb{C} $ with $ S $ as basis and multiplication induced by the semigroup structure of $ S $.
\end{dfn}

To make this precise, recall that $ m \in M $ gives the character $ \chi^m $. Then
\[ \mathbb{C}[S] = \left. \left\{ \sum_{m\in S} c_m\chi^m \right| c_m \in \mathbb{C} \text{ and } c_m = 0 \text{ for all but finitely many } m \right\} \]

with multiplication induced by
\[ \chi^m\cdot \chi^{m'}=\chi^{m+m'} \]

If $ S = \mathbb{N}\mathscr{A} $ for $ \mathscr{A}= \{m_1,\dots ,m_s\} $, then $ \mathbb{C}[S] =\mathbb{C} [\chi^{m_1},\dots,\chi^{m_s}] $.\\

The following proposition will tell us how to obtain an affine toric variety from an affine semigroup.

\begin{prop}[{\cite[Proposition~1.1.14]{cox}}] \label{atvsemigroup}
	Let $ S \subset M  $ be an affine semigroup. Then:
	\begin{enumerate}
		\item $ \mathbb{C}[S] $ is an integral domain and finitely generated as a $ \mathbb{C} $-algebra.
	
	\item $ \text{Spec} (\mathbb{C}(S)) $ is an affine toric variety whose algebraic torus has character lattice $ \mathbb{Z}S $, and if $ S = \mathbb{N}\mathscr{A} $ for a finite set $ \mathscr{A} \subset M $, then $ \text{Spec} (\mathbb{C}(S))=Y_{\mathscr{A}} $.
	\end{enumerate}
\end{prop}

A fundamental definition in this article is the idea of a convex polyhedral cone and how to obtain an affine toric variety from it. Set $ N_{\mathbb{R}} := N \otimes_{\mathbb{Z}} \mathbb{R} $ and $ M_{\mathbb{R}} := M \otimes_{\mathbb{Z}} \mathbb{R} $.

\begin{dfn}[{\cite[Definition~1.2.1]{cox}}]
	A \textbf{convex polyhedral cone} in $ N_{\mathbb{R}} $ is a set of the form
	\[ \sigma = \text{Cone}(S')=\left. \left\{ \sum_{u\in S'} \lambda_u u \right| \lambda_u \geq 0\right\} \subset N_{\mathbb{R}} \]
		where $ S' \subset N_{\mathbb{R}}  $ is finite. We say that $ \sigma $ is \textbf{generated} by $ S' $. Also set $ \text{Cone}(\emptyset) = \{0\} $. We say that $ \sigma $ is \textbf{rational} if $ S' \subset N $.
\end{dfn}

\begin{dfn}[{\cite[Proposition 1.2.12]{cox}}]
	A convex polyhedral cone $\sigma$ is \textbf{strongly convex} if $\sigma \cap (-\sigma)=\{0\}$.
\end{dfn}

\begin{dfn}[{\cite[Definition~1.2.3]{cox}}]\label{dual}
	Given a polyhedral cone $\sigma \subset N_{\mathbb{R}} $,  its \textbf{dual cone} is defined by
	\[ \sigma^\vee = \{m\in M_{\mathbb{R}}|\langle m,u \rangle\geq0\text{ for all }u\in\sigma\}\]
\end{dfn}

Given a rational polyhedral cone
$ \sigma \subset N_{\mathbb{R}}$, the lattice points
\[ S_{\sigma} := \sigma^\vee \cap M \subset M \]
form a semigroup. A key fact is that this semigroup is finitely generated.

\begin{prop}[Gordan's Lemma. {\cite[Proposition~1.2.17]{cox}}]
	$ S_{\sigma}$ is finitely generated and is an affine semigroup.
\end{prop}

\begin{thm}[{\cite[Theorem~1.2.18]{cox}}] \label{atvcone}
	Let $ \sigma  \subset  N_{\mathbb{R}} \cong \mathbb{R}^n $ be a rational polyhedral cone with semigroup
	$ S_{\sigma} $. Then
	\[ Y_{\sigma}:=Y_{S_{\sigma}} = \text{Spec}(\mathbb{C}[S_{\sigma}])  \]
	is an affine toric variety. 
\end{thm}

From Theorem \ref{atvcone} and Proposition \ref{atvsemigroup}, we conclude that if $\mathscr{A}\subset M$ is a finite set such that $\mathbb{N}\mathscr{A}=S_{\sigma}$, then $Y_{\mathscr{A}}=Y_\sigma$. This is the motivation for the construction of the set $\mathscr{A}_{\mathscr{H}}$ in the proof of Theorem \ref{fibration}. 

\begin{prop}[{\cite[Proposition~1.2.23]{cox}}]
Let $ \sigma \subset N_{\mathbb{R}}  $ be a strongly convex rational polyhedral cone of
maximal dimension. Then
\[ \mathscr{H} = \{m \in S_{\sigma} | m \text{ is irreducible }\} \]
has the following properties:
\begin{enumerate}
	\item $\mathscr{H}$ is finite and generates $ S_{\sigma} $ as a semigroup.
	\item $ \mathscr{H} $ contains the ray generators of the edges of $ \sigma^\vee $.
	\item $ \mathscr{H} $ is the minimal generating set of $ S_{\sigma} $ with respect to inclusion.
\end{enumerate}
\end{prop}

\begin{dfn}
The set $ \mathscr{H} \subset S_{\sigma} $ is called the \textbf{Hilbert basis} of $ S_{\sigma} $ and its elements are the
\textbf{minimal generators} of $ S_{\sigma} $.
\end{dfn}

In this article, we are interested in the affine varieties given by the cone of a polytope, as defined below.
\begin{dfn}[{\cite[Definition~1.2.2]{cox}}]
	A \textbf{polytope} in $ N_{\mathbb{R}} $ is a set of the form
	\[ Q = \text{Conv}(S') = \left. \left\{ \sum_{u\in S'} \lambda_u u \right| \lambda_u \geq 0, \sum_{u\in S'} \lambda_u =1 \right\} \subset N_{\mathbb{R}} \]
	where $ S'  \subset N_{\mathbb{R}}  $ is finite. We say that $ Q $ is the \textbf{convex hull} of $ S' $.
\end{dfn}

A polytope $ Q\subset N_{\mathbb{R}} $ gives a polyhedral cone $ C(Q) \subset N_{\mathbb{R}} \times \mathbb{R} $, called the \textbf{cone of $ Q $} and defined by 
\[ C(Q) = \{\lambda(u,1) \in N_{\mathbb{R}} \times \mathbb{R} | u \in Q, \lambda \geq 0\}. \]

\begin{dfn} \label{coneQ}
	If $ \sigma = C(Q)$, where $ Q $ is a polytope,  we also refer to $ Y_{\sigma} $ as the \textbf{cone of $ Q $}.
\end{dfn}

\begin{exmp}
	Let $Q:= \text{Conv}\{ (0,0),(1,0),(0,1),(1,1) \} \subseteq \mathbb{R}^2$ and $ \sigma = C(Q)$. In this setting, we have
	\[ \sigma=\text{Cone}\{ (0,0,1),(1,0,1),(0,1,1),(1,1,1)  \} \text{ and} \]
	\[\sigma^\vee=\text{Cone}\{ (-1,0,1),(0,-1,1),(0,1,0),(1,0,0) \}.\]
	Finally, we get $Y_{\sigma}=V(xw-yz)$ by Example \ref{exmpva1}.
\end{exmp}

\begin{dfn}\label{minksum}
The Minkowski sum of subsets $ A_1,A_2 \subset M_{\mathbb{R}}  $ is
\[ A_1+A_2 = \{m_1+m_2 | m_1 \in A_1,m_2 \in A_2\}. \]
\end{dfn}

Given polytopes $ M_1 =\text{Conv}(C_1) $ and $ M_2 =\text{Conv}(C_2) $, their Minkowski sum $ Q=M_1+M_2 =
\text{Conv}(C_1 +C_2) $ is again a polytope. If $ Q=M_1 + \dots +M_k $, we say that $ M_1 + \dots +M_k $ is a \textbf{Minkowski decomposition} of $ Q. $\\

In \cite{altmann1997versal}, Altmann studied the cone $ \sigma=C(Q) $, where $ Q \subseteq \mathbb{R}^n  $ is a lattice polygon, i.e. the vertices are contained in $ \mathbb{Z}^n.$ He gave a description of a set of generators of $ \sigma^\vee  $ as follows:\\

To each $ c\in \mathbb{Z}^n $ we associate an integer by $ \eta_0(c):= \max \{
\langle c, -Q \rangle\} $. By the definition of $ \eta_0 $, we have
\[  \partial\sigma^\vee \cap  \mathbb{Z}^{n+1} = \{(c, \eta_0(c)) | c \in \mathbb{Z}^n\} . \]
Moreover, if $ c_1,\dots, c_w \in \mathbb{Z}^n \setminus \textbf{0} $ are those elements producing irreducible pairs
$ (c, \eta_0(c)) $ (i.e. not allowing any non-trivial lattice decomposition $ (c, \eta_0(c)) = (c', \eta_0(c'))+
(c'', \eta_0(c'')) $), then the elements
\[ (c_1, \eta_0(c_1)), \dots, (c_w, \eta_0(c_w)), (\textbf{0},1) \]
form a generator set for $ \sigma^\vee \cap \mathbb{Z}^{n+1} $ as a semigroup.\\

In order to study the versal deformation of $ Y_{\sigma} $, in \cite{altmann1997versal}, Altmann used the cone 
\[ \widetilde{\sigma}=\text{Cone} \left(  \bigcup_{i=1}^k((M_i \cap \mathbb{Z}^n) \times \{e_i \}) \right) \subseteq \mathbb{R}^n\times \mathbb{R}^k .\]
where $ Q=M_1 + \dots +M_k $ is a Minkowski decomposition of $ Q $.\\

In this setup, we define a function $ \phi $ that will play the same role as $ \eta_0 $.
\begin{dfn} \label{phi}
	Let $ Q=M_1 +M_2+ ... +M_k $ be a Minkowski decomposition of $ Q $, and $ n=\dim Q $. Define
	\begin{alignat*}{2}
		\phi:\mathbb{Z}^n&\longrightarrow&&\mathbb{Z}^k\\
		v&\longmapsto&&\sum_{i=1}^{k} \max\{\langle v,-M_i \rangle\} e_i
	\end{alignat*}
	where $ \{e_1,\dots,e_k\} $ is the standard basis of $ \mathbb{Z}^k $.
\end{dfn}

 \subsection{Symplectic Geometry. On base diagrams for singular Lagrangian fibrations} \label{affine}
 
 In this section, we comment on Lagrangian fibrations, as in \cite{symington71four}, and later we provide definitions for what we will call restricted almost toric fibrations.
 
 \begin{dfn}[{\cite[Definition~2.1]{symington71four}}]
	A locally trivial fibration of a symplectic manifold is a \textbf{regular Lagrangian fibration} if the fibres are smooth Lagrangians. A map $ \widetilde{\pi}:M^{2n} \to \widetilde{B}^n $ is a \textbf{Lagrangian fibration} if it restricts to a regular Lagrangian fibration over $ \widetilde{B} \setminus \widetilde{\Sigma} $, where $ \widetilde{B} \setminus \widetilde{\Sigma} $ is an open dense set of $ \widetilde{B} $.
 \end{dfn}
 
  \begin{dfn}[{\cite[Definition~2.3]{miranda2020geometric}}] \label{defsinglagfib}
	A \textbf{singular Lagrangian fibration} of a symplectic manifold
$(M, \omega)$ of dimension $2n$ is a surjective map $\pi: M \to B$, where
$B$ is a topological space of dimension $n$, such that for every point in $B$ there exist an open neighborhood $V \subset B$ and a homeomorphism $\chi: V \to U \subset \mathbb{R}^{n}$ satisfying that $\chi \circ \pi|_{\pi^{-1}(V)} $ is an integrable system on $(\pi^{-1}(V), \omega|_{\pi^{-1}(V )})$.
 \end{dfn}

 \begin{rmk}
	We assume that our fibres are compact and connected, then by Arnold-Liouville, the regular fibres are isomorphic to $ T^n $.
\end{rmk}
 
 The action coordinates, given by flux, allow us to locally identify an open simply connected neighborhood of $ b\in \widetilde{B}\setminus \widetilde{\Sigma} $ with an open set in $ H^1(F_b;\mathbb{R}) $. Hence, $ T_bB \cong H^1(F_b;\mathbb{R}) $ endows a lattice $ H^1(F_b;\mathbb{Z}) $, inducing an affine structure  in $ \widetilde{B}\setminus \widetilde{\Sigma} $.\\
 
 As in \cite{symington71four}, we will consider cuts in the base to obtain an affine embedding to $\mathbb{R}^n$ in the complement of the cuts. The topology of these cuts will be the topology of a Whitney stratified space.
 
 \begin{dfn}
 	A \textbf{cover} of a set $X$ is a collection of subsets of $X$ whose union is all of $X$.
 \end{dfn}
 
\begin{dfn} [{\cite[\textsection 5]{mather1970notes}}]
	A \textbf{stratification} $\mathscr{X}$ of $X\subseteq M$ is a cover of $X$ by pairwise disjoint smooth submanifolds $X_{\alpha}$, $\alpha \in A$.   
\end{dfn}
Let $\Delta_{M}$ be the diagonal subset of $M\times M$. Denote by $F(M)$ the blowing up of $M\times M$ along $\Delta_{M}$. We will use the following identity
\[F(M)=PTM \sqcup ((M\times M) \setminus \Delta_{M}), \]
where $PTM$ denotes the projective tangent bundle of $M$.
\begin{dfn} [{\cite[\textsection 5]{mather1970notes}}]
	$\mathscr{X}$ is a \textbf{Whitney stratification} if it satisfies:
	\begin{enumerate}
		\item (Locally finite) Each point $x \in M$ has a neighborhood $U_x$ such that $U_x \cap X_{\alpha} \neq \emptyset$ for at most finitely many $\alpha \in A$.
		\item (Condition of the frontier) For each $\alpha \in A$, its frontier $(\overline{X_{\alpha}} \setminus X_{\alpha})\cap X= \bigcup_{\beta \in B}X_{\beta}$ for some $B \subseteq A$.
		\item (Whitney's condition B) Let $\{x_i\}$ be a sequence of points in $X_{\alpha}$ and $\{y_i\}$ be a sequence of points in $X_{\beta}$ such that $x_i\neq y_i$. If $\{x_i\}\to y$, $\{y_i\} \to y$, $\{(x_i,y_i)\}$ converges to a line $l\subseteq PTM_y$ in $F(M)$, and $\{ {TX_\alpha}_{x_i} \}$ converges to an $\dim X_\alpha$-plane $\tau \subseteq TM_y$. Then $l\subset \tau$.
	\end{enumerate}
\end{dfn}

\begin{dfn}
	A \textbf{Whitney stratified space} $X$ of $M$ is a subset of $M$ with a Whitney stratification. The dimension of $X$ is defined as the maximum of the dimensions of $X_\alpha$ in the stratification of $X$.
	\end{dfn}

\begin{dfn}
	 A Lagrangian fibration $ \widetilde{\pi} $ admits a \textbf{convex base diagram} if there exists a homeomorphism $ \psi:\widetilde{B}\to B \subseteq  \mathbb{R}^n $ such that $ \pi= \psi \circ \widetilde{\pi}: M \to \mathbb{R}^n $ satisfies that there exists a collection $ \mathscr{C} \supseteq \widetilde{\Sigma} $ of codimension one Whitney stratified spaces, called cuts, such that the symplectic affine structure on $ \psi(\widetilde{B}\setminus \mathscr{C}) $ agrees with the standard affine structure on $ \mathbb{R}^n $.
\end{dfn}

\begin{rmk}
	Almost toric fibrations over disks defined in \cite{symington71four} admit a convex base diagram by taking cuts on eigendirections.
\end{rmk}

\begin{dfn} \label{rest}
	A \textbf{restricted almost-toric fibration} is a singular Lagrangian fibration that admits a convex base diagram whose singular fibres are toric, homeomorphic to $T^n/T^k$ (the group quotient that appears as the standard local model in toric fibrations), or homeomorphic to $T^n/\mathfrak{T}^k$ (the quotient space obtained by collapsing the $k$-cycle $\mathfrak{T}^k$).
\end{dfn}

\begin{rmk}
	There exist ideas of what an almost-toric fibration should be in higher dimensions, in particular, including singular fibres described by Matessi and Castaño-Bernard \cite{bernard2009lagrangian}. Not all are like our local model, that is why we named these restricted.
\end{rmk}

The Lagrangian fibrations we will consider in this paper all have a convex base diagram satisfying the following assumptions on its cuts.

\begin{asmp}
	Let $ \pi: M^{2n} \to B \subseteq \mathbb{R}^n $ be a convex base diagram for a singular Lagrangian fibration. Let $ \Sigma $ be the image of singular fibres, $ S_i $'s the connected components of $ \Sigma=\bigcup_{i}S_i $, and $\mathscr{C}=\bigcup_iC_i\subseteq B$ is the collection of cuts. Assume that each cut $ C_i $ is the cone of $ S_i $ with respect to some $ v_i \in \mathbb{R}^n $, i.e., 
	\[ C_i= \{x+tv_i| x\in S_i, t\in \mathbb{R}_{\geq 0}\} \cap \text{Im}(\pi). \]
	Also that the image of $ C_i $ in $ \mathbb{R}^{n-1}=v_i^{\bot} $ under the projection with respect to $ v_i $ is of the form
	\[ \left. \left\{  \sum_{\sigma \in \mathscr{A}_i}\lambda_{\sigma} \sigma \right| \lambda_{\sigma} \geq 0 \right\}, \]	
	for some finite set $ \mathscr{A}_i \subseteq \mathbb{R}^{n-1} $, satisfying the balancing condition $ \sum_{\sigma \in \mathscr{A}_i} \sigma =0 $. Let $C_i^{\mathbb{R}} = \{x+tv_i|t\in \mathbb{R}\} $ and $ \text{Im}(\pi) \setminus C_i^{\mathbb{R}} = \bigcup_{j=0}^N D_{i,j} $. We assume that $\overline{D_{i,k}} \cap \overline{D_{i,l}} \neq \emptyset, \forall k,l$.
\end{asmp}

\begin{exmp}[{\cite[\textsection 3]{auroux2009special}}] \label{examplexyz}
	Consider the complex fibration
	\begin{align*}
		 f: \mathbb{C}^3\longrightarrow &\mathbb{C}\\
		(x,y,z)\longmapsto &xyz.
	\end{align*}
	There is an action of $T^2$ on each fibre of $f$, let $\mu$ be its moment map and let $\gamma(r)\subseteq \mathbb{C}$ be the circle with center in $1$ and radius $r$. Given real numbers $\delta_1$ and $\delta_2$, we define:
	\[T_{\gamma(r),\delta_1,\delta_2}= f^{-1}(\gamma(r)) \cap \mu^{-1}(\delta_1,\delta_2).\]
	$T_{\gamma(r),\delta_1,\delta_2} $ is an embedded Lagrangian torus in $\mathbb{C}^3$, except possibly when $0\in\gamma(r)$ or in the limit when $r=0$ and we have an isotropic $T^2$.\\
	We have collapsing cycles in the following cases:
\begin{enumerate}
	\item When $x=y=0$ and $z\neq 0$, the collapsing cycle is $\theta_1-\theta_2$.
	\item When $x=z=0$ and $y\neq 0$, the collapsing cycle is $\theta_1$.
	\item When $y=z=0$ and $x\neq 0$, the collapsing cycle is $\theta_2$.
	\item When $x=y=z=0$, all the $T^2$ collapses.
\end{enumerate}
The convex base diagram produced by this singular Lagrangian fibration will have codimension one cuts, in the complement of which we will have a toric structure.  If $ r<1 $, $ T_{\gamma(r),\delta_1,\delta_2} $ determine a toric structure, and we can take action coordinates $ (\lambda_1,\lambda_2, \lambda_3) $ given by the symplectic flux concerning a reference fibre $ T_{\gamma(r_0),0,0} $, as in the action-angle coordinates given in \cite{duistermaat1980global}. Considering the limit when $ r_0=0 $, we can interpret $ \lambda_3 $ as the symplectic area of a disk with boundary in a cycle of $ T_{\gamma(r),\delta_1,\delta_2} $ that collapses as $ r \to 0 $.\\
For $ r\geq 1 $ we then consider cuts (i.e., we disregard fibres) for $ \lambda_1 \geq 0, \lambda_2=0 $, or $ \lambda_1=0, \lambda_2\geq 0 $, or $ \lambda_1=\lambda_2\leq 0 $. This way, we killed monodromies around singular fibres and, hence, possible ambiguity to extend $ \lambda_3 $ via symplectic flux. Note that we can extend $ \lambda_3 $ to a continuous (but not smooth) function on the whole $ \mathbb{C}^3$.\\
We call the image $B$ of $\pi:\mathbb{C}^3\to \mathbb{R}^3$ given by $(\lambda_1,\lambda_2,\lambda_3)$ and denote by $\Sigma$ the locus on $B$ corresponding to the singular Lagrangian fibres.
  Let $S_1$ be the image of the singular fibres corresponding to the case when $0\in\gamma(r)$. The cone $C_1$ is defined with $v_1=(0,0,1)$. We obtain the left diagram of Figure \ref{fig:firstex} as the convex base diagram with cuts.

\begin{figure}[h!]
	
	\centering
\begin{tikzpicture}
	\begin{axis}
		[scale=0.6,
		axis x line=center,
		axis y line=center,
		axis z line=middle,
		zmin=0.0,
		zmax=3.5,
		ticks=none,
		enlargelimits=false,
		view={-30}{-45},
		xlabel = \(\lambda_1\),
		ylabel = {\(\lambda_2\)},
		zlabel = {\(\lambda_3\)}
		]
		\addplot3[name path=A,color=blue,domain=0:10,samples y=1] (x,0, {exp(-x^2/10)});
		\addplot3[name path=C,color=blue,domain=0:10,samples y=1] (0,x, {exp(-x^2/10)});
		\addplot3[name path=D,color=blue,domain=0:10,samples y=1] (-x,-x, {exp(-x^2/10)});
		\addplot3[name path=B,variable=t,smooth,draw=none,domain=0:10]  (t,0, 3);
		\addplot3[name path=E,variable=t,smooth,draw=none,domain=0:10]  (0,t, 3);
		\addplot3[name path=F,variable=t,smooth,draw=none,domain=0:10]  (-t,-t, 3);
		\addplot3 [blue!30,fill opacity=0.5] fill between [of=A and B];
		\addplot3 [blue!30,fill opacity=0.5] fill between [of=C and E];
		\addplot3 [blue!30,fill opacity=0.5] fill between [of=D and F];
		
	\end{axis}
\end{tikzpicture}
\begin{tikzpicture}
	\begin{axis}
			[scale=0.65,
			hide axis,
			zmin=0.0,
			zmax=3.5,
			ticks=none,
			enlargelimits=false,
			view={-15}{-45}
			]
			\addplot3[name path=A,color=blue,domain=0:10,samples y=1] (x,0, {exp(-x^2/10)});
			\addplot3[name path=C,color=blue,domain=0:10,samples y=1] (0,x, {exp(-x^2/10)});
			\addplot3[name path=D,color=blue,domain=0:7,samples y=1] (-x,-x, {exp(-x^2/8)+x});
			\addplot3[name path=B,variable=t,domain=0:10]  (t,0, 0) node [pos=0.997] {$ \begin{pmatrix}
					1\\0\\0
				\end{pmatrix} $};
			\addplot3[name path=E,variable=t,domain=0:10]  (0,t, 0) node [pos=0.992] {$ \begin{pmatrix}
					0\\1\\0
				\end{pmatrix} $};
			\addplot3[name path=F,variable=t,domain=0:7]  (-t,-t, t) node [pos=0.007] {$ \begin{pmatrix}
					-1\\-1\\1
				\end{pmatrix} $};
			\addplot3 [blue!30,fill opacity=0.5] fill between [of=A and B];
			\addplot3 [blue!30,fill opacity=0.5] fill between [of=C and E];
			\addplot3 [blue!30,fill opacity=0.5] fill between [of=D and F];
			
		\end{axis}

\end{tikzpicture}
\caption{Convex base diagrams corresponding to the restricted Lagrangian fibration in Example \ref{examplexyz}}
\label{fig:firstex}
\end{figure}
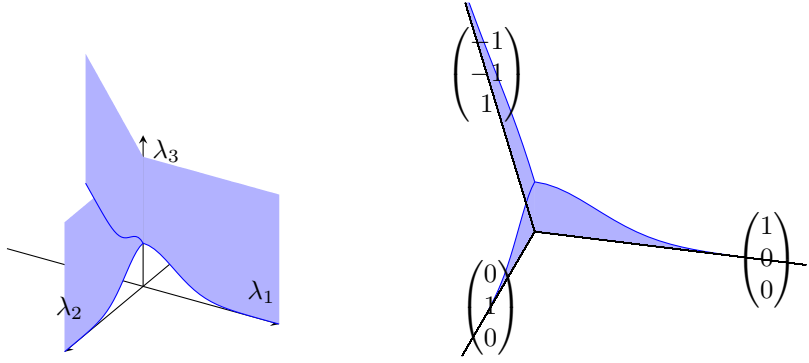

\end{exmp}

\begin{dfn}
	An element $ [\gamma] \in \pi_1(B \setminus \Sigma,b) $ gives us an isotopy class of self-diffeomorphism of the torus fibre $ F_b $. This self-diffeomorphism induces an automorphism on $ H_1(F_b,\mathbb{Z}) $. We call the map from $\pi_1(B\setminus \Sigma)$ to $\text{Aut}(H_1(F_b))$ \textbf{topological monodromy}. A choice of basis for $H_1(F_b)$ allows us to see the topological monodromy as a map from $\pi_1(B\setminus \Sigma)$ to $GL(n,\mathbb{Z})$.
\end{dfn}

To study the topological monodromy of our restricted Lagrangian fibration, consider $  \widetilde{C}_i= \{x-tv_i| x\in S_i, t\in \mathbb{R}_{\geq 0}\} \cap \text{Im}(\pi),   C_i^{\mathbb{R}} = \{x+tv_i|t\in \mathbb{R}\} $ and $ \text{Im}(\pi) \setminus C_i^{\mathbb{R}} = \bigcup_{j=0}^N D_{i,j} $. Fix $ D_{i,0} $, and by the assumption, all $ D_{i,j} $ are adjacent to $ D_{i,0} $. For all $ j=1,\dots,N $, consider the loop $ \mathfrak{b}_ {i,j} \subset \overline{D_{i,0}} \cup \overline{D_{i,j}} $, passing through $ C_i $ once, through $ \widetilde{C}_i $ once, and with orientation given by passing through $ C_i $ from $ D_{i,0} $ to $ D_{i,j}$. Finally, let $  M_{\mathfrak{b}_ {i,j}} $ be the topological monodromy around the loop $ \mathfrak{b}_{i,j}$.\\

Because $ H^1(F_b, \mathbb{Z})=\hom(H_1(F_b,\mathbb{Z}),\mathbb{Z}) $, we see that the affine monodromy (the monodromy for the affine structure in $ B\setminus \Sigma $) concerning each loop $ \mathfrak{b}_ {i,j} \in \pi_1(B\setminus \Sigma) $ is given by the transpose inverse of the monodromies $ M_{\mathfrak{b}_ {i,j}} $. Denote the affine monodromy by $ M_{\mathfrak{b}_ {i,j}}^{\text{af}} $.\\

We now generalize the notion of transferring the cut move defined in \cite{vianna2016infinitely}. We want to change the direction of all the cuts associated with a connected family of singularities, obtaining a different convex base diagram for the same singular Lagrangian fibration.

\begin{dfn} \label{tcut}
	A \textbf{transferring the cut} operation with respect to $ C_i $ consists in applying a piecewise linear map $ \psi_i: \mathbb{R}^n \to \mathbb{R}^n $, preserving $ C_i $, obtaining $ \pi_i = \psi_i \circ \pi : M \to \mathbb{R}^n $ with associated cuts $C'_j $ such that $ C'_j = \psi_i(C_j) $ for $j\neq i$ and $C_i'=\widetilde{C_i}$. The map $ \psi_i $ is defined by applying $ M_{\mathfrak{b}_ {i,j}}^{\text{af}} $ to $ \overline{D_{i,j}} $ for all $j$. Note that this operation fixes $ C_i^{\mathbb{R}} $ and that as a result of applying the monodromies, the affine structure of $ \pi_i $ agrees with the affine structure of $ \mathbb{R}^n $ on $ C_i \setminus \bigcup_{j \neq i} \psi_i (C_j) \subseteq \text{Im} (\pi_i) $, but is not defined in $C_i'$. 
\end{dfn}
	
\begin{exmp}
	The two convex base diagrams in Figure \ref{fig:firstex} are related by transferring the cut operation.
\end{exmp}
\section{Toy example} \label{examples} \label{Q5}
	
In this section, we will work in detail on an example to illustrate how to obtain a restricted Lagrangian fibration on the smoothing of a cone associated with a toric manifold $ Q $.  This will help us settle notation and introduce concepts used in the statement and proof of the main theorems.\\

	Consider $ Q_5:= \text{Conv}\{ (0,0),(1,0),(0,1),(2,1),(1,2) \} \subset \mathbb{R}^2 $ and the following  Minkowski decomposition:
	\[ \vcenter{\hbox{\begin{tikzpicture}
				\draw[fill=gray!25] (0, 0)  -- (1,0) -- (2,1) --(1,2)-- (0,1) -- cycle;
				\path (1,1) -- (0,0) node[pos=1.3]{$ (0,0) $};
				\path (1,1) -- (1,0) node[pos=1.3]{$ (1,0) $};
				\path (1,1) -- (0,1) node[pos=1.5]{$ (0,1) $};
				\path (1,1) -- (2,1) node[pos=1.5]{$ (2,1) $};
				\path (1,1) -- (1,2) node[pos=1.3]{$ (1,2) $};
	\end{tikzpicture}}} = \vcenter{\hbox{\begin{tikzpicture}
	\draw[fill=gray!25] (0, 0)  -- (1,0) -- (0,1) -- cycle;
	\path (1,1) -- (0,0) node[pos=1.3]{$ (0,0) $};
	\path (1,1) -- (1,0) node[pos=1.3]{$ (1,0) $};
	\path (0.5,0.5) -- (0,1) node[pos=1.5]{$ (0,1) $};
\end{tikzpicture}}} + \vcenter{\hbox{\begin{tikzpicture}
\draw[-] (0, 0)  -- (1,1);
\path (1,1) -- (0,0) node[pos=1.3]{$ (0,0) $};
\path (0.5,0.5) -- (1,1) node[pos=1.3]{$ (1,1) $};
\end{tikzpicture}}}\]
We write $ Q_5 =M_1 +M_2$, where $ M_1 =\text{Conv}\{ (0,0),(1,0),(0,1)\} $ and $ M_2 = \text{Conv}\{ (0,0),(1,1) \}  $. Let $ \sigma = C(Q_5) \subseteq N_{\mathbb{R}} \times \mathbb{R} \cong \mathbb{R}^3 $  (see Definition \ref{coneQ}), and $ Y_{\sigma} $ the associated affine toric variety, which is singular. Note that
\[ \sigma = \text{Cone} (\{ (0,0,1),(1,0,1),(0,1,1),(2,1,1),(1,2,1)  \}).  \]

Following \cite{altmann1997versal}, we will describe a way to obtain a deformation of the cone $ Y_{\sigma} $ from the Minkowski decomposition $ Q=M_1+M_2 $. We begin by considering $\mathscr{A}=((M_1\cap \mathbb{Z}^2) \times \{e_1 \}) \cup  ((M_2\cap \mathbb{Z}^2) \times \{e_2 \}) \subseteq \mathbb{Z}^2\times \mathbb{Z}^2$ and  $ \widetilde{\sigma}=\text{Cone} \left(  \mathscr{A} \right) $
where $ \{e_1,e_2\} $ denotes the standard basis of $ \mathbb{Z}^2 $. We see that
\[ \widetilde{\sigma}= \text{Cone} (\{ (0,0;1,0),(1,0;1,0),(0,1;1,0),(0,0;0,1),(1,1;0,1) \}).  \]
It contains the cone $ \sigma  $ via the diagonal embedding $ \mathbb{R}^{2}\times\mathbb{R} \hookrightarrow \mathbb{R}^{2}\times\mathbb{R}^2 $ given by $ (a,1) \mapsto (a,1,1). $\\

We use the computer program Normaliz \cite{normaliz} to obtain the Hilbert Basis of $ S_{\widetilde{\sigma}} = \sigma^\vee \cap \mathbb{Z}^4 $:
\begin{align*}
\mathscr{H}_{\widetilde{\sigma}} = &\{ (-1,-1,1,2), (-1,0,1,1) ,(-1,1,1,0), (0,-1,1,1),\\
 &(1,-1,1,0), (1,0,0,0),(0,1,0,0),(0,0,1,0),(0,0,0,1) \}.
\end{align*}

To understand better Altmann's deformation of $Y_{\sigma}$, we will describe it using characters associated with $Y_{\widetilde{\sigma}}$, see also \cite[\textsection 3 ]{gross2001examples}. In addition,  we will systematically associate some elements  of the Hilbert Basis $ \mathscr{H}_{\widetilde{\sigma}} $ with each term of the Minkowski decomposition in the following manner: To $ M_i $, we associate a $ \dim M_i\times 2 $ matrix $ V_i $, such that the rows of $ V_i $ are the coordinates of the non-zero vertices of $ M_i $:     
      \[   V_1= \begin{pmatrix}
	1 & 0 \\
	0 & 1
\end{pmatrix} \text{,} \quad V_2=\begin{pmatrix}
1 & 1
\end{pmatrix}. \]

Each $V_i$ can be completed to a $2\times 2$ matrix $X_i$ such that the rows of $X_i$ form a basis of $\mathbb{Z}^2$.  Let $Z_i=X_i^{-1}$, define $A_i$ a $2 \times \dim M_i$ matrix and, when $\dim M_i <2$, $C_i$ a $2 \times (2-\dim M_i)$ matrix given by $Z_i=(A_i \quad C_i)$. In our case, we choose 
\[ X_1=Z_1=A_1= \begin{pmatrix}
	1&0 \\
	0&1
\end{pmatrix} \text{,} \quad X_2=Z_2 = \begin{pmatrix}
1 & 1 \\
0 & -1
\end{pmatrix}, \quad A_2=\begin{pmatrix}
	1\\
	0
\end{pmatrix}, \quad C_2=\begin{pmatrix}
	1\\-1
\end{pmatrix} . \]

Recall our definition of the map $ \phi:\mathbb{Z}^2 \to \mathbb{Z}^2 $ associated to a Minkowski decomposition (Definition \ref{phi}) and let $ a_{i,j} \in \mathbb{Z}^2 $ be the $ j $-th column of $ A_i$, then we associate $ \widetilde{a}_{i,j}=(a_{i,j},\phi(a_{i,j})) \in \mathbb{Z}^{4}$  to $ a_{i,j} $. In the same way, let $c_{i,l}$ be the $ l $-th column of $ C_i$ and associate $ (\widetilde{c}^+)_{i,l}=(c_{i,l},\phi(c_{i,l})) \in \mathbb{Z}^{4} $,  $ (\widetilde{c}^-)_{i,l}=(-c_{i,l},\phi(-c_{i,l})) \in \mathbb{Z}^{4}$ to $c_{i,l}$. We also associate to $ M_i $ the vector:
\[\widetilde{b}_i = \left(-\sum_{j=1}^{\dim M_i} a_{i,j}, \phi \left(-\sum_{j=1}^{\dim M_i} a_{i,j} \right) \right) \in \mathbb{Z}^{4}.\]

Set
\[ \mathscr{A}_1=\left\{ \widetilde{a}_{1,1}= \begin{pmatrix}
	1 \\ 0 \\ 0\\0
\end{pmatrix} \text{, }
\widetilde{a}_{1,2}= \begin{pmatrix}
	0 \\ 1 \\ 0\\0
\end{pmatrix} \text{, }
 \widetilde{b}_1= \begin{pmatrix}
 	-1 \\ -1 \\ 1\\2
 \end{pmatrix}\right\}, \]
associated with $ M_1 $,
\[ \mathscr{A}_2= \left\{    
\widetilde{a}_{2,1}= \begin{pmatrix}
	1 \\ 0 \\ 0\\0
\end{pmatrix} \text{, }
\widetilde{b}_2= \begin{pmatrix}
	-1 \\ 0 \\ 1\\1
\end{pmatrix} \text{, }
(\widetilde{c}^+)_{2,1}= \begin{pmatrix}
	1 \\ -1 \\ 1\\0
\end{pmatrix} \text{, }
(\widetilde{c}^-)_{2,1}= \begin{pmatrix}
	-1 \\ 1 \\ 1\\0
\end{pmatrix}
\right\},
  \]
  associated with $ M_2 $, and note that: 
\[ \mathscr{H}_{\widetilde{\sigma}}= \mathscr{A}_1 \cup \mathscr{A}_2 \cup \left\{
\begin{pmatrix}
	0 \\ -1 \\ 1 \\ 1
\end{pmatrix}\text{, }
 \begin{pmatrix}
	0 \\ 0 \\ 1 \\ 0
\end{pmatrix}\text{, }
 \begin{pmatrix}
	0 \\ 0 \\ 0 \\ 1
\end{pmatrix}\right\}.\]

We now associate the characters $ z=\chi ^{(0,-1,1,1)} $, $ x_{i,j}=\chi^{\widetilde{a}_{i,j}} $, $ y_i= \chi^{\widetilde{b}_i}$, $ w^+_{i,j}=\chi^{(\widetilde{c}^+)_{i,j}} $, $ w^-_{i,j}=\chi^{(\widetilde{c}^-)_{i,j}} $, $ t_1=\chi^{(0,0,1,0)} $, and $ t_2=\chi^{(0,0,0,1)} $, to each element of the Hilbert basis $ \mathscr{H}_{\widetilde{\sigma}} $. The characters $ t_i $, associated with the vectors in $ M_i\times \mathbb{Z}^2 $ of the form $ (0,0,e_i) $, will be viewed as deformation parameters of $ Y_{\sigma} $.\\

 Let $ Y_{\widetilde{\sigma}} = \text{Spec}(\mathbb{C}[S_{\widetilde{\sigma}}]) = Y_{\mathscr{H}_{\widetilde{\sigma}}} $, then the ideal $ I(Y_{\widetilde{\sigma}}) $ is generated by: 

\begin{equation} \label{eqwide}
	\begin{aligned}
		&x_{1,1}t_1- x_{1,2}w_{2,1}^+  &\qquad &x_{1,1}y_1-zt_2  &\qquad  &x_{1,1}y_2  - t_1t_2\\
		&x_{1,1}w_{2,1}^--x_{1,2}t_1  &\qquad  &x_{1,1}z-w_{2,1}^+t_2 &\qquad  &x_{1,2}y_1  - y_2t_2\\
		&x_{1,2}y_2  - w_{2,1}^-t_2  &\qquad  &x_{1,2}z  - t_1t_2  &\qquad  &w_{2,1}^+w_{2,1}^-  - t_1^2\\
		&y_2z  -y_1t_1 &\qquad &w_{2,1}^- z  - y_2t_1 &\qquad &y_2w_{2,1}^+ - zt_1 \\
		&y_1w_{2,1}^-  -y_2^2 &\qquad &y_1w_{2,1}^+  - z^2 &\qquad &y_2^2w_{2,1}^+ - w_{2,1}^- z^2.
	\end{aligned}
\end{equation}

 The inclusion $ \sigma \subseteq \widetilde{\sigma} $ induces an embedding $ Y_{\sigma} \subseteq Y_{\widetilde{\sigma}} $ such that $ Y_{\sigma} = Y_{\widetilde{\sigma}} \cap \{t_1-t_2=0)\} $, as it is described in \cite{altmann1997versal}.  Consider now $ \Psi: Y_{\widetilde{\sigma}} \to \mathbb{C} $ given by $ t_1-t_2 $, and $ Y_{\widetilde{\sigma},\epsilon}:=\Psi^{-1}(\epsilon)  $ with $ \epsilon \neq 0 $, a deformation of $ Y_{\sigma} $ (see \cite[\textsection 3 ]{gross2001examples}). In this case, $Y_{\widetilde{\sigma},\epsilon}$ is smooth. We aim to study the complex fibration $ f: Y_{\widetilde{\sigma},\epsilon} \to \mathbb{C}  $ given by the projection to $ t_1 $, which we will eventually use as an auxiliary fibration for constructing a singular Lagrangian fibration as in \cite{auroux2007mirror,auroux2009special,vianna2014exotic}.\\

Note that the singular fibres of $ f $ lie over $ t_1=0 $ and $ t_1=\epsilon $. Indeed, if $ t_1 \neq 0, \epsilon $, $ f^{-1}(t_1) \cong (\mathbb{C}^*)^2 $ because
\[ x_{1,1}y_2=t_1t_2\neq 0,~ w^+_{2,1}w^-_{2,1}=t_1^2 \neq 0, \]
and from the other relations in (\ref{eqwide}), we can show that all the characters in $ \mathscr{H}_{\widetilde{\sigma}} $ will depend only in the non-zero characters $ x_{1,1}$ and $w_{2,1}^+ $.\\

We then look at the singular fibre $ f^{-1}(0) $ and identify it with a subset of $ \mathbb{C}^3 $ given by:   \[x_{1,1}x_{1,2}y_1=t_1t_2^2=0.\]
All the characters in $ \mathscr{H}_{\widetilde{\sigma}} $ will depend only in $ x_{1,1}, x_{1,2}, y_1 $ by the relations:
\[ y_2=x_{1,2}y_1(-\epsilon)^{-1},~ w^+_{2,1}=x_{1,1}^2y_1(\epsilon)^{-2},~ w^-_{2,1}=x_{1,2}^2y_1(\epsilon)^{-2},~z=x_{1,1}y_1(-\epsilon)^{-1}, \]
obtained from (\ref{eqwide}).\\

We identify the singular fibre $ f^{-1}(\epsilon) $ with a subset of $ \mathbb{C}^2 \times \mathbb{C}^* $ given by:   
\[x_{1,1}y_2=t_1t_2=0,~ w^+_{2,1}w^-_{2,1}=t_1^2= \epsilon^2.  \]
All the characters in $ \mathscr{H}_{\widetilde{\sigma}} $ will depend only in $ x_{1,1}, y_2, w^+_{2,1} $ by the relations:
\[x_{1,2}=x_{1,1}w^-_{2,1}\epsilon^{-1},~y_1=y_2^2 w^+_{2,1}\epsilon^{-2},~ z=y_2 w^+_{2,1} \epsilon^{-1}, \]

obtained from (\ref{eqwide}).\\

Now that we understand the complex fibration $ f $, we will construct a Lagrangian fibration associated with it using the same methods used in \cite{auroux2007mirror, auroux2009special, vianna2014exotic}. First, we endow $ Y_{\widetilde{\sigma},\epsilon} $ with the symplectic form coming from the embedding as a subset of $ \mathbb{C}^{|\mathscr{H}_{\widetilde{\sigma}}|} $. Then, we study the action of $ S^1 \times S^1 $ on each fibre of $ f $ given by  
\[ (x_{1,1},x_{1,2},y_1,y_2,w^+_{2,1},w^-_{2,1},z)  \to \] 
\[  (e^{i\theta_1}x_{1,1},e^{i\theta_2}x_{1,2},e^{-i\theta_1}e^{-i\theta_2}y_1,e^{-i\theta_1}y_2, e^{i\theta_1}e^{-i\theta_2}w^+_{2,1},e^{-i\theta_1}e^{i\theta_2}w^-_{2,1},e^{-i\theta_2}z). \]
The moment maps corresponding to the action of $ S^1\times \{1\} $ and $ \{1\} \times S^1 $ are:
\[ \lambda_1(x_{1,1},x_{1,2},y_1,y_2,w^+_{2,1},w^-_{2,1},z) =(|x_{1,1}|^2+|w_{2,1}^+|^2-|y_1|^2-|y_2|^2 - |w_{2,1}^-|^2)/2 \]
\[ \lambda_2(x_{1,1},x_{1,2},y_1,y_2,w^+_{2,1},w^-_{2,1},z) =(|x_{1,2}|^2+|w_{2,1}^-|^2-|z|^2-|y_1|^2 - |w_{2,1}^+|^2)/2 \]
respectively.\\

We will consider a Lagrangian torus, which is contained in $ f^{-1}(\gamma(r)) $ for the circle $ \gamma (r) \subset \mathbb{C} $ with center in $ 1 $ and radius $ r $, and consists of a single $ S^1\times S^1$-orbit inside each fibre of a point of $ \gamma(r) $. Our Lagrangian torus will be given by symplectic parallel transport of the $ T^2 $ orbits along $ \gamma(r) $. 

\begin{dfn}\label{s}
	Given the circle $ \gamma (r) \subset \mathbb{C} $ and real numbers $ \delta_1,\delta_2 $, we define
	\[ T_{\gamma(r),\delta_1,\delta_2}= \{(x_{1,1},x_{1,2},y_1,y_2,w^+_{2,1},w^-_{2,1},z) \in f^{-1}(\gamma(r))| \lambda_1=\delta_1, \lambda_2=\delta_2\}. \]
\end{dfn}

	We showed above that the general fibre of $ f $ is isomorphic to $ (\mathbb{C}^*)^2 $, then $ T_{\gamma(r),\delta_1,\delta_2} $ is an embedded Lagrangian torus in $ Y_{\widetilde{\sigma}} $, except possibly when $ 0 \in \gamma(r) $, $ \epsilon \in \gamma(r) $, or in the limit when $ r=0 $ and we have an isotropic $ T^2 $. Let $ \lambda =2(\lambda_1,\lambda_2) $,
	
	\begin{align*}
		\lambda=& \begin{pmatrix}
		1\\
		0
	\end{pmatrix} |x_{1,1}|^2 + \begin{pmatrix}
	0\\
	1
\end{pmatrix} |x_{1,2}|^2 +\begin{pmatrix}
-1\\
-1
\end{pmatrix} |y_1|^2 + \begin{pmatrix}
-1 \\
0
\end{pmatrix} |y_2|^2 \\&+ \begin{pmatrix}
1\\
-1
\end{pmatrix} |w^+_{2,1}|^2 +\begin{pmatrix}
-1\\
1
\end{pmatrix} |w^-_{2,1}|^2+ \begin{pmatrix}
0\\
-1
\end{pmatrix} |z|^2.
	\end{align*}

We assume $ \epsilon,0 $ are not in the same circle $ \gamma(r) $. Let us now understand the singular Lagrangians for when $ 0\in \gamma (r) $.  Recall that we identify $ f^{-1}(0) $ as a subset of $ \mathbb{C}^3 $ given by
\[ x_{1,1}x_{1,2}y_1=0. \]

Note that if two of the characters $ x_{1,1},x_{1,2},y_1 $ are non-zero, we have a $ T^2 $-orbit, and hence the Lagrangian passing through it is a smooth $ T^3 $. When $ x_{1,1}=x_{1,2}=y_1=0 $, the whole $ T^2 $ collapses. Let us analyze the remaining cases, identifying the cycle in $ T^2 $ that collapses as we approach one of these singular points.\\

When $ x_{1,1}=y_1=0 $ and $ x_{1,2} \neq 0 $, looking at equations (\ref{eqwide}) we see that:
\[ \lambda= \begin{pmatrix}
	0\\
	1
\end{pmatrix} |x_{1,2}|^2. \]
So the collapsing class is the cycle generated by the $\theta_1$ coordinate of the $T^2$ action. We will call it $[\theta_1]$ when viewed in $H_1(T^3,\mathbb{Z}) \supset H_1(T^2,\mathbb{Z})$. Here we think of $T^3$ as a reference Lagrangian fibre with $H_1(T^3,\mathbb{Z})= \langle [\theta_1],[\theta_2],[\theta_3] \rangle$, where $[\theta_3]$ is a cycle corresponding to a choice of a lift of the curve $\gamma$ on the base.\\
When $ x_{1,2}=y_1=0 $ and $ x_{1,1} \neq 0 $, we see that
\[ \lambda= \begin{pmatrix}
	1\\
	0
\end{pmatrix} |x_{1,1}|^2, \]
so the collapsing class is $ [\theta_2] $.\\
When $ x_{1,2}=x_{1,1}=0 $ and $ y_1 \neq 0 $, we see that
\[ \lambda= \begin{pmatrix}
	-1\\
	-1
\end{pmatrix} |y_1|^2, \]
so the collapsing class is $ [\theta_1-\theta_2] $.\\

If $ \epsilon \in \gamma(r) $, recall that we identify $ f^{-1}(\epsilon) $ as a subset of $ \mathbb{C}^2 \times (\mathbb{C})^* $ given by
\[ x_{1,1}y_2=0,~ w^+_{2,1}w^-_{2,1} \neq 0.\]
Note that if one of the characters $ x_{1,1}, y_2 $ is non-zero, we have a $ T^2 $-orbit, and hence the Lagrangian passing through it is a smooth $ T^3 $. We will analyze the remaining case identifying the cycle in $ T^2 $ that collapses as we approach a point for which $ x_{1,1}=y_2=0 $. In this case: 
\[ \lambda= \begin{pmatrix}
	1\\
	-1
\end{pmatrix} |w^+_{2,1}|^2 +\begin{pmatrix}
	-1\\
	1
\end{pmatrix} |w^-_{2,1}|^2, \]
so the collapsing class is $ [\theta_1+\theta_2] $.\\

In summarizing, the singular Lagrangians, which are topologically pinched torus times a circle, are the fibres $ T_{\gamma(r),\delta_1,\delta_2} $ corresponding to:

\begin{enumerate}
	\item $ 0 \in \gamma(r)  $,	$ \lambda_2 = 0 $, and $ \lambda_1>0 $.
	\item $ 0 \in \gamma(r)  $,	$ \lambda_1 = 0 $, and $ \lambda_2>0 $.
	\item $ 0 \in \gamma(r)  $,	$ \lambda_1 -\lambda_2= 0 $, and $ \lambda_1<0.$
	\item $ \epsilon \in \gamma(r)  $ and	$ \lambda_1+\lambda_2 = 0 .$
\end{enumerate}

Now we will study the convex base diagram produced by this singular Lagrangian fibration. For this purpose, we will follow the same approach as in \cite{symington71four}, in the sense that we will consider codimension one cuts in the base of the Lagrangian fibration, in the complement of which we will have a toric structure. Let us now assume that $ |\epsilon-1|>1 $, so the circles $ \gamma(r) $ pass first through 0 and then through $ \epsilon $, as $ r $ increases. So, if $ r<1 $, $ T_{\gamma(r),\lambda_1,\lambda_2} $ determine a toric structure, and we can take action coordinates $ (\lambda_1, \lambda_2, \lambda_3) $ given by the symplectic flux concerning a reference fibre $ T_{\gamma(r_0),0,0} $, as in the action-angle coordinates given in \cite{duistermaat1980global}. Considering the limit when $ r_0=0 $, we can interpret $ \lambda_3 $ as the symplectic area of a disk with boundary in a cycle of $ T_{\gamma(r),\lambda_1,\lambda_2} $ that collapses as $ r \to 0 $.\\

 We will introduce cuts to represent our fibration by a convex base diagram as in Figure \ref{fig:firstex}. For $ r\geq 1 $, we then consider cuts (i.e., we disregard fibres) for $ \lambda_1 \geq 0, \lambda_2=0 $, or $ \lambda_1=0, \lambda_2\geq 0 $, or $ \lambda_1=\lambda_2\leq 0 $. For $ r\geq|1-\epsilon| $, we take cuts for $ \lambda_1=-\lambda_2 $. This way, we killed monodromies around singular fibres and, hence, possible ambiguity to extend $ \lambda_3 $ via symplectic flux. Note that we can extend $ \lambda_3 $ to a continuous (but not smooth) function on the whole $ Y_{\widetilde{\sigma},\epsilon} $.\\

We call the image $ B $ of $ \pi: Y_{\widetilde{\sigma},\epsilon} \to \mathbb{R}^3 $ given by $ (\lambda_1,\lambda_2,\lambda_3) $ and denote by $ \Sigma $ the locus on $ B $ corresponding to singular Lagrangian fibres, as in Section \ref{affine}.  Let $S_1$ be the image of the singular fibres corresponding to the case when $0\in\gamma(r)$, and $S_2$ be the image of the singular fibres corresponding to the case when $\epsilon \in \gamma(r)$. The cones $C_i$ are defined with $v_i=(0,0,1)$, for $i=1,2$. We obtain Figure \ref{fig:cbdQ5} as the convex base diagram with cuts.

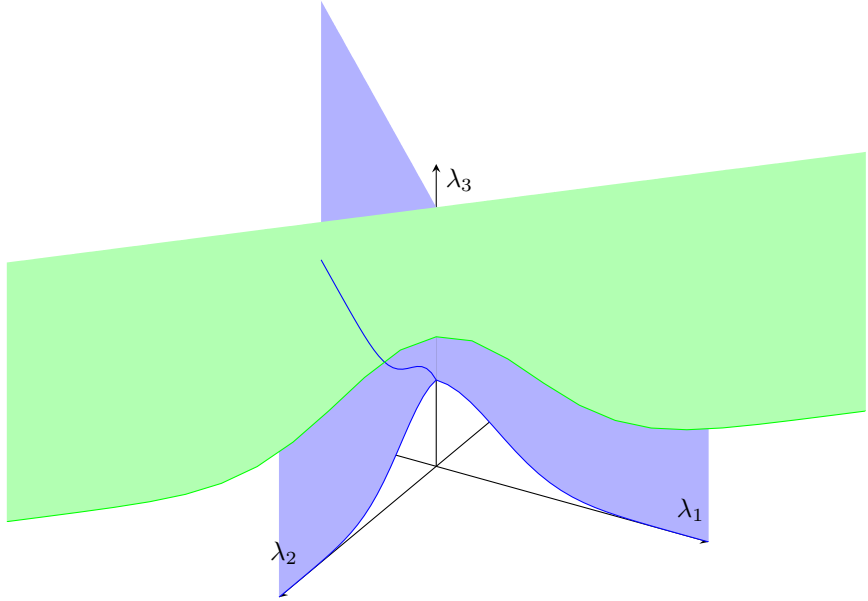
\begin{figure}[h!]
	
	\centering
\begin{tikzpicture}
	\begin{axis}
		[scale=1.2,
		axis x line=center,
		axis y line=center,
		axis z line=middle,
		zmin=0.0,
		zmax=3.5,
		ticks=none,
		enlargelimits=false,
		view={-30}{-45},
		xlabel = \(\lambda_1\),
		ylabel = {\(\lambda_2\)},
		zlabel = {\(\lambda_3\)}
		]
		\addplot3[name path=A,color=blue,domain=0:10,samples y=1] (x,0, {exp(-x^2/10)});
		\addplot3[name path=C,color=blue,domain=0:10,samples y=1] (0,x, {exp(-x^2/10)});
		\addplot3[name path=D,color=blue,domain=0:10,samples y=1] (-x,-x, {exp(-x^2/10)});
		\addplot3[name path=B,variable=t,smooth,draw=none,domain=0:10]  (t,0, 3);
		\addplot3[name path=E,variable=t,smooth,draw=none,domain=0:10]  (0,t, 3);
		\addplot3[name path=F,variable=t,smooth,draw=none,domain=0:10]  (-t,-t, 3);
		\addplot3 [blue!30,fill opacity=0.5] fill between [of=A and B];
		\addplot3 [blue!30,fill opacity=0.5] fill between [of=C and E];
		\addplot3 [blue!30,fill opacity=0.5] fill between [of=D and F];
		
		\addplot3[name path=G,color=green,domain=-10:10,samples y=1] (-x,x, {1.5*exp(-x^2/10)});
		\addplot3[name path=H,variable=t,smooth,draw=none,domain=-10:10]  (t,-t, 3);
		\addplot3 [green!30,fill opacity=0.5] fill between [of=G and H];
		
	\end{axis}

\end{tikzpicture}
\caption{Convex base diagram corresponding to the restricted Lagrangian fibration related with $ Q_5 $}
\label{fig:cbdQ5}
\end{figure}

\begin{rmk}
	One can check that	$ \lambda_3 $ parameter of singular fibres tend to zero as $ |\lambda_1|^2+|\lambda_2|^2 \to \infty $, with $r$ fixed.
\end{rmk}

  We are now going to study the topological monodromy, as defined in Section \ref{affine}, as we go through the cuts of this fibration. Dropping, from now on, the brackets from the notation, let $ \{ \theta_1, \theta_2, \theta_3 \} $ be a basis for $ H_1(T_{\gamma(r),\delta_1,\delta_2}) $ where $\theta_i $ is the circle class of orbit corresponding to $ \lambda_i $ for $ i=1,2 $, and $ \theta_3 $ the collapsing class associated with $ \lambda_3 $. Let
  \[ D_{1,0} = \{ (\lambda_1,\lambda_2,\lambda_3) \in \text{Im}(\pi) | \lambda_1>0,\lambda_2>0 \}, \]   
  \[ D_{1,1} = \{ (\lambda_1,\lambda_2,\lambda_3) \in \text{Im}(\pi) | \lambda_2<0,\lambda_1>\lambda_2\}, \] 
  \[ D_{1,2} = \{ (\lambda_1,\lambda_2,\lambda_3) \in \text{Im}(\pi) | \lambda_1<0,\lambda_2>\lambda_1\}, \] 
  \[ D_{2,0} = \{ (\lambda_1,\lambda_2,\lambda_3) \in \text{Im}(\pi) | \lambda_1+\lambda_2>0 \}, \] 
  \[ D_{2,1} = \{ (\lambda_1,\lambda_2,\lambda_3) \in \text{Im}(\pi) | \lambda_1+\lambda_2<0 \}, \] 
  and define $\mathfrak{b}_{i,j}$ an oriented loop from $D_{i,0}$ to $D_{i,j}$ and back, for $j\neq0$, as in Section \ref{affine}.\\
  
    Taking into account that the topological monodromy leaves the $T^2$ orbit (generated by $\theta_1,\theta_2$) invariant and shears the $\theta_3$ cycle concerning the collapsing cycle, we have that the topological monodromies $ M_{\mathfrak{b}_{i,j}}$ concerning $ \mathfrak{b}_{i,j} $ are:
  
  \[ M_{\mathfrak{b}_{1,1}}=\begin{pmatrix}
  	1 & 0 & 0 \\
  	0 & 1 & 1 \\
  	0 & 0 & 1
  \end{pmatrix}, M_{\mathfrak{b}_{1,2}}=\begin{pmatrix}
  1 & 0 & 1 \\
  0 & 1 & 0 \\
  0 & 0 & 1
\end{pmatrix},   M_{\mathfrak{b}_{2,1}} = \begin{pmatrix}
1 & 0 & 1 \\
0 & 1 & 1 \\
0 & 0 & 1
\end{pmatrix} \]

The action coordinates, given by flux, allow us to locally identify an open simply connected neighborhood of $ b\in B\setminus \Sigma $ with an open set in $ H^1(F_b;\mathbb{R}) $. Hence, $ T_bB \cong H^1(F_b;\mathbb{R}) $ endows a lattice $ H^1(F_b;\mathbb{Z}) $, inducing an affine structure  in $ B\setminus \Sigma  $. Since our Lagrangian fibration $\pi=(\lambda_1,\lambda_2,\lambda_3)$ is induced by flux, the affine structure on $B$ in the complement of the cuts coincide with the standard affine structure in $\mathbb{R}^3$.\\

Because $ H^1(F_b, \mathbb{Z})=\hom(H_1(F_b,\mathbb{Z}),\mathbb{Z}) $ and our choice of basis $\{\theta_1,\theta_2,\theta_3\}$ for $H_1(F_b,\mathbb{Z}) $, associated with the $(\lambda_1,\lambda_2,\lambda_3)$ coordinates of $\mathbb{R}^3$, we see that the affine monodromy (the monodromy for the affine structure in $ B\setminus \Sigma $) about each loop $ \mathfrak{b}_ {i,j} \in \pi_1(B\setminus \Sigma) $ is given by the transpose inverse of the monodromies $ M_{\mathfrak{b}_ {i,j}} $. The corresponding affine monodromies are:
\begin{center}
	\begin{enumerate}
	\item $ M_{\mathfrak{b}_{1,1}}^{\text{af}}= \begin{pmatrix}
		1 & 0 & 0 \\
		0 & 1 & 0 \\
		0 & -1 & 1
	\end{pmatrix} $
with eigenvectors 
$ \begin{pmatrix}
	0 \\
	0 \\
	1
\end{pmatrix} $, $ \begin{pmatrix}
	1 \\
	0 \\
	0
\end{pmatrix} $.
	\item $ M_{\mathfrak{b}_{1,2}}^{\text{af}}= \begin{pmatrix}
		1 & 0 & 0 \\
		0 & 1 & 0 \\
		-1 & 0 & 1
	\end{pmatrix} $
with eigenvectors 
$ \begin{pmatrix}
	0 \\
	0 \\
	1
\end{pmatrix} $, $ \begin{pmatrix}
	0 \\
	1 \\
	0
\end{pmatrix} $.
	\item $ M_{\mathfrak{b}_{2,1}}^{\text{af}}= \begin{pmatrix}
		1 & 0 & 0 \\
		0 & 1 & 0 \\
		-1 & -1 & 1
	\end{pmatrix} $
with eigenvectors 
$ \begin{pmatrix}
	0 \\
	0 \\
	1
\end{pmatrix} $, $ \begin{pmatrix}
	-1 \\
	1 \\
	0
\end{pmatrix} $.
\end{enumerate}
\end{center}

We want to change the direction of all the cuts associated with a connected family of singularities. The result of applying the transferring the cut operations (Definition \ref{tcut}) concerning the two cuts of our fibration $ \pi: Y_{\widetilde{\sigma},\epsilon} \to \mathbb{R}^3 $, is obtained by following the steps:
\begin{enumerate}
	\item Apply the matrix $  M_{\mathfrak{b}_{1,1}}^{\text{af}} $ to $ \overline{D_{1,1} }  $. 
	\item Apply the matrix $  M_{\mathfrak{b}_{1,2}}^{\text{af}} $ to $ \overline{D_{1,2}} $.
	\item Apply the matrix $  M_{\mathfrak{b}_{2,1}}^{\text{af}} $ to the image of $ \overline{D_{2,1}} $ after the first two operations.
\end{enumerate}
The resulting convex base diagram is presented in Figure \ref{fig:cbdQ52}.\\

\begin{figure}[h!]
	\begin{tikzpicture}
		\begin{axis}
			[scale=1.2,
			hide axis,
			zmin=0.0,
			zmax=3.5,
			ticks=none,
			enlargelimits=false,
			view={-15}{-45}
			]
			\addplot3[name path=A,color=blue,domain=0:10,samples y=1] (x,0, {exp(-x^2/10)});
			\addplot3[name path=C,color=blue,domain=0:10,samples y=1] (0,x, {exp(-x^2/10)});
			\addplot3[name path=D,color=blue,domain=0:7,samples y=1] (-x,-x, {exp(-x^2/10)+3*x});
			\addplot3[name path=B,variable=t,domain=0:10]  (t,0, 0) node [pos=0.998] {$ \begin{pmatrix}
					1\\0\\0
				\end{pmatrix} $};
			\addplot3[name path=E,variable=t,domain=0:10]  (0,t, 0) node [pos=0.996] {$ \begin{pmatrix}
					0\\1\\0
				\end{pmatrix} $};
			\addplot3[name path=F,variable=t,domain=0:7]  (-t,-t, 3*t) node [pos=0.003] {$ \begin{pmatrix}
					-1\\-1\\3
				\end{pmatrix} $};
			\addplot3 [blue!30,fill opacity=0.5] fill between [of=A and B];
			\addplot3 [blue!30,fill opacity=0.5] fill between [of=C and E];
			\addplot3 [blue!30,fill opacity=0.5] fill between [of=D and F];
			
			\addplot3[name path=G,color=green,domain=0:10,samples y=1] (-x,x, {1.5*exp(-x^2/10)+x});
			\addplot3[name path=J,color=green,domain=0:10,samples y=1] (x,-x, {1.5*exp(-x^2/10)+x});
			\addplot3[name path=H,variable=t,domain=0:10]  (t,-t, t) node [pos=0.0055] {$ \begin{pmatrix}
					1\\-1\\1
				\end{pmatrix} $};
			\addplot3[name path=I,variable=t,domain=0:10]  (-t,t, t) node [pos=0.0055] {$ \begin{pmatrix}
					-1\\1\\1
				\end{pmatrix} $};
			\addplot3 [green!30,fill opacity=0.5] fill between [of=G and I];
			\addplot3 [green!30,fill opacity=0.5] fill between [of=J and H];
			
		\end{axis}
	\end{tikzpicture}
\caption{Convex base diagram related with Figure \ref{fig:cbdQ5} after applying the transferring the cut operations.}
\label{fig:cbdQ52}
\end{figure}
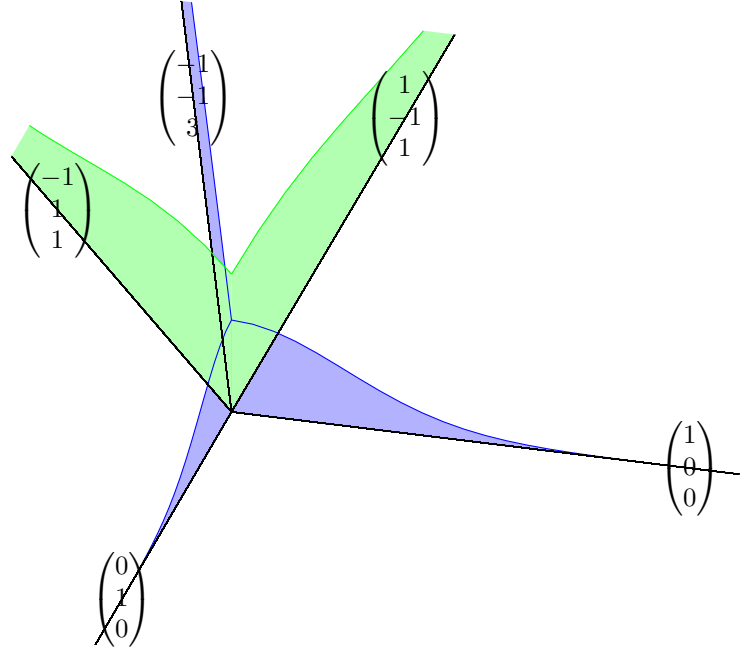

Note that the cone presented in Figure \ref{fig:cbdQ52} coincides with:
\[ \sigma^\vee =\text{Cone}\left\{  \begin{pmatrix}
	-1\\-1\\3
\end{pmatrix},
\begin{pmatrix}
	-1\\0\\2
\end{pmatrix},
\begin{pmatrix}
	-1\\1\\1
\end{pmatrix},
\begin{pmatrix}
	0\\-1\\2
\end{pmatrix},
\begin{pmatrix}
	0\\0\\1
\end{pmatrix},
\begin{pmatrix}
	0\\1\\0
\end{pmatrix},
\begin{pmatrix}
	1\\-1\\1
\end{pmatrix},
\begin{pmatrix}
	1\\0\\0
\end{pmatrix} \right\} \]
After applying the matrix 
\[ \mathscr{M}=\begin{pmatrix}
	1 & 0 & 0 \\
	0 & 1 & 0 \\
	1 & 1 & 1
	
\end{pmatrix} \in SL(3,\mathbb{Z})\]
to $ \sigma^\vee $, we obtain:
\[ \mathscr{M}\sigma^\vee =\text{Cone}\left\{  \begin{pmatrix}
	-1\\-1\\1
\end{pmatrix},
\begin{pmatrix}
	-1\\0\\1
\end{pmatrix},
\begin{pmatrix}
	-1\\1\\1
\end{pmatrix},
\begin{pmatrix}
	0\\-1\\1
\end{pmatrix},
\begin{pmatrix}
	0\\0\\1
\end{pmatrix},
\begin{pmatrix}
	0\\1\\1
\end{pmatrix},
\begin{pmatrix}
	1\\-1\\1
\end{pmatrix},
\begin{pmatrix}
	1\\0\\1
\end{pmatrix} \right\}, \]
where we see it as a cone having $ Q_5^\vee $ (Figure \ref{fig:Q5dual}) at height 1.\\

\begin{figure}[h!]
\centering
	\begin{tikzpicture}
	\draw[fill=gray!25] (-1,-1)  -- (-1,0) -- (-1,1) --(0,1)-- (1,0) -- (1,-1) -- (0,-1) -- cycle;
	\path (0,0) -- (-1,-1) node[pos=1.3]{$ (-1,-1) $};
	\path (0,0) -- (-1,1) node[pos=1.3]{$ (-1,1) $};
	\path (0,0) -- (0,1) node[pos=1.3]{$ (0,1) $};
	\path (0,0) -- (1,0) node[pos=1.5]{$ (1,0) $};
	\path (0,0) -- (1,-1) node[pos=1.3]{$ (1,-1) $};
\end{tikzpicture} 
\caption{$ Q_5^\vee $}
\label{fig:Q5dual}
\end{figure}

Section \ref{thm} aims to construct a Lagrangian fibration on the smoothing of a cone singularity associated with the Minkowski decomposition in a general setting and to prove a duality result relating the convex base diagram to the dual of the original cone.

	\section{Main Results} \label{thm}
	
	We are interested in lattice polytopes $ Q $ that have a particular type of Minkowski decomposition, as we define below.
	
	\begin{dfn} \label{admdecomp}
		We call a Minkowski decomposition $ Q=M_1 + \dots+ M_k \subset \mathbb{R}^n $ \textbf{admissible} if:
		\begin{itemize}
			\item Each $ M_i $ is a lattice polytope that contains the origin as a vertex.
			\item The non-zero vertices $ v_{i,1},\dots,v_{i,m_i} $ of $ M_i $ are linearly independent, and there exist $v_{i,m_i+1},\dots,v_{i,n}\in \mathbb{Z}^n$ such that $v_{i,1},\dots,v_{i,n}$ are a basis for the lattice $\mathbb{Z}^n$ (all the $M_i$ of the Minkowski decomposition are congruent to standard simplices of dimension $m_i$ in $\mathbb{R}^{m_i}\times \mathbb{R}^{n-m_i} $ under the action of $SL_n(\mathbb{Z})$).
		\end{itemize}
		Henceforth, $m_i$ refers to the number of non-zero vectors in the lattice polytope $M_i$, and $V_i$ denotes the $m_i \times n$ matrix with rows representing the non-zero vertices of $M_i$.
	\end{dfn}

\begin{rmk}
	An admissible decomposition of $Q$ implies that Altmann's versal deformation provides a smoothing of $Y_{\sigma}$, as demonstrated in the following Theorem \ref{fibration}.
\end{rmk}

Given a matrix $V_i$ as in Definition \ref{admdecomp}, there is a $(n-m_i) \times n$ matrix $E_i$ such that the rows of the matrix
\[ X_i=\begin{pmatrix}
	V_i\\
	E_i
\end{pmatrix} \]
is a basis of $\mathbb{Z}^n$. The matrix $X_i$ is invertible, and its inverse $Z_i$ has integer entries. Let $A_i$ be a $n\times m_i$ matrix and $C_i$ be a $n \times (n-m_i)$ matrix defined by the formula:
\[Z_i=\begin{pmatrix}
	A_i & C_i
\end{pmatrix}. \]
We obtain the following relations:
\begin{equation} \label{relationsmatrix}
	\begin{pmatrix}
	V_i\\
	E_i
\end{pmatrix} \begin{pmatrix}
	A_i & C_i
\end{pmatrix} = \begin{pmatrix}
	V_i\cdot A_i &V_i \cdot C_i\\
	E_i \cdot A_i & E_i \cdot C_i
\end{pmatrix}= \begin{pmatrix}
	Id&0\\
	0 & Id
\end{pmatrix}.
\end{equation}

\begin{dfn}\label{matrixesVi}
	After fixing a choice for $E_i$, we call the triple of matrices $(A_i, C_i, E_i)$ satisfying \eqref{relationsmatrix}, the matrices associated with $V_i$.
\end{dfn}

\subsection{Auxiliary complex fibration} \label{subacf}
The theorem below provides a complex fibration on a smoothing of $Y_\sigma$. This result was already known (see \cite{lau2014open}), however, in our proof we find explicit global coordinates associated with the terms in the Minkowski decomposition. These coordinates are used to analyze the singular Lagrangian torus fibration in Section \ref{singlagfib}.

	\begin{thm} \label{fibration}
		Let $ Q=M_1 +M_2+ ... +M_k $ be an admissible Minkowski decomposition of $ Q $. Then, there exists a complex fibration on a deformation $Y_{\widetilde{\sigma},\epsilon}$ of $Y_{\sigma}$, the toric variety associated with $\sigma=C(Q)$, over $ \mathbb{C} $ such that it associates to each $ M_i $ a singular fibre that is isomorphic to the subvariety $ x_0\dots x_{m_i}=0 $ inside $ \mathbb{C}^{m_i+1} \times(\mathbb{C}^*)^{n-m_i} $, and the general fibre is $ (\mathbb{C}^*)^n $ where $ n=\dim Q $. In particular, we see that $Y_{\widetilde{\sigma},\epsilon}$ is smooth.
	\end{thm}
	\begin{proof}
		Let $ \{e_1,\dots,e_k\} $ denotes the standard basis of $ \mathbb{Z}^k $, and define
		
		\[ \widetilde{\sigma}=\text{Cone} \left(  \bigcup_{i=1}^k((M_i \cap \mathbb{Z}^n) \times \{e_i \}) \right) \subseteq \mathbb{R}^n\times \mathbb{R}^k .\]
		Let $ \mathscr{H} $ be the Hilbert Basis of $ \widetilde{\sigma}^\vee $. To have a set of characters that describes our fibration, we are going to define a new set of generators  $ \mathscr{A}_{\mathscr{H}} $  as follows:
		\begin{enumerate}
			\item The vectors $ (\textbf{0},e_i) $, $ \textbf{0} \in \mathbb{Z}^n $ for $ i=1,\dots,k $ are elements of $ \mathscr{A}_{\mathscr{H}} $.
			\item For each $ (v,w) \in \mathscr{H} $ with $ v\in \mathbb{Z}^n \setminus \textbf{0}  $ and $ w \in \mathbb{Z}^k  $,  $ (v,\phi(v)) $ is an element of $ \mathscr{A}_{\mathscr{H}} $. Recall the definition of $\phi$ from Definition \ref{phi}.
			\item Let $ V_i $ be the matrix associated to $ M_i $, and consider its associated matrices $A_i$ and $C_i$ (Definition \ref{matrixesVi}). Let $ a_{i,j} $ be the columns of $ A_i $, $c_{i,l}$ be the columns of $C_i$, and $ b_i = - \sum_{j=1}^{m_i} a_{i,j}  $. The vectors $ \widetilde{a}_{i,j}=(a_{i,j}, \phi(a_{i,j}) ) $, $ (\widetilde{c}^+)_{i,l}=(c_{i,l},\phi(c_{i,l})) $,  $ (\widetilde{c}^-)_{i,l}=(-c_{i,l},\phi(-c_{i,l})) $, and $ \widetilde{b}_i=(b_i, \phi(b_i)) $ are elements of $  \mathscr{A}_{\mathscr{H}} $.
		\end{enumerate} 
		\[V_i=\begin{pmatrix}
				\horzbar & v_{i,1} & \horzbar \\
				 	& \vdots &  \\
				\horzbar & v_{i,m_i} & \horzbar
			\end{pmatrix}, \quad A_i= \begin{pmatrix}
			\vertbar & & \vertbar\\
			a_{i,1} & \dots & a_{i,m_i} \\ 
			\vertbar & & \vertbar
		\end{pmatrix}, 
			\]
			\[C_i= \begin{pmatrix}
			\vertbar & & \vertbar\\
			c_{i,1} & \dots & c_{i,n-m_i} \\ 
			\vertbar & & \vertbar
		\end{pmatrix}.\]
	\begin{lem}
		The set $\mathscr{A}_{\mathscr{H}} $ generates $\widetilde{\sigma}^\vee \cap M $, i.e, $ \mathbb{N}  \mathscr{A}_{\mathscr{H}}  = \widetilde{\sigma}^\vee \cap M $.
	\end{lem}
\begin{proof}
	Our strategy is to prove that $ \mathscr{A}_{\mathscr{H}} \subseteq \mathbb{N} \mathscr{H} $, and $ \mathscr{H} \subseteq  \mathbb{N} \mathscr{A}_{\mathscr{H}} $. Then, we conclude that $ \mathbb{N}  \mathscr{A}_{\mathscr{H}}  = \mathbb{N} \mathscr{H} = \widetilde{\sigma}^\vee \cap M  $.\\
	First, we are going to prove the inclusion $ \mathscr{A}_{\mathscr{H}} \subseteq \mathbb{N} \mathscr{H}= \widetilde{\sigma}^\vee \cap M$. By Definition \ref{dual}, it is equivalent to prove that $ \langle a , m \rangle \geq 0, \forall a\in \mathscr{A}_{\mathscr{H}},\forall m \in \widetilde{\sigma} $. Recall that in our construction the elements of $ \mathscr{A}_{\mathscr{H}} $ are of the form $ (\textbf{0},e_i)  $, where $ \textbf{0} \in \mathbb{Z}^n $ and $ \{e_1,\dots,e_k\} $ is the standard basis of $ \mathbb{Z}^k $,  or of the form $ (v,\phi(v)) $ where $ v \in \mathbb{Z}^n \setminus \textbf{0} $. It is clear that $ \langle (\textbf{0},e_i) , m \rangle \geq 0, \forall m\in\widetilde{\sigma}   $. For $ (v,\phi(v)) $ with $ v \neq 0 $ and $ (q_i,e_i) \in (M_i \cap \mathbb{Z}^n) \times \{e_i \}  $, we have: 
		\begin{align*}
		\left\langle (v,\phi(v)),(q_i,e_i) \right\rangle & =\langle v , q_i \rangle + \langle \phi(v) , e_i \rangle\\
		& = \langle v , q_i \rangle + \max \{ \langle v , - M_i \rangle \}\\
		& \geq \langle v , q_i \rangle + \langle v , - q_i \rangle \\
		&=0
	\end{align*}
	We conclude that $ \mathscr{A}_{\mathscr{H}} \subseteq \mathbb{N} \mathscr{H}= \widetilde{\sigma}^\vee \cap M$.\\
	Now, we are going to prove that $ \mathscr{H} \subseteq  \mathbb{N} \mathscr{A}_{\mathscr{H}} $. If $ (v,w) \in \mathscr{H}  $ with $ v \in \mathbb{Z}^n $ and $ w \in \mathbb{Z}^k $, then 
	\[ \left\langle (v,w),(q_i,e_i) \right\rangle =\langle v , q_i \rangle + \langle w , e_i \rangle \geq 0 \]
	for all $ (q_i,e_i) \in  (M_i \cap \mathbb{Z}^n) \times \{e_i \} $. This implies that $ \langle w , e_i \rangle \geq \langle v , - q_i \rangle, \forall q_i \in M_i  $, therefore $ \langle w , e_i \rangle \geq \max\{ \langle v , - M_i \rangle\}= \langle \phi(v) , e_i \rangle $. Therefore, 
 \[ (v,w) = (v,\phi(v)) + \sum_{i=1}^{k} (\langle w , e_i \rangle - \langle \phi(v) , e_i \rangle) (\textbf{0},e_i).  \]
 We conclude that $ (v,w) \in \mathbb{N} \mathscr{A}_{\mathscr{H}}  $, since $ (\textbf{0},e_i), (v,\phi(v)) \in  \mathscr{A}_{\mathscr{H}}  $ and $ \langle w , e_i \rangle - \langle \phi(v) , e_i \rangle \geq 0 $.
\end{proof}

Let us now describe the deformation $Y_{\widetilde{\sigma},\epsilon}$ of $Y_{\sigma}$ associated with the characters of $\mathscr{A}_{\mathscr{H}} $, and a complex fibration on $Y_{\widetilde{\sigma},\epsilon}$. Define	

	 \[ Y_{\widetilde{\sigma}}:=Y_{\mathscr{A}_{\mathscr{H}}}=\text{Spec} (\mathbb{C}([\widetilde{\sigma}^\vee \cap M]).\] 
Let $ t_i = \chi^{(\textbf{0},e_i)}$  and $ \Psi:Y_{\widetilde{\sigma}} \to \mathbb{C}^{k-1} $ given by $ (t_1-t_2,\dots,t_1-t_k) $ similar to \cite[\textsection 3]{gross2001examples} . From \cite{altmann1997versal}, $ Y_{\sigma} = \Psi^{-1}(\textbf{0}) $.  Let $ \epsilon_1=0 $, $ \epsilon:=(\epsilon_2,\dots,\epsilon_{k}) $, where the $ \epsilon_i $ are pairwise distinct and distinct from $ 1 $. Consider $ Y_{\widetilde{\sigma},\epsilon}:=\Psi^{-1}(\epsilon) $ and the complex fibration $ f: Y_{\widetilde{\sigma},\epsilon} \to \mathbb{C} $ given by the projection to $ t_1 $.\\

We note that the singular fibres of $ f $ lie over $ t_1=\epsilon_i $ for $ i=1,\dots,k $. Indeed, we have the following lemma.
\begin{lem} \label{nonsingularfibre}
	If $ t_1\neq \epsilon_i $, we have that $ f^{-1}(t_1) \cong (\mathbb{C}^*)^n $.
\end{lem} 
\begin{proof}
	The main idea is to find $ n $ non-zero characters of $ Y_{\widetilde{\sigma},\epsilon} $ in $ f^{-1}(t_1) $ such that they generate all the other characters. For instance, we are going to prove that for all $ p=1,\dots,k $ the characters $ x_{p,l}:= \chi^{\widetilde{a}_{p,l}} $, $ y_p:= \chi^{\widetilde{b}_p}$,$w^+_{p,j}:=\chi^{(\widetilde{c}^+)_{p,j}} $, and $ w^-_{p,j}:=\chi^{(\widetilde{c}^-)_{p,j}} $ generate all the other characters. Also, these characters satisfy the relations
	
	\begin{equation} \label{eqxy}
		y_p\prod_{l=1}^{m_p}  x_{p,l} = t_p \prod_{\substack{j=1\\
				j\neq p}}^{k} t_j^{\beta_j(M_p)},
	\end{equation}
	for some $ \beta_j(M_p) \in \mathbb{Z} $.
	
		\begin{equation} \label{eqw}
		w^+_{p,j} w^-_{p,j} = \prod_{\substack{l=1\\
				l\neq p}}^{k} t_l^{\eta_{j,l}(M_p)}, 
	\end{equation}
	for all $ j=1,\dots, n-m_p $, and for some $ \eta_{j,l}(M_p) \in \mathbb{Z} $.\\
	
	 Since in $\Psi^{-1}(\epsilon)$, $ t_j\neq 0, \forall j=1,\dots,k $, we get that $ y_p\prod_{l=1}^{m_p}  x_{p,l} \neq 0  $, and  $ w^+_{p,j} w^-_{p,j} \neq 0 $. So, we can choose $ x_{p,1},\dots,x_{p,m_p},w^+_{p,1},\dots,w^+_{p,n-m_p} $ as our $ n $ non-zero characters that generate all the others.\\
	 
	 First, we are going to prove that equations (\ref{eqxy}) and (\ref{eqw}) hold. They will also be useful later, in particular the fact that the exponent of $t_p$ is $1$ in (\ref{eqxy}), and $\eta_{j,p}(M_p)=0$ in (\ref{eqw}).  Recall that we are using the multiplication induced by the semigroup structure of $ \widetilde{\sigma}^\vee \cap \mathbb{Z}^{n+k} $. Therefore
	 \[ b_p+\sum_{j=1}^{m_p} a_{p,j}=0, \]
	 implies
	 \[ 	y_p\prod_{l=1}^{m_p}  x_{p,l} =  \prod_{j=1}^{k} t_j^{\beta_j(M_p)}, \]
	 for some $ \beta_j(M_p)\in \mathbb{Z} $. In order to prove that $ \beta_p(M_p)=1 $, we are going to check that 
	 \[ \langle \phi(b_p),e_p \rangle + \sum_{j=1}^{m_p} \langle \phi(a_{p,l}),e_p \rangle=1. \]
	 
	 By definition $ \langle \phi(a_{p,l}),e_p \rangle=\max\{\langle a_{p,l},-M_p \rangle\} $. Recall the relations in \eqref{relationsmatrix}. Since $ V_p\cdot A_p=Id $, we have that $ \langle a_{p,l},-v_{p,j} \rangle = 0 $ for $ l\neq j $, and $ \langle a_{p,l},-v_{p,l} \rangle = -1  $. We conclude that $ \langle \phi(a_{p,l}),e_p \rangle=0 $. On the other hand, $ \langle \phi(b_p),e_p \rangle= \max\{\langle b_p,-M_p \rangle\}$  and it follows from $ V_p \cdot A_p=Id $ that $ \langle b_p,-v_{p,j} \rangle=1 $ for all $ j $. Hence $ \langle \phi(b_p),e_p \rangle=1 $. Therefore, we get $ \beta_p(M_p)=1 $ and we obtain Equation (\ref{eqxy}). \\
	 
	 As before, the first $n$ coordinates of $\widetilde{c}^+_{p,l}-\widetilde{c}^-_{p,l}$ are equal to $ c_{p,l} - c_{p,l}=\bold{0} $, which implies that
	 \[ 	w^+_{p,j} w^-_{p,j} = \prod_{l=1}^{k} t_l^{\eta_{j,l}(M_p)}  \]
	 for some $ \eta_{j,l}(M_p) \in \mathbb{Z} $. The fact that $ \eta_{j,p}(M_p)=0 $ follows from $ \langle \phi(c_{p,l}),e_p \rangle=\langle \phi(-c_{p,l}),e_p \rangle=0 $ since $ c_{p,l} \in \ker V_p \text{,   } \forall l=1,\dots,n-m_p$ ($V_p \cdot C_p =0$, see \eqref{relationsmatrix}). This concludes the proof of the equations (\ref{eqxy}) and (\ref{eqw}).\\
	 
	 Now we are going to prove that all the characters of $ Y_{\widetilde{\sigma},\epsilon} $ in $ f^{-1}(t_1) $ depend only on 
	 \[ x_{p,1},\dots, x_{p,m_p}, w^+_{p,1},\dots,w^+_{p,n-m_p}. \]

	 By our construction, the elements of $ \mathscr{A}_{\mathscr{H}} $, which are different from $ (\textbf{0},e_i) $, are of the form $ ( \widehat{z},\phi(\widehat{z}) ) $ with $ \widehat{z} \in \mathbb{Z}^n \setminus \textbf{0} $. There exists a unique way to write $ \widehat{z} $ as a sum of the vectors $ a_{p,1},\dots,a_{p,m_p},c_{p,1},\dots, c_{p,n-m_p} $ over $ \mathbb{Z} $, because $(A_p \quad C_p)$ is invertible, see \eqref{relationsmatrix}. Hence, for all $\widehat{z}\in \mathbb{Z}^n \setminus \textbf{0}  $,
	 \begin{equation} \label{widehatz}
\widehat{z} = \sum_{l=1}^{m_p} \xi_l a_{p,l} + \sum_{j=1}^{n-m_p} \xi_{m_p+j} c_{p,j} 
	 \end{equation}
	 with $ \xi_l \in \mathbb{Z} $ for $ l=1,\dots,n $. Letting $ z $ be the character corresponding to $ (\widehat{z},\phi(\widehat{z})) $ we obtain:
	  \begin{equation}\label{znonsingular}
	 	z=\prod_{i=1}^{m_p} (x_{p,i})^{\xi_i}\prod_{l=1}^{n-m_p}(w^+_{p,l})^{\xi_{m_p+l}} \prod_{j=1}^{k} t_j^{\eta_{p,j}(z)}
	 \end{equation}
 for some $ \eta_{p,j}(z) \in \mathbb{Z}. $ Note that all the $ \xi_i $ and all the $ \eta_{p,j}(z) $ are allowed to be negative since by Equation (\ref{eqxy}) and (\ref{eqw}), we have $ x_{p,l} $ and $ w^+_{p,l} $ non-zero. Equation (\ref{znonsingular}) proves the dependency that we want.
\end{proof}

We conclude the proof of the proposition by analyzing the singular fibres.

\begin{lem} \label{singularfiblem}
	$ f^{-1}(\epsilon_p) $ is isomorphic to the subvariety $ x_0x_{p,1}\dots x_{p,m_p}=0 $ inside $ \mathbb{C}^{m_p+1} \times(\mathbb{C}^*)^{n-m_p} $.
\end{lem}
	 
	 \begin{proof}
	 	 In the same way as in the proof of Lemma \ref{nonsingularfibre}, we are going to prove that all the characters of $ Y_{\widetilde{\sigma},\epsilon}, $ in $ f^{-1}(\epsilon_p) $ depend only in
	 	 \[  x_0=y_p,x_{p,1},\dots, x_{p,m_p}, w^+_{p,1},\dots,w^+_{p,n-m_p}. \]

	 	Note that Equations (\ref{eqxy}) and (\ref{eqw}) hold in general, and that in $ f^{-1}(\epsilon_{p}) $ we have $ t_{p}=0 $. In this setting, we obtain the following relations:
	 	\[ y_p\prod_{l=1}^{m_p}  x_{p,l} = 0, \]
	 	\[ w^+_{p,j} w^-_{p,j} = \prod_{\substack{l=1\\
	 			l\neq p}}^{k} t_l^{\eta_{j,l}(M_p)} \neq 0,  \]
	 	for all $ j=1,\dots, n-m_p $ and some $ \eta_{j,l}(M_p) \in \mathbb{Z} $.\\
	 	
	 	Recall that the elements of $ \mathscr{A}_{\mathscr{H}} $, different from $ (\textbf{0},e_i) $, are of the form $(\widehat{z},\phi(\widehat{z}) ) $ with $ \widehat{z} \in \mathbb{Z}^n \setminus \textbf{0} $. In this context, equation (\ref{widehatz}) holds, but we will need to modify it since now $x_{p,l}$ can be zero, so we need non-negative coefficients. For that, we will make use of the character $y_p$.\\
	 	
Recalling that $b_p=-\sum a_{p,l}$, define $ \xi^+= \max \{ 0, -\xi_1,\dots,-\xi_{m_p} \} \geq 0 $ and $ \xi^+_l = \xi_l +   \xi^+ \geq 0 $ for $ l=1,\dots,m_p $. We have: 
	 	\begin{equation} \label{zsingular}
	 		\widehat{z} = \xi^+b_p+\sum_{l=1}^{m_p} \xi^+_l a_{p,l} + \sum_{j=1}^{n-m_p} \xi_{m_p+j} c_{p,j}.
	 	\end{equation}
	 	
Now, since $t_p=0$, we will not be able to correct the $p$-th coordinate of $\phi(\widehat{z})$ as we write the character $z$ associated with $(\widehat{z},\phi(\widehat{z}))$ in terms of $y_p,x_{p,l},w^+_{p,j},t_k$ $k\neq p$. It will suffice to show that:
\begin{equation} \label{xiplus}
	\langle \phi(\widehat{z}),e_p \rangle =\xi ^+
\end{equation}
because from $0\in M_p$, $V_p\cdot A_p=\text{Id}$ and $c_{p,j}\in \ker V_p$ we get that $\langle \phi(a_{p,l}),e_p \rangle = \langle \phi(c_{p,j}),e_p \rangle=0$ and $\langle \phi(b_p),e_p \rangle=1$. So equation (\ref{xiplus}) together with equation (\ref{zsingular}) implies:
\begin{equation} \label{zformula}
	z=y_p^{\xi^+}\prod_{i=1}^{m_p} (x_{p,i})^{\xi_i^+}\prod_{l=1}^{n-m_p}(w^+_{p,l})^{\xi_{m_p+l}} \prod_{\substack{j=1\\
	 			j\neq p}}^{k} t_j^{\eta_{p,j}(z)}
\end{equation}
for some $ \eta_{p,j}(z) \in \mathbb{Z} $, which concludes the proof of Lemma \ref{singularfiblem}.\\
	 	 
	 	 Note that by definition, $ \xi^+, \xi^+_1,\dots, \xi^+_{m_p}  $ are non-negative, and at least one is zero. If $ \xi^+=0 $, then:
	 	 \begin{align*}
	 	 	\left\langle \phi(\widehat{z}),e_p \right\rangle & = \left\langle \phi \left( \sum_{l=1}^{m_p} \xi^+_l a_{p,l} + \sum_{j=1}^{n-m_p} \xi_{m_p+j} c_{p,j} \right), e_p\right\rangle\\
	 	 	& = \max \left\{  \left\langle \sum_{l=1}^{m_p} \xi^+_l a_{p,l} + \sum_{j=1}^{n-m_p} \xi_{m_p+j} c_{p,j}, -M_p\right\rangle \right\} \\
	 	 	& = \max \left\{  \left\langle \sum_{l=1}^{m_p} \xi^+_l a_{p,l}, -M_p\right\rangle \right\} \\
	 	 	& \leq  \sum_{l=1}^{m_p} \xi^+_l \max \left\{  \left\langle  a_{p,l} , -M_p\right\rangle \right\} \\
	 	 	& = 0
	 	 \end{align*}
	 	 since $ \textbf{0} \in M_p $ and $V_p \cdot A_p=\text{Id}$. Therefore, $\langle\phi(\widehat{z}),e_p \rangle =0$.
	 	 
	 	 If $ \xi^+ \neq 0 $, then there exists $ j \in \{ 1,\dots, m_p \}$ such that $ \xi^+_j=0 $, then 
	 	 \begin{align*}
	 	 	\left\langle \phi(\widehat{z}),e_p \right\rangle & = \left\langle \phi \left( \xi^+b_p+\sum_{l=1}^{m_p} \xi^+_l a_{p,l} + \sum_{j=1}^{n-m_p} \xi_{m_p+j} c_{p,j} \right), e_p\right\rangle\\
	 	 	& = \max \left\{  \left\langle \xi^+b_p+\sum_{l=1}^{m_p} \xi^+_l a_{p,l} + \sum_{j=1}^{n-m_p} \xi_{m_p+j} c_{p,j} , -M_p\right\rangle \right\} \\
	 	 	& = \max \left\{  \left\langle \xi^+b_p+\sum_{l=1}^{m_p} \xi^+_l a_{p,l}, -M_p\right\rangle \right\} \\
	 	 	& \leq \xi^+ \max  \left\{  \left\langle  b_p , -M_p\right\rangle \right\}  + \sum_{l=1}^{m_p} \xi^+_l \max \left\{  \left\langle  a_{p,l} , -M_p\right\rangle \right\}   \\
	 	 	& = \xi^+.
	 	 \end{align*}
	 	 
	 	 Since $ v_{p,j} \in M_p $ and $\xi_j^+=0$, we get: 
	 	 \[ \langle \phi(\widehat{z}),e_p\rangle \geq \langle \xi^+b_p + \sum \xi^+_l a_{p,l},-v_{p,j}  \rangle = \xi^+, \]
	 	 recalling again that $V_p \cdot A_p=\text{Id}$. Therefore,  $\langle \phi(\widehat{z}),e_p\rangle =\xi^+$.
	 \end{proof} 
We see now that Lemma \ref{nonsingularfibre} and Lemma \ref{singularfiblem} imply Theorem \ref{fibration}.

	 \end{proof}
	 
\subsection{The singular Lagrangian fibration} \label{singlagfib}

As in Section \ref{examples}, we will construct a Lagrangian fibration associated with the complex fibration $f$. First, we endow $ Y_{\widetilde{\sigma},\epsilon} $ with the symplectic form coming from the embedding as a subset of $ \mathbb{C}^{|\mathscr{A}_{\mathscr{H}}|} $. Next, we are going to study the action of $T^n$ on each fibre of $f$.\\

Recall that in $Y_{\widetilde{\sigma}}=Y_{\mathscr{A} _{\mathscr{H}}}$ we have an algebraic torus given by the image of $\Phi_{\mathscr{A} _{\mathscr{H}}} $ (see Definition \ref{toricfromA}). We are interested in the induced action of the circles given by 
\begin{equation} \label{circles}
\theta_i:=\Phi_{\mathscr{A} _{\mathscr{H}}} ( \{1\} \times \dots \times \{1\}\times S^1 \times \{1\} \dots \times \{1\}),
\end{equation}
where $ S^1\subseteq \mathbb{C}^* $ is in the position $ i $ for $i=1,\dots,n$. Let $\mu_{\theta_i}$ be the moment map of the action of $\theta_i$ and define:
\[ \lambda_i:=2\mu_{\theta_i}= \sum_{q\in \mathscr{A}_{\mathscr{H}}} q_i |\chi^q|^2 \]
where $ q_i $ is the $ i $-th coordinate of $ q $. We define $ \lambda=(\lambda_1,\dots,\lambda_n) $, by our construction, we have that

\begin{equation} \label{lambdazeta}
	\lambda=\sum_{q\in \mathscr{A}_{\mathscr{H}}} \begin{pmatrix}
	q_1\\
	\vdots\\
	q_n
\end{pmatrix} |\chi^q|^2 = \sum_{(\widehat{z},\phi(\widehat{z})) \in \mathscr{A}_{\mathscr{H}} } \widehat{z} \cdot |\chi^{(\widehat{z},\phi(\widehat{z}))}|^2.
\end{equation}

We will consider Lagrangian tori, which are contained in $ f^{-1}(\gamma(r)) $ for the circle $ \gamma (r) \subset \mathbb{C} $ with center in $ 1 $ and radius $ r $, and consist of a single $ (S^1)^n$-orbit inside each fibre of $ \gamma(r) $. Our Lagrangian torus will be given by symplectic parallel transport of the $T^n$ orbits along $\gamma(r)$.

\begin{dfn}
	\label{sing}
	Given the circle $ \gamma (r) \subset \mathbb{C} $ and $ \delta=(\delta_1,\dots,\delta_n) \in \mathbb{R}^n $, we define
	\[ T_{\gamma(r),\delta} := f^{-1}(\gamma(r))\cap \lambda^{-1}(\delta). \]
\end{dfn}

We showed above that the general fibre of $f$ is isomorphic to $(\mathbb{C}^*)^n$, therefore $ T_{\gamma(r),\delta} $ is an embedded Lagrangian torus in $Y_{\widetilde{\sigma}}$, except possibly when $\epsilon_i \in \gamma(r)$, or in the limit when $r=0$, and we have an isotropic $T^n$.\\

We assume that there are not two $\epsilon_i$ in the same circle $\gamma(r)$. Let us understand the singular Lagrangian for when $\epsilon_p \in \gamma(r)$. Recall that we identify $f^{-1}(\epsilon_p)$ as a subset of $\mathbb{C}^{m_p +1} \times (\mathbb{C}^*)^{n- m_p}$ given by the characters

\begin{equation} \label{char}
	y_p,x_{p,1},\dots,x_{p,m_p},w^+_1,\dots, w^+_{n-m_p}
\end{equation}

and the equation

\[ y_p\prod_{l=1}^{m_p}  x_{p,l} = 0.\]

  Note that if $n$ of the characters in (\ref{char}) are non-zero, we have a $T^n$-orbit in the fibre, and hence the Lagrangian passing through it is a smooth $T^{n+1}$. When $k \geq 2$ coordinates are zero, we get a singular Lagrangian $T^{n+1}/ \mathfrak{T}_{k-1}$, where $\mathfrak{T}_{k-1}$ is a $k-1$ cycle inside $T^n \subseteq T^{n+1}$. This gives rise to a codimension $k$ strata of the singular locus $\Sigma$ of the Lagrangian fibration.\\
\subsubsection{Singular fibres and collapsing cycles}
We will analyze the monodromies around the codimension $k=2$ strata of $\Sigma$, i.e., the cases where exactly $n-1$ of the characters in (\ref{char}) are non-zero, identifying the cycle in $T^n$ that collapses as we approach that point. We then get $\frac{m_p(m_p+1)}{2}$ collapsing 1-cycles associated with $M_p$.
  
  \begin{lem} \label{collapsingclasses}
  	 In $f^{-1}(\epsilon_p)$,
  	\begin{enumerate}
  		\item $x_{p,i}=y_p=0$ if and only if $\langle v_{p,i},\lambda \rangle=0$ and $\langle v_{p,l},\lambda \rangle \geq 0, \forall l\neq i $. 
  		
  		\item $x_{p,i}=x_{p,j}=0$ with $i\neq j$ if and only if $\langle v_{p,i}-v_{p,j}, \lambda  \rangle=0$, $\langle v_{p,i}, \lambda  \rangle\leq 0$, and  $\langle v_{p,i} - v_{p,l}, \lambda  \rangle \leq 0,\forall l\neq i,j$.
  	\end{enumerate}
  	As a consequence, the collapsing cycles related with $M_p$ are $v_{p,i}\cdot(\theta_1,\dots,\theta_n) $ for $i=1,\dots, m_p$ and $(v_{p,i}-v_{p,j})\cdot(\theta_1,\dots,\theta_n) $ for $i\neq j$.
  \end{lem}

\begin{proof}
	Recall that the elements of $ \mathscr{A}_{\mathscr{H}} $, different from $ (\textbf{0},e_i) $, are of the form $ ( \widehat{z},\phi(\widehat{z}) ) $ and that we proved equation (\ref{zsingular}). In this setting, we have: 
	\[ \chi^{(\widehat{z},\phi(\widehat{z}))} =y_p^{\xi^+(\widehat{z})}\prod_{l=1}^{m_p} (x_{p,l})^{\xi_l^+(\widehat{z})}\prod_{l=1}^{n-m_p}(w^+_{p,l})^{\xi_{m_p+l}(\widehat{z})} \prod_{\substack{l=1\\
			j\neq p}}^{k} t_l^{\eta_{p,l}(z)},  \]
			and also from \eqref{lambdazeta}
			
			\[\lambda=\sum_{(\widehat{z},\phi(\widehat{z})) \in \mathscr{A}_{\mathscr{H}} } \widehat{z} \cdot |\chi^{(\widehat{z},\phi(\widehat{z}))}|^2.\]
	\begin{enumerate}
		\item $(\Rightarrow)$ If $x_{p,i}=y_p=0$, we get:
		\[ \lambda= \sum_{\substack{(\widehat{z},\phi(\widehat{z})) \in \mathscr{A}_{\mathscr{H}}\\ \xi^+(\widehat{z})=\xi^+_i(\widehat{z})=0 }} \widehat{z} \cdot |\chi^{(\widehat{z},\phi(\widehat{z}))}|^2.\]
		Using equation (\ref{zsingular}), $V_pA_p=Id$, and $c_{p,l} \in \ker V_p$, we have $\langle v_{p,i}, \widehat{z} \rangle=\xi_i(\widehat{z})=\xi_i^+(\widehat{z})-\xi^+(\widehat{z})=0$ and  $\langle v_{p,l}, \widehat{z} \rangle=\xi_l(\widehat{z})=\xi_l^+(\widehat{z})-\xi^+(\widehat{z})=\xi_l^+(\widehat{z}) \geq 0$,  when $\xi^+(\widehat{z})=\xi^+_i(\widehat{z})=0$. Therefore, $\langle v_{p,i},\lambda \rangle=0$ and $\langle v_{p,l},\lambda \rangle \geq 0, \forall l\neq i $. \\
		
		$(\Leftarrow)$ If $\langle v_{p,i},\lambda \rangle=0$ and $\langle v_{p,l},\lambda \rangle \geq 0, \forall l\neq i $. As before, we have $\langle v_{p,i}, \widehat{z} \rangle=\xi_i(\widehat{z})=\xi_i^+(\widehat{z})-\xi^+(\widehat{z})$ and $\langle v_{p,l}, \widehat{z}\rangle=\xi_l(\widehat{z})=\xi_l^+(\widehat{z})-\xi^+(\widehat{z})$, therefore: 
		\begin{equation} \label{eqvpi}
\langle v_{p,i} , \lambda \rangle = \sum_{(\widehat{z},\phi(\widehat{z})) \in \mathscr{A}_{\mathscr{H}}} (\xi_i^+(\widehat{z})-\xi^+(\widehat{z})) |\chi^{(\widehat{z},\phi(\widehat{z}))}|^2 =0,
		\end{equation}
			and
			\begin{equation} \label{eqvpl}
\langle v_{p,l}, \lambda \rangle=\sum_{(\widehat{z},\phi(\widehat{z})) \in \mathscr{A}_{\mathscr{H}}} (\xi_l^+(\widehat{z})-\xi^+(\widehat{z})) |\chi^{(\widehat{z},\phi(\widehat{z}))}|^2 \geq 0.
			\end{equation}
			
		Recall that at least one of the characters $ y_p,x_{p,1},\dots,x_{p,m_p} $ is equal to zero. If $y_p=0$, then from equation \eqref{eqvpi} we get:
		\[ \langle v_{p,i} , \lambda \rangle = |x_{p,i}|^2 + \sum_{\substack{(\widehat{z},\phi(\widehat{z})) \in \mathscr{A}_{\mathscr{H}} \setminus \{\widetilde{a}_{p,i}\} \\ \xi^+(\widehat{z})=0 }} \xi_i^+(\widehat{z}) |\chi^{(\widehat{z},\phi(\widehat{z}))}|^2 =0, \]
		so $ x_{p,i}=0 $ as we want.\\
		 If $ x_{p,i}=0 $, then from equation \eqref{eqvpi} we get:
		 \[ \langle v_{p,i} , \lambda \rangle = -|y_p|^2 - \sum_{\substack{(\widehat{z},\phi(\widehat{z})) \in \mathscr{A}_{\mathscr{H}} \setminus \{\widetilde{b}_p\} \\ \xi_i^+(\widehat{z})=0 }} \xi^+(\widehat{z}) |\chi^{(\widehat{z},\phi(\widehat{z}))}|^2 \geq 0, \]
		 so $ y_p=0 $ as we want.\\
		 If $ x_{p,l}=0, l \neq i $, then from equation \eqref{eqvpl} we get:
		 \[ \langle v_{p,l} , \lambda \rangle = -|y_p|^2 - \sum_{\substack{(\widehat{z},\phi(\widehat{z})) \in \mathscr{A}_{\mathscr{H}} \setminus \{\widetilde{b}_p\} \\ \xi_l^+(\widehat{z})=0 }} \xi^+(\widehat{z}) |\chi^{(\widehat{z},\phi(\widehat{z}))}|^2 \geq 0, \]
		 so $ y_p=0 $. From equation \eqref{eqvpi}, we get $ x_{p,i}=0 $ as above.
		 
		\item $(\Rightarrow)$ If $x_{p,i}=x_{p,j}=0,i\neq j$, we get:
		\[ \lambda= \sum_{\substack{(\widehat{z},\phi(\widehat{z})) \in \mathscr{A}_{\mathscr{H}}\\ \xi^+_i(\widehat{z})=\xi^+_j(\widehat{z})=0 }} \widehat{z} \cdot |\chi^{(\widehat{z},\phi(\widehat{z}))}|^2.\]
		As before, using equation (\ref{zsingular}), $V_pA_p=Id$, and $c_{p,l} \in \ker V_p$, we have $\langle v_{p,i}-v_{p,j}, \widehat{z} \rangle=\xi_i(\widehat{z})-\xi_j(\widehat{z})=\xi_i^+(\widehat{z})-\xi_j^+(\widehat{z})=0$, $ \langle v_{p,i}, \widehat{z} \rangle=\xi_i(\widehat{z})=\xi_i^+(\widehat{z})-\xi^+(\widehat{z})\leq 0 $, and $ \langle v_{p,i}-v_{p,l}, \widehat{z} \rangle=\xi_i(\widehat{z})-\xi_l(\widehat{z})=\xi_i^+(\widehat{z})-\xi_l^+(\widehat{z})\leq 0 $, when $\xi^+_i(\widehat{z})=\xi^+_j(\widehat{z})=0$. Therefore, $\langle v_{p,i}-v_{p,j},\lambda \rangle=0$, $\langle v_{p,i},\lambda \rangle \leq 0$, and $\langle v_{p,i}-v_{p,l},\lambda \rangle\leq0,\forall l\neq i,j$.\\
		
		$(\Leftarrow)$  If $\langle v_{p,i}-v_{p,j}, \lambda  \rangle=0$, $\langle v_{p,i}, \lambda  \rangle\leq 0$, and  $\langle v_{p,i} - v_{p,l}, \lambda  \rangle \leq 0,\forall l\neq i,j$. As above:
		\begin{equation} \label{eqvpij}
	\langle  v_{p,i}-v_{p,j}, \lambda \rangle=\sum_{(\widehat{z},\phi(\widehat{z})) \in \mathscr{A}_{\mathscr{H}}} (\xi_i^+(\widehat{z})-\xi_j^+(\widehat{z})) |\chi^{(\widehat{z},\phi(\widehat{z}))}|^2 = 0,
\end{equation}
		\begin{equation} \label{eqvpieq0}
			\langle v_{p,i} , \lambda \rangle = \sum_{(\widehat{z},\phi(\widehat{z})) \in \mathscr{A}_{\mathscr{H}}} (\xi_i^+(\widehat{z})-\xi^+(\widehat{z})) |\chi^{(\widehat{z},\phi(\widehat{z}))}|^2 \leq 0,
		\end{equation}
	and,
	\begin{equation} \label{eqvpil}
		\langle  v_{p,i}-v_{p,l}, \lambda \rangle=\sum_{(\widehat{z},\phi(\widehat{z})) \in \mathscr{A}_{\mathscr{H}}} (\xi_i^+(\widehat{z})-\xi_l^+(\widehat{z})) |\chi^{(\widehat{z},\phi(\widehat{z}))}|^2 \leq 0.
	\end{equation}
	Recall that at least one of the characters $ y_p,x_{p,1},\dots,x_{p,m_p} $ is equal to zero. If $x_{p,i}=0$, then from equation \eqref{eqvpij} we get:
\[ \langle v_{p,i} - v_{p,j}  , \lambda \rangle = -|x_{p,j}|^2 - \sum_{\substack{(\widehat{z},\phi(\widehat{z})) \in \mathscr{A}_{\mathscr{H}} \setminus \{\widetilde{a}_{p,j}\} \\ \xi_i^+(\widehat{z})=0 }} \xi_j^+(\widehat{z}) |\chi^{(\widehat{z},\phi(\widehat{z}))}|^2 =0, \]
so $ x_{p,j}=0 $ as we want.\\
If $x_{p,j}=0$, then from equation \eqref{eqvpij} we get:
\[ \langle v_{p,i} - v_{p,j}  , \lambda \rangle = |x_{p,i}|^2 + \sum_{\substack{(\widehat{z},\phi(\widehat{z})) \in \mathscr{A}_{\mathscr{H}} \setminus \{\widetilde{a}_{p,i}\} \\ \xi_j^+(\widehat{z})=0 }} \xi_i^+(\widehat{z}) |\chi^{(\widehat{z},\phi(\widehat{z}))}|^2 =0, \]
so $ x_{p,i}=0 $ as we want.\\
If $ x_{p,l}=0, l \neq i,j $, then from equation \eqref{eqvpil} we get:
\[ \langle v_{p,i} - v_{p,l} , \lambda \rangle = |x_{p,i}|^2 + \sum_{\substack{(\widehat{z},\phi(\widehat{z})) \in \mathscr{A}_{\mathscr{H}} \setminus \{\widetilde{b}_p\} \\ \xi_l^+(\widehat{z})=0 }} \xi_i^+(\widehat{z}) |\chi^{(\widehat{z},\phi(\widehat{z}))}|^2 \leq 0, \]
so $ x_{p,i}=0 $. From equation \eqref{eqvpij}, we get $ x_{p,j}=0 $ as above.\\
If $y_p=0$, then from equation \eqref{eqvpieq0} we get:
\[ \langle v_{p,i} , \lambda \rangle = |x_{p,i}|^2 + \sum_{\substack{(\widehat{z},\phi(\widehat{z})) \in \mathscr{A}_{\mathscr{H}} \setminus \{\widetilde{a}_{p,i}\} \\ \xi_i^+(\widehat{z})=0 }} \xi_i^+(\widehat{z}) |\chi^{(\widehat{z},\phi(\widehat{z}))}|^2 \leq 0, \]
so $ x_{p,i}=0 $. From equation \eqref{eqvpij}, we get $ x_{p,j}=0 $ as above.
	\end{enumerate}
\end{proof}
\subsubsection{Convex base diagram}
Now we will study the convex base diagram produced by this singular Lagrangian fibration. As in Section \ref{examples}, we will consider codimension one cuts in the base of the Lagrangian fibration, in the complement of which we will have a toric structure. Let us now assume that $ |\epsilon_{i+1}-1|>|\epsilon_{i}-1| $ for $ i=1,\dots,k-1$, so the circles $ \gamma(r) $ pass first through $ \epsilon_i $ and then through $ \epsilon_{i+1} $ as $ r $ increases. If $ r<1 $, $ T_{\gamma(r),\delta} $ determine a toric structure and we can take action coordinates $ (\lambda_1, \dots, \lambda_{n+1}) $ given by the symplectic flux concerning a reference fibre $ T_{\gamma(r_0),\textbf{0}} $, as in the action-angle coordinates given in \cite{duistermaat1980global}. Considering the limit when $ r_0=0 $, we can interpret $ \lambda_{n+1} $ as the symplectic area of a disk with boundary in a cycle of $ T_{\gamma(r),\delta} $ that collapses as $ r \to 0 $.\\

Our tori, $ T_{\gamma(r),\lambda} $, have fixed values $ \lambda=(\lambda_1,\dots,\lambda_n) $. Therefore, when $r=|\epsilon_p-1|$, we have singular fibres for values $ \lambda$ related with $ M_p $ and given by Lemma \ref{collapsingclasses}. To kill monodromies around singular fibres, we introduce codimension one cuts (i.e., disregard fibres) for which $ r\geq |\epsilon_{p}-1|  $ and $ \lambda$ is related with $ M_p $ and given by Lemma \ref{collapsingclasses}, as above. (Recall Figure \ref{fig:cbdQ5} on our toy model.)\\

We call the image $ B $ of $ \pi: Y_{\widetilde{\sigma},\epsilon} \to \mathbb{R}^{n+1} $ given by $ (\lambda_1,\dots,\lambda_{n+1}) $, and denote by $ \Sigma $ the locus on $ B $ corresponding to singular Lagrangian fibres. With the same notations as in Section \ref{affine}, let $S_p$ be the image of the singular fibres corresponding to the case when $\epsilon_{p}\in\gamma(r)$, and let the cones $C_p$ be defined with $v_p=(\bold{0},1), \bold{0} \in \mathbb{R}^n$, for $p=1,\dots,k$. Let
\begin{equation} \label{Dp0}
	D_{p,0} = \{ (\lambda_1,\dots,\lambda_{n+1}) \in \text{Im}(\pi) | \langle v_{p,l}, \lambda \rangle >0, \forall l=1,\dots,m_p \},
\end{equation}
\begin{equation} \label{Dpj}
	D_{p,j} = \{ (\lambda_1,\dots,\lambda_{n+1}) \in \text{Im}(\pi) | \langle v_{p,j}, \lambda \rangle <0, \langle v_{p,j}-v_{p,l}, \lambda \rangle <0, \forall l\neq j \}, 
\end{equation}
  for $ j=1,\dots,m_p $, and define $\mathfrak{b}_{p,j}$, a circle in the convex base diagram going around the $ j $-th codimension two strata of the set of singular fibres associated with $ M_p $, as in Section \ref{affine}.\\

In order to fix a basis for $ H_1(T_{\gamma(r),\delta}) $, recall the definition of $\theta_i$ in equation (\ref{circles}) and  let $ \theta_{n+1} $ be the collapsing class associated with $ \lambda_{n+1} $. Fix $\{\theta_1,\dots,\theta_{n+1}\}$ as the basis for $ H_1(T_{\gamma(r),\delta}) $. We will study the topological and affine monodromy of the circles $\mathfrak{b}_{p,j}$.\\

  Each $\mathfrak{b}_{p,j}$ will cross the hyperplane $\langle v_{p,j} , \lambda \rangle =0$, so the topological monodromy of $\mathfrak{b}_{p,j}$ is:
 \[ 	M_{\mathfrak{b}_{p,j}}=\begin{pmatrix}
 	\makebox(0,0){\text{\huge Id}} & \begin{matrix}
 			 \vertbar \\
 			 v_{p,j} \\
 			\vertbar \\
 		\end{matrix} \\
 	\begin{matrix}
0 & \dots & 0 
 	\end{matrix} &1
 \end{pmatrix}, \]
because the $ (\theta_1,\dots,\theta_{n}) $ cycles, corresponding to the $ T^n $-action, remain invariant, while the $ \theta_{n+1} $ cycle is twisted in the direction of the collapsing cycle of the monodromy.\\

Hence the affine monodromy, in dual coordinates of $ (\theta_1,\dots,\theta_{n+1}) $, is the transpose inverse of $ M_{\mathfrak{b}_{p,j}} $:
 \[M_{\mathfrak{b}_{p,j}}^{\text{af}}= \begin{pmatrix}
 	\makebox(0,0){\text{\huge Id}} & \begin{matrix}
 		0 \\
 		\vdots \\
 		0 \\
 	\end{matrix} \\
 	\begin{matrix}
 		\horzbar & -v_{p,j} & \horzbar  
 	\end{matrix} &1
 \end{pmatrix} \]
 for $j=1,\dots,m_p$.\\
 
In this setting, we do the transferring the cut operation (Definition \ref{tcut}) concerning $C_p$ for $p=1,\dots,k$. We claim that applying the matrices $ M_{\mathfrak{b}_{p,j}}^{\text{af}} $ as above, we cancel the monodromies related to crossing the hyperplane $\langle v_{p, i}-v_{p,j}, \lambda \rangle=0$. These affine monodromies are of the form:
 \[M_{ v_{p,j}- v_{p,i}}^{\text{af}} =\begin{pmatrix}
	\makebox(0,0){\text{\huge Id}} & \begin{matrix}
		0 \\
		\vdots \\
		0 \\
	\end{matrix} \\
	\begin{matrix}
		\horzbar & v_{p,i}- v_{p,j} & \horzbar  
	\end{matrix} &1
\end{pmatrix} \]
with $ j\neq i  $, and act in $ \overline{D_{p,j}}$. The monodromy $M_{ v_{p,i}- v_{p,j}}^{\text{af}} =M_{\mathfrak{b}_{p,i}} \cdot (M_{\mathfrak{b}_{p,j}})^{-1} $ acts in $ \overline{D_{p,i}} $. Note that in the transferring the cut operation, we are applying $M_{\mathfrak{b}_{p,j}}^{\text{af}} $ in $ \overline{D_{p,j}} $.  We can conclude that the monodromies related to $ v_{p,j}- v_{p, i} $ are canceled.\\

We conclude this section with a lemma that relates the convex base diagram obtained after applying all the transferring the cut operations and the cone $\sigma^\vee $.

\begin{lem} \label{newdiag}
	The convex base diagram, obtained after applying all the transferring the cut operations to the convex base diagram $ B $, has image $\sigma^\vee $.
\end{lem}
 \begin{proof}
 	We will consider a description of a set of generators of $ \sigma^\vee $  given in \cite{altmann1997versal}. To each $ c\in \mathbb{Z}^n $, we associate an integer by $ \eta_0(c):= \max \{
 \langle c, -Q \rangle\} $. By the definition of $ \eta_0 $, we have
\[  \partial\sigma^\vee \cap  \mathbb{Z}^{n+1} = \{(c, \eta_0(c)) | c \in \mathbb{Z}^n\} . \]

Moreover, if $ c_1,\dots, c_w \in \mathbb{Z}^n \setminus \textbf{0} $ are those elements producing irreducible pairs
$ (c, \eta_0(c)) $ (i.e. not allowing any non-trivial lattice decomposition $ (c, \eta_0(c))=(c', \eta_0(c'))$ $+(c'', \eta_0(c''))$), then the elements
\[ (c_1, \eta_0(c_1)), \dots, (c_w, \eta_0(c_w)), (\textbf{0},1) \]
form a generator set for $ \sigma^\vee \cap \mathbb{Z}^{n+1} $ as a semigroup.\\

We also have the following equality:
\begin{align*}
	\eta_0(c)&= \max \{\langle c, -Q \rangle\} \\
	&= \sum_{p=1}^{k} \max \{\langle c, -M_p \rangle\} \\
	&= \sum_{p=1}^{k} \langle \phi(c), e_p \rangle \\
\end{align*}
This equality shows the relation between generators of $ \sigma^\vee $ and our set $ \mathscr{A}_{\mathscr{H}} $ of  generators of $ \widetilde{\sigma}^\vee $.\\

Let $ (c,d) \in \mathbb{R}^n\times\mathbb{R} $. If $ (c,d) \in \overline{D_{p,0}} $, by (\ref{Dp0}), we have that $ \langle v_{p,l}, c \rangle \geq 0, \forall l=1,\dots,m_p $, then $ \max \{\langle c, -M_p \rangle\} =0 $. Note that this coincides with the fact that we do not apply any matrix to $ \overline{D_{p,0}} $.\\

 If $ (c,d) \in \overline{D_{p,j}} $, by (\ref{Dpj}), we have that $ \langle v_{p,j}, c \rangle \leq 0, \langle v_{p,j}-v_{p,l}, c \rangle \leq 0, \forall l\neq j $, then $ \max \{\langle c, -M_p \rangle\} = \langle c, -v_{p,j} \rangle  $. Recall that in the transferring the cut operation related with $C_p$, we apply the matrix $ M_{\mathfrak{b}_{p,j}}^{\text{af}} $ to $ \overline{D_{p,j}} $, therefore we get:
 
 \begin{equation} \label{applymonodromy}
 	\begin{pmatrix}
 	\makebox(0,0){\text{\huge Id}} & \begin{matrix}
 		0 \\
 		\vdots \\
 		0 \\
 	\end{matrix} \\
 	\begin{matrix}
 		\horzbar & -v_{p,j} & \horzbar  
 	\end{matrix} &1
 \end{pmatrix} \begin{pmatrix}
 	\vertbar \\
 	c \\
 	\vertbar \\
 	d
 \end{pmatrix} = \begin{pmatrix}
 \vertbar \\
 c \\
 \vertbar \\
  \max \{\langle c, -M_p \rangle\} + d
\end{pmatrix}. 
 \end{equation}

 In equation \eqref{applymonodromy}, the additive property in the last coordinate of applying the transferring the cut operations is clear. After applying all the monodromies to $ (c,0) $, we get:
\[ \left(c,\sum_{p=1}^{k} \max \{\langle c, -M_p \rangle\} \right)=\left(c,\sum_{p=1}^{k} \langle \phi(c), e_p \rangle \right)= (c,\eta_0(c)). \]
Since the boundary of our initial diagram $ B $ is formed by $ (c,0), $ $ c \in \mathbb{R}^n $, we deduce the relation between $ \sigma^\vee  $ and the convex base diagram. 
 \end{proof}
 
\subsubsection{Monotone fibres} 
 
 Considering the new diagram from Lemma \ref{newdiag},  we are going to prove that the fibres over $\mathbb{R}e_{n+1} \cap \sigma^{\vee} $ outside of the cuts are monotone. Recall that the faces of the cone $\sigma^{\vee} $ are of the form $H_m\cap \sigma^\vee$ for some $m\in \sigma $, where 
 \[ H_m:=\{ u\in \mathbb{R}^{n+1} |\langle u , m \rangle =0  \}. \]
 
Our cone $\sigma$ is generated by elements of the form $(v,1), v\in \mathbb{R}^n$. Therefore, we have the following result.

\begin{lem}
The area of a disk with boundary on the torus $ \mathfrak{L}_s $, associated to $ (\textbf{0},s) \in \sigma^\vee $ with $ \textbf{0} \in \mathbb{R}^n $ and $ s \in \mathbb{R}_{\geq0} $, and Maslov index 2 is $ s $.
\end{lem}

\begin{proof}
	From toric geometry, we know that the disk given by carrying the collapsing cycle $ \mathfrak{v}$ associated to a facet along a path $ c(t) \subset B^{\text{reg}} $ from the reference fibre to the corresponding facet, has Maslov index two and area $ \int_{0}^{T} \langle c'(t),\mathfrak{v} \rangle dt $. The normals to the faces of the cone $\sigma^\vee$ are of the form $(v,1), v\in \mathbb{R}^n$, then we consider
	\[c(t)=L_{v}(t)= \begin{pmatrix}
		\bold{0}\\
		s
	\end{pmatrix} - t \begin{pmatrix}
		v\\
		1
	\end{pmatrix} \]
	
	and the time $t_0$ such that $L_v(t_0)\in H_{(v,1)}$. We get the following equation:
	\[ \left\langle \begin{pmatrix}
		\bold{0}\\
		s
	\end{pmatrix} - t_0 \begin{pmatrix}
		v\\
		1
	\end{pmatrix}, \begin{pmatrix}
		v\\
		1
	\end{pmatrix} \right\rangle=0, \]
	therefore:
	\[t_0=\frac{s}{|v|^2+1}.\]
	
	Then the area of the disk is:
	\[ \int_{0}^{t_0} \langle c'(t),\mathfrak{v} \rangle dt = \int_{0}^{t_0} \langle (v,1),(v,1) \rangle dt=s.  \]
	
	Now note that the boundaries of the above disks generate $ \pi_1(\mathfrak{L}_s) $. In particular,  $ \pi_1(\mathfrak{L}_s)\xrightarrow{0}\pi_1(Y_{\widetilde{\sigma},\epsilon}) $ and hence $ \pi_2(Y_{\widetilde{\sigma},\epsilon},\mathfrak{L}_s) \cong \pi_1(\mathfrak{L}_s) \oplus \pi_2(Y_{\widetilde{\sigma},\epsilon}) $. We know that $ Y_{\widetilde{\sigma},\epsilon} $ is affine, hence Weinstein, so $ c_1=\omega=0 $ in $ \pi_2(Y_{\widetilde{\sigma},\epsilon}) $, therefore all Maslov index 2 disks with boundary on $ \mathfrak{L}_s $ have area $ s $, i.e., $ \omega=\frac{s}{2} \mu $ in $ \pi_2(Y_{\widetilde{\sigma},\epsilon},\mathfrak{L}_s)  $.
\end{proof}

\subsection{Special Lagrangian condition}

Let $(X,\omega, J)$ be a Kähler $n$-dimensional manifold, $D$ a divisor of $X$, and $\Omega$ a non-vanishing holomorphic $n$-form on $X\setminus D$. Recall the definition of special Lagrangian submanifold from \cite{auroux2007mirror}. 

\begin{dfn} [{\cite[Definition~2.1]{auroux2007mirror}}]
	A Lagrangian submanifold $L\subset X \setminus D$ is \textbf{special Lagrangian} with phase $\phi$ if $\text{Im }(e^{-i\phi}\Omega)_{|_L}=0$.
\end{dfn}

In $Y_{\widetilde{\sigma},\epsilon}\setminus \{t_1=1\}$, consider the holomorphic $(n+1,0)$-form given by 

\[\Omega = i^{n+1}\frac{dx_{1,1} \wedge \dots \wedge dx_{1,m_1} \wedge dw_{1,1}^+\wedge \dots dw_{1,n-m_1}^+ \wedge dt_1 }{(t_1-1)\prod x_{1,j}\prod w_{1,l}^+}. \]

Similar to the approach in \cite[Proposition~5.2]{auroux2007mirror}, we can prove that the Lagrangians from Definition \ref{sing} are special.

\begin{prop}
	The tori $T_{\gamma(r),\delta}$ are special Lagrangian with respect to $\Omega$.
\end{prop}

\begin{proof}
	Let $H:Y_{\widetilde{\sigma},\epsilon} \to \mathbb{R}$ given by $|t_1 -1|^2$, and let $X_H$ the vector field such that $\iota_{X_H}\omega=dH$. As in the proof of \cite[Proposition~5.2]{auroux2007mirror}, we have that $X_H$ is tangent to $T_{\gamma(r),\delta}$.\\
	The tangent space to $T_{\gamma(r),\delta}$ is spanned by $X_H$ and by the vector fields generating the $T^n$-action. These vector fields are given by
	
\begin{equation*}
	X_{\theta_k}=i\left(\sum_{j=1}^{m_p} (a_{1,j})_k x_{1,j} \frac{\partial}{\partial x_{1,j}} + 
	\sum_{l=1}^{n-m_p} (c_{1,l})_k w^+_{1,l} \frac{\partial}{\partial w^+_{1,l}} \right)
\end{equation*}

for $k=1,\dots,n$, where $(a_{1,j})_k$ and $(c_{1,l})_k $ denotes the $k$-th coordinates of $a_{1,j}$ and $c_{1,l}$, respectively. Recalling that $\det Z_i=\pm 1$ (see \eqref{relationsmatrix})
\begin{equation*}
	\Omega(X_{\theta_1},\dots,X_{\theta_n},\cdot)= \pm i^{2n+1} \frac{dt_1}{t_1-1}=\pm i^{2n+1}d(\log(t_1-1))
\end{equation*}
Therefore, $\text{Im }\Omega(X_{\theta_1},\dots,X_{\theta_n},X_H)= \pm d(\log |t_1-1|)(X_H)$, which vanishes since  $X_H$ is tangent to the level sets of $H$.  

\end{proof}

\section{Wall-crossing} \label{secwallc}

In this section, we discuss a general formula for the potential by studying the wall-crossing phenomena. We are going to consider the monotone fibres $ T_{\gamma(r),0} $; for $ r >|\epsilon_k-1| $ that arises as a series of "wall-crossing" mutations (\cite{pascaleff2020wall}) from the monotone fibre $ T_{\gamma(r),0}, r<|\epsilon_1-1| $.\\

The tori $T_{\gamma(r),0}$ is an oriented spin (taking the standard spin structure on the torus) monotone Lagrangian of dimension $n+1$ and, given a choice of basis for $H_1(T_{\gamma(r),0})$ and assuming $T_{\gamma(r),0}$ bounds no Maslov zero disks, its potential $W_{T_{\gamma(r),0}} \in \mathbb{C}[z_1^{\pm},\dots,z_{n+1}^{\pm}]$, encoding information of Maslov index two holomorphic disks with boundary on $T_{\gamma(r),0}$, is given by \cite[Definition~2.6]{pascaleff2020wall}. Recall the basis $ \{\theta_1,\dots, \theta_{n+1}\} $ of $H_1(T_{\gamma(r),0},\mathbb{Z})$ defined in Section \ref{singlagfib} and let $ z_i $ be the variable associated to a disk with boundary in $ \theta_i $. With this basis, the potential takes the form:
\begin{equation}
	W_{T_{\gamma(r),0}}= \sum_{\substack{\beta \in H_2( Y_{\widetilde{\sigma},\epsilon},T_{\gamma(r),0})\\ \mu(\beta)=2 }} n_\beta \cdot \textbf{z}^{\partial \beta}
\end{equation}
where $n_{\beta}$ is the count of disks in class $\beta$, we consider $\partial \beta \in H_1(T_{\gamma(r),0}) \cong \mathbb{Z}^{n+1}$, and denote $\textbf{z}^l=z_1^{l_1}\cdots z_{n+1}^{l_{n+1}}$. We are interested in studying how this potential changes when we increase the value of $ r $ in $ T_{\gamma(r),0} $. These changes in the potential occur whenever $T_{\gamma(r),0}$ passes through a wall formed by tori $T_{\gamma(r),\lambda}$ that bound Maslov zero disks. Hence the walls are formed by the tori corresponding to $ r $ equals to $|\epsilon_{p}-1|$ for $ p=1,\dots,k $. By Theorem \ref{fibration}, our local model around the walls corresponds to the one studied in \cite[Section~5.4]{pascaleff2020wall}. Therefore we can use the mutation formula \cite[Lemma~5.17]{pascaleff2020wall} in the form:
\begin{lem}[{\cite[Lemma~5.17]{pascaleff2020wall}}]
	Let $|\epsilon_{p-1}-1|<r<|\epsilon_{p}-1|<r'<|\epsilon_{p+1}-1|$. The potentials $W_{T_{\gamma(r),0}}$ and $W_{T_{\gamma(r'),0}}$ are related by the mutation
	\[ W_{T_{\gamma(r),0}}=\mu (W_{T_{\gamma(r),0}}), \]
	where $\mu$ is given by
	\begin{align*}
		z_i&\mapsto z_i, \quad i=1,\dots,n,\\
		z_{n+1}&\mapsto z_{n+1}(1+z^{v_{p,1}}+\dots+z^{v_{p,m_p}} ),
	\end{align*}
	where we recall that $v_{p,l}\in \mathbb{Z}^n$ and consider $z^{v}=z_1^{v_1}\cdots z_n^{v_n}$.
\end{lem}

  From this, the potential function, $\mathfrak{PO}:=W_{T_{\gamma(r),0}}$ for $r>|\epsilon_k-1|$, associated to our construction based on an admissible Minkowski decomposition $Q=M_1+\dots + M_k$ is:
  \begin{equation} \label{potential}
  	 \mathfrak{PO}(z_1,\dots,z_{n+1})=z_{n+1}\prod_{i=1}^k \left( 1+\sum_{j=1}^{m_i} z^{v_{i,j}} \right),
  \end{equation}

  where $v_{i,j} \in \mathbb{Z}^n$ are the non-zero vertices of $M_i$ and we use the notation $z^{(v_1,\dots,v_n)}=z_1^{v_1}\cdots z_n^{v_n}$.\\
  
  \begin{rmk}
  	The potential \eqref{potential} was found by Lau \cite[Corollary 4.16]{lau2014open} using a different method.
  \end{rmk}
  
  Let $\rho \in (\mathbb{K}^*)^{n+1}$ encode a local system of $T_{\gamma(r),0}$, with respect to the given base of $H_1(T_{\gamma(r),0})$. Here we will take $K^*= \Lambda^*, \mathbb{C}^*,$ or $ \mathrm{U}(1)$ depending on the setting, where $\Lambda$ denotes the Novikov field over $\mathbb{C}$. The self Floer cohomology of $({T_{\gamma(r),0}},\rho)$ is non-zero if and only if $\rho$ is a critical point of $W_{T_{\gamma(r),0}}$ (see \cite[Remark~2.2]{pascaleff2020wall}, \cite[Corollary~2.8]{vianna2016continuum}, \cite[Theorem~2.3]{fukaya2012toric}). Let
  \[ P_i:= \left( 1+\sum_{j=1}^{m_i} z^{v_{i,j}} \right)\]
  for $i=1,\dots,k$.
  
  \begin{lem} \label{critpoint}
  	For $K^*= \Lambda^*, \mathbb{C}^*,$ or $ \mathrm{U}(1)$, the potential function (\ref{potential}) has a critical point in $(K^*)^{n+1}$ if and only if there exist $i,j \in \{1,\dots,k\},i\neq j, $ such that the system of equations $P_i=P_j=0$ has a solution in $(K^*)^{n+1} $.
  \end{lem}
  \begin{proof}
  	First, suppose that there exist $i,j \in \{1,\dots,k\},i\neq j, $ such that the system of equations $P_i=P_j=0$ has a solution in $(K^*)^{n+1} $, then:
  	\begin{align*}
  		\frac{\partial \mathfrak{PO}}{\partial z_{n+1}}= \prod_{l=1}^k P_l=0,
  	\end{align*}
  	and
  	\[ \frac{\partial \mathfrak{PO}}{\partial z_{m}}= z_{n+1} \left( \sum_{l=1}^k \frac{\partial P_l}{\partial z_{m}} \cdot \prod_{q\neq l} P_q \right) =0 \]
  	for $m\neq n+1$. Therefore, the potential function has a critical point in $(K^*)^{n+1} $.\\
  	
  	Now, suppose that the potential function has a critical point in $(K^*)^{n+1} $, then
  	\[ \frac{\partial \mathfrak{PO}}{\partial z_{n+1}}= \prod_{l=1}^k P_l=0\]
  	implies that there exists $i \in \{1,\dots,k\}$, such that $P_i=0$. Therefore, for $m\neq n+1$, we have:
  	
  	\begin{align*}
  		\frac{\partial \mathfrak{PO}}{\partial z_{m}} &= z_{n+1} \left( \frac{\partial P_i}{\partial z_{m}} \cdot \prod_{q\neq i} P_q +  P_i \left( \sum_{l\neq i} \frac{\partial P_l}{\partial z_{m}} \cdot \prod_{\substack{q\neq l \\ q\neq i } } P_q \right) \right) \\
  		&= z_{n+1} \frac{\partial P_i}{\partial z_{m}} \cdot \prod_{q\neq i} P_q =0.
  	\end{align*}
  	
  	Since the critical point is in $(K^*)^{n+1} $, $\frac{\partial P_i}{\partial z_{m}}=0 $ or there exists $j\neq i$  such that $P_j=0$. For the sake of contradiction, suppose that there does not exist $j\neq i$ such that $P_j=0$. In this case $\frac{\partial P_i}{\partial z_{m}}=0,\forall m $, moreover:
  	\[ \begin{pmatrix}
  		0\\
  		\vdots \\
  		0
  	\end{pmatrix} =\begin{pmatrix}
  		z_1 \frac{\partial P_i}{\partial z_{1}}\\
  		\vdots \\
  		z_n \frac{\partial P_i}{\partial z_{n}}
  	\end{pmatrix} = \begin{pmatrix}
			\vertbar & & \vertbar\\
			v_{i,1} & \dots & v_{i,m_i} \\ 
			\vertbar & & \vertbar
		\end{pmatrix}  \begin{pmatrix}
  		z^{v_{i,1}}\\
  		\vdots \\
  		z^{v_{i,m_i}}
  	\end{pmatrix}. \]  
  	The non-zero vertices $v_{i,1},\dots,v_{i,m_i}$ of $M_i$ are linearly independent because we are working with an admissible Mikowski decomposition. Therefore, $z^{v_{i,1}} =\dots = z^{v_{i,m_i}}=0 $, and it is a contradiction as our critical point is in $(K^*)^{n+1} $. We conclude that there exists $j\neq i$  such that $P_j=0$ as we want.
  \end{proof}

\begin{exmp}
	The potential function related to our toy example in Section \ref{Q5} is:
	
	\[ \mathfrak{PO}(z_1,z_2,z_3)= z_3(1+z_1z_2)(1+z_1+z_2). \]
Using Lemma \ref{critpoint}, we conclude that this potential has two one-parameter families of critical points.
\end{exmp}

\subsection{On the convex hull of the potential}

We will analyze the monomials in $ \mathfrak{PO} $.
\begin{lem}
	The Newton polytope of the potential is $Q\times \{1\} \subset \mathbb{R}^{n+1}.$ 
\end{lem}
\begin{proof}
	 Set $ v_{i,0}=\textbf{0} \in \mathbb{Z}^n $ for $ i=1,\dots,k $. Observe that

\begin{align*}
\prod_{i=1}^k \left( 1+\sum_{j=1}^{m_i} z^{v_{i,j}} \right) &= \prod_{i=1}^k \left( \sum_{j=0}^{m_i} z^{v_{i,j}} \right)\\
&=\sum_{i_1,\dots,i_k} z^{v_{1,i_1}+\dots+v_{k,i_k}}.
\end{align*}

Let $ \mathfrak{P}:=\{ v_{1,i_1}+\dots+v_{k,i_k} | i_j\in \{0,\dots,m_j\},\forall j=1,\dots,k \} $. From the definition of a Minkowski sum (Definition \ref{minksum}), we have $ \partial Q \cap \mathbb{Z}^n \subseteq \mathfrak{P} \subseteq Q \cap \mathbb{Z}^n $. Therefore, using \eqref{potential}, we conclude that the Newton polytope of the potential is $Q\times \{1\} \subset \mathbb{R}^{n+1}.$ 
\end{proof}

\section{Other Examples} \label{otherex}

In this section, we will study some examples following the same steps as above. We will leave some details to the reader. 

\subsection{Cone of $ Q_6 $} \label{Q6}
Consider $ Q_6:= \text{Conv}\{ (0,0),(1,0),(0,1),(2,1),(1,2),(2,2) \} \subset \mathbb{R}^2 $ and its two Minkowski decompositions:

\subsubsection{First decomposition of $ Q_6 $ } \label{firstdecomp}

\begin{equation} 
	\vcenter{\hbox{\begin{tikzpicture}
				\draw[fill=gray!25] (0, 0)  -- (1,0) -- (2,1) -- (2,2) --(1,2)-- (0,1) -- cycle;
				\path (1,1) -- (0,0) node[pos=1.3]{$ (0,0) $};
				\path (1,1) -- (1,0) node[pos=1.3]{$ (1,0) $};
				\path (1,1) -- (0,1) node[pos=1.5]{$ (0,1) $};
				\path (1,1) -- (2,1) node[pos=1.5]{$ (2,1) $};
				\path (1,1) -- (1,2) node[pos=1.3]{$ (1,2) $};
				\path (1,1) -- (2,2) node[pos=1.3]{$ (2,2) $};
	\end{tikzpicture}}} = \vcenter{\hbox{\begin{tikzpicture}
				\draw[fill=gray!25] (0, 0)  -- (1,0) -- (1,1) -- cycle;
				\path (1,1) -- (0,0) node[pos=1.3]{$ (0,0) $};
				\path (1,1) -- (1,0) node[pos=1.3]{$ (1,0) $};
				\path (0.5,0.5) -- (1,1) node[pos=1.5]{$ (1,1) $};
	\end{tikzpicture}}} + \vcenter{\hbox{\begin{tikzpicture}
				\draw[fill=gray!25] (0, 0)  -- (0,1) -- (1,1) -- cycle;
				\path (1,1) -- (0,0) node[pos=1.3]{$ (0,0) $};
				\path (1,0) -- (0,1) node[pos=1.3]{$ (0,1) $};
				\path (0.5,0.5) -- (1,1) node[pos=1.5]{$ (1,1) $};
	\end{tikzpicture}}}
\end{equation}

We write $ Q_6 =M_1 +M_2$, where $ M_1 =\text{Conv}\{ (0,0),(1,0),(1,1)\} $ and $ M_2 = \text{Conv}\{ (0,0),(0,1), (1,1) \}  $. We follow the same steps as in Section \ref{thm} to study the convex base diagram and the potential.\\

Recalling our restricted Lagrangian fibration construction, the walls are in $r=|\epsilon_i-1|$. The collapsing classes corresponding to the wall $ r=1 $ are $ (1,0,0) $ and $ (1,1,0) $, and the collapsing classes corresponding to the wall $ r=|1-\epsilon| $ are $ (0,1,0) $ and $ (1,1,0) $ in the base $ (\theta_1,\theta_2,\theta_3) $, as in \eqref{circles}. So, we obtain the convex base diagram presented on the left of Figure \ref{fig:cbdQ61d}.\\

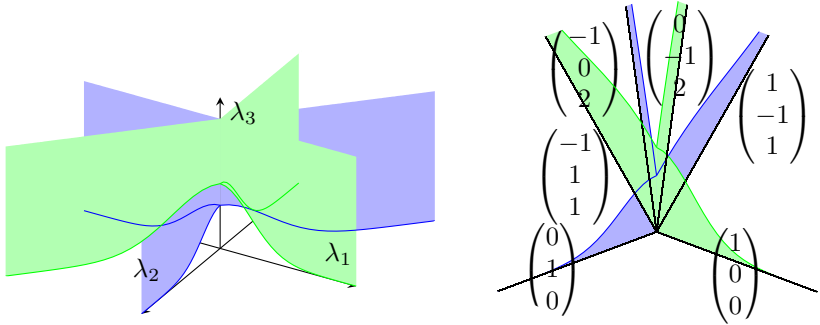
\begin{figure}[h!] 
	
	\centering
	\begin{tikzpicture}
		\begin{axis}
			[scale=0.6,
			axis x line=center,
			axis y line=center,
			axis z line=middle,
			zmin=0.0,
			zmax=3.5,
			ticks=none,
			enlargelimits=false,
			view={-30}{-45},
			xlabel = \(\lambda_1\),
			ylabel = {\(\lambda_2\)},
			zlabel = {\(\lambda_3\)}
			]
			\addplot3[name path=A,color=blue,domain=0:10,samples y=1] (x,-x, {exp(-x^2/10)});
			\addplot3[name path=C,color=blue,domain=0:10,samples y=1] (0,x, {exp(-x^2/10)});
			\addplot3[name path=D,color=blue,domain=0:10,samples y=1] (-x,0, {exp(-x^2/10)});
			\addplot3[name path=B,variable=t,smooth,draw=none,domain=0:10]  (t,-t, 3);
			\addplot3[name path=E,variable=t,smooth,draw=none,domain=0:10]  (0,t, 3);
			\addplot3[name path=F,variable=t,smooth,draw=none,domain=0:10]  (-t,0, 3);
			\addplot3 [blue!30,fill opacity=0.5] fill between [of=A and B];
			\addplot3 [blue!30,fill opacity=0.5] fill between [of=C and E];
			\addplot3 [blue!30,fill opacity=0.5] fill between [of=D and F];

			\addplot3[name path=A,color=green,domain=0:10,samples y=1] (-x,x, {1.5*exp(-x^2/10)});
			\addplot3[name path=C,color=green,domain=0:10,samples y=1] (x,0, {1.5*exp(-x^2/10)});
			\addplot3[name path=D,color=green,domain=0:10,samples y=1] (0,-x, {1.5*exp(-x^2/10)});
			\addplot3[name path=B,variable=t,smooth,draw=none,domain=0:10]  (-t,t, 3);
			\addplot3[name path=E,variable=t,smooth,draw=none,domain=0:10]  (t,0, 3);
			\addplot3[name path=F,variable=t,smooth,draw=none,domain=0:10]  (0,-t, 3);
			
			\addplot3 [green!30,fill opacity=0.5] fill between [of=A and B];
			\addplot3 [green!30,fill opacity=0.5] fill between [of=C and E];
			\addplot3 [green!30,fill opacity=0.5] fill between [of=D and F];
		\end{axis}
		
	\end{tikzpicture}
	\begin{tikzpicture}
		\begin{axis}
			[scale=0.6,
			hide axis,
			zmin=0.0,
			zmax=3.5,
			ticks=none,
			enlargelimits=false,
			view={-45}{-30}
			]
			\addplot3[name path=A,color=blue,domain=0:10,samples y=1] (0,x, {exp(-x^2/10)});
			\addplot3[name path=B,color=blue,domain=0:7,samples y=1] (-x,0, {exp(-x^2/10)+2*x});
			\addplot3[name path=C,color=blue,domain=0:7,samples y=1] (x,-x, {exp(-x^2/10)+x});
			\addplot3[name path=D,variable=t,domain=0:10]  (0,t, 0) node [pos=0.095] {$ \begin{pmatrix}
					0\\1\\0
				\end{pmatrix} $};
			\addplot3[name path=E,variable=t,domain=0:7]  (-t,0, 2*t) node [left, pos=0.004] {$ \begin{pmatrix}
					-1\\0\\2
				\end{pmatrix} $};
			\addplot3[name path=F,variable=t,domain=0:7]  (t,-t, t) node [right, pos=0.006] {$ \begin{pmatrix}
					1\\-1\\1
				\end{pmatrix} $};
			\addplot3 [blue!30,fill opacity=0.5] fill between [of=A and D];
			\addplot3 [blue!30,fill opacity=0.5] fill between [of=B and E];
			\addplot3 [blue!30,fill opacity=0.5] fill between [of=C and F];
			
			\addplot3[name path=A,color=green,domain=0:10,samples y=1] (-x,x, {1.5*exp(-x^2/10)+x});
			\addplot3[name path=B,color=green,domain=0:10,samples y=1] (0,-x, {1.5*exp(-x^2/10)}+2*x);
			\addplot3[name path=C,color=green,domain=0:10,samples y=1] (x,0, {1.5*exp(-x^2/10)});
			\addplot3[name path=D,variable=t,domain=0:10]  (-t,t, t) node [left, pos=0.002] {$ \begin{pmatrix}
					-1\\1\\1
				\end{pmatrix} $};
			\addplot3[name path=E,variable=t,domain=0:10]  (0,-t, 2*t) node [pos=0.003] {$ \begin{pmatrix}
					0\\-1\\2
				\end{pmatrix} $};
			\addplot3[name path=F,variable=t,domain=0:10]  (t,0, 0) node [left, pos=0.015] {$ \begin{pmatrix}
					1\\0\\0
				\end{pmatrix} $};
			\addplot3 [green!30,fill opacity=0.5] fill between [of=A and D];
			\addplot3 [green!30,fill opacity=0.5] fill between [of=B and E];
			\addplot3 [green!30,fill opacity=0.5] fill between [of=C and F];
			
		\end{axis}
	\end{tikzpicture}
	\caption{Convex base diagrams corresponding to the restricted Lagrangian fibration related with the first decomposition of $Q_6$}
	\label{fig:cbdQ61d}
\end{figure}

With the same notations as in Section \ref{affine}, let $S_1$ be the image of the singular fibres corresponding to the case when $0\in\gamma(r)$, and $S_2$ be the image of the singular fibres corresponding to the case when $\epsilon \in \gamma(r)$. The cones $C_i$ for $i=1,2$, are defined with $v_1=v_2=(0,0,1)$. Let
\[ D_{1,0} = \{ (\lambda_1,\lambda_2,\lambda_3) \in \text{Im}(\pi) | \lambda_1>0, \lambda_1+\lambda_2>0  \}, \]   
\[ D_{1,1} = \{ (\lambda_1,\lambda_2,\lambda_3) \in \text{Im}(\pi) | \lambda_1<0, \lambda_2>0\}, \] 
\[ D_{1,2} = \{ (\lambda_1,\lambda_2,\lambda_3) \in \text{Im}(\pi) | \lambda_1+\lambda_2<0, \lambda_2<0\}, \] 
\[ D_{2,0} = \{ (\lambda_1,\lambda_2,\lambda_3) \in \text{Im}(\pi) | \lambda_2>0, \lambda_1+\lambda_2>0 \}, \] 
\[ D_{2,1} = \{ (\lambda_1,\lambda_2,\lambda_3) \in \text{Im}(\pi) | \lambda_1<0, \lambda_1+\lambda_2<0\}, \] 
\[ D_{2,2} = \{ (\lambda_1,\lambda_2,\lambda_3) \in \text{Im}(\pi) | \lambda_1>0, \lambda_2<0\}, \] 
and define $\mathfrak{b}_{i,j}$ as in Section \ref{affine}.\\

The topological monodromies $ M_{\mathfrak{b}_{i,j}}$ around the circles $ \mathfrak{b}_{i,j} $ are:

\[ M_{\mathfrak{b}_{1,1}}=\begin{pmatrix}
	1 & 0 & 1 \\
	0 & 1 & 0 \\
	0 & 0 & 1
\end{pmatrix},  M_{\mathfrak{b}_{1,2}}=\begin{pmatrix}
	1 & 0 & 1 \\
	0 & 1 & 1 \\
	0 & 0 & 1
\end{pmatrix}, M_{\mathfrak{b}_{2,1}} = \begin{pmatrix}
	1 & 0 & 0 \\
	0 & 1 & 1 \\
	0 & 0 & 1
\end{pmatrix},\]
\[ M_{\mathfrak{b}_{2,2}} = \begin{pmatrix}
	1 & 0 & 1 \\
	0 & 1 & 1 \\
	0 & 0 & 1
\end{pmatrix}. \]
respectively.\\

The corresponding affine monodromies are:

\[M_{\mathfrak{b}_{1,1}}^{\text{af}}= \begin{pmatrix}
			1 & 0 & 0 \\
			0 & 1 & 0 \\
			-1 & 0 & 1
		\end{pmatrix}, M_{\mathfrak{b}_{1,2}}^{\text{af}}= \begin{pmatrix}
			1 & 0 & 0 \\
			0 & 1 & 0 \\
			-1 & -1 & 1
		\end{pmatrix}, M_{\mathfrak{b}_{2,1}}^{\text{af}}= \begin{pmatrix}
			1 & 0 & 0 \\
			0 & 1 & 0 \\
			-1 & -1 & 1
		\end{pmatrix},\] \[ M_{\mathfrak{b}_{2,1}}^{\text{af}}= \begin{pmatrix}
			1 & 0 & 0 \\
			0 & 1 & 0 \\
			0 & -1 & 1
		\end{pmatrix}. \]

The result of applying two transferring the cut operations (Definition \ref{tcut}) is obtained by following the steps:

\begin{enumerate}
	\item Apply the matrix $  M_{\mathfrak{b}_{1,1}}^{\text{af}} $ to $ \overline{D_{1,1}} $. 
	\item Apply the matrix $  M_{\mathfrak{b}_{1,2}}^{\text{af}} $ to $ \overline{D_{1,2}} $.
	\item Apply the matrix $  M_{\mathfrak{b}_{2,1}}^{\text{af}} $ to the image of $ \overline{D_{2,1}} $ after applying step 1 and 2. 
	\item Apply the matrix $  M_{\mathfrak{b}_{2,2}}^{\text{af}} $ to the image of $ \overline{D_{2,2}} $ after applying step 1 and 2.
\end{enumerate}
The resulting convex base diagram is presented on the right of Figure \ref{fig:cbdQ61d}.

As we proved in Section \ref{thm}, this cone coincides with:
\[ \sigma^\vee =\text{Cone}\left\{  \begin{pmatrix}
	-1\\0\\2
\end{pmatrix},
\begin{pmatrix}
	-1\\1\\1
\end{pmatrix},
\begin{pmatrix}
	0\\-1\\2
\end{pmatrix},
\begin{pmatrix}
	0\\0\\1
\end{pmatrix},
\begin{pmatrix}
	0\\1\\0
\end{pmatrix},
\begin{pmatrix}
	1\\-1\\1
\end{pmatrix},
\begin{pmatrix}
	1\\0\\0
\end{pmatrix} \right\}. \]
After applying the matrix 
\[ \mathscr{M}=\begin{pmatrix}
	1 & 0 & 0 \\
	0 & 1 & 0 \\
	1 & 1 & 1
	
\end{pmatrix} \in SL(3,\mathbb{Z})\]
to $ \sigma^\vee $, we obtain:
\[ \mathscr{M}\sigma^\vee =\text{Cone}\left\{  \begin{pmatrix}
	-1\\0\\1
\end{pmatrix},
\begin{pmatrix}
	-1\\1\\1
\end{pmatrix},
\begin{pmatrix}
	0\\-1\\1
\end{pmatrix},
\begin{pmatrix}
	0\\0\\1
\end{pmatrix},
\begin{pmatrix}
	0\\1\\1
\end{pmatrix},
\begin{pmatrix}
	1\\-1\\1
\end{pmatrix},
\begin{pmatrix}
	1\\0\\1
\end{pmatrix} \right\}. \]
In this new cone, we have $ Q_6^\vee $ at height 1. 
\begin{figure}[h!]
	\centering
	\begin{tikzpicture}
		\draw[fill=gray!25] (-1,0)  -- (-1,1) --(0,1)-- (1,0) -- (1,-1) -- (0,-1) -- cycle;
		\path (0,0) -- (1,-1) node[pos=1.3]{$ (1,-1) $};
		\path (0,0) -- (-1,1) node[pos=1.3]{$ (-1,1) $};
		\path (0,0) -- (0,1) node[pos=1.3]{$ (0,1) $};
		\path (0,0) -- (1,0) node[pos=1.5]{$ (1,0) $};
		\path (0,0) -- (0,-1) node[pos=1.3]{$ (0,-1) $};
		\path (0,0) -- (-1,0) node[pos=1.7]{$ (-1,0) $};
	\end{tikzpicture}
	\caption{$Q_6^\vee$}
\end{figure}

The corresponding potential function for the monotone Lagrangian $L_a$ is:

\begin{equation} \label{potentialQ6d1}
	\mathfrak{PO}(z_1,z_2,z_3)= z_3(1+z_1+z_1z_2)(1+z_2+z_1z_2).
\end{equation}

Using Lemma \ref{critpoint}, we conclude that this potential has critical points for $K^*= \Lambda^*, \mathbb{C}^*,$ or $ \mathrm{U}(1)$, where $z_1=z_2=e^{\frac{2\pi}{3}ki}$ for $k=1,2$, i.e., are roots of $z^2+z+1$, and $z_3$ is any element of $K^*$. 

\subsubsection{Second decomposition of $ Q_6 $ } \label{seconddecompQ6}

\begin{equation}
	\vcenter{\hbox{\begin{tikzpicture}
				\draw[fill=gray!25] (0, 0)  -- (1,0) -- (2,1) -- (2,2) --(1,2)-- (0,1) -- cycle;
				\path (1,1) -- (0,0) node[pos=1.3]{$ (0,0) $};
				\path (1,1) -- (1,0) node[pos=1.3]{$ (1,0) $};
				\path (1,1) -- (0,1) node[pos=1.5]{$ (0,1) $};
				\path (1,1) -- (2,1) node[pos=1.5]{$ (2,1) $};
				\path (1,1) -- (1,2) node[pos=1.3]{$ (1,2) $};
				\path (1,1) -- (2,2) node[pos=1.3]{$ (2,2) $};
	\end{tikzpicture}}} = \vcenter{\hbox{\begin{tikzpicture}
				\draw[-] (0, 0)  -- (1,0);
				\path (1,1) -- (0,0) node[pos=1.3]{$ (0,0) $};
				\path (0,1) -- (1,0) node[pos=1.3]{$ (1,0) $};
	\end{tikzpicture}}} + \vcenter{\hbox{\begin{tikzpicture}
				\draw[-] (0, 0)  -- (0,1);
				\path (0,1) -- (0,0) node[pos=1.3]{$ (0,0) $};
				\path (0,0) -- (0,1) node[pos=1.3]{$ (0,1) $};
	\end{tikzpicture}}}+ \vcenter{\hbox{\begin{tikzpicture}
				\draw[-] (0, 0)  -- (1,1);
				\path (1,1) -- (0,0) node[pos=1.3]{$ (0,0) $};
				\path (0.5,0.5) -- (1,1) node[pos=1.3]{$ (1,1) $};
	\end{tikzpicture}}}
\end{equation}

We write $ Q_6 =M_1 +M_2+M_3$, where $ M_1 =\text{Conv}\{ (0,0),(1,0)\} $, $ M_2 = \text{Conv}\{ (0,0),(0,1) \}  $, and $ M_3 = \text{Conv}\{ (0,0),(1,1) \}  $. We follow the same steps as in Section \ref{thm} to study the convex base diagram and the potential.\\

Recalling our restricted Lagrangian fibration construction, the walls are in $r=|\epsilon_i-1|$. The collapsing classes corresponding to the walls $ r=1,|1-\epsilon_1|, |1-\epsilon_2| $ are $ (1,0,0) $, $ (0,1,0) $, and $ (1,1,0) $ in the base $ (\theta_1,\theta_2,\theta_3) $, as in \eqref{circles}, respectively. So, we obtain the convex base diagram presented on the left of Figure \ref{fig:cbdQ62d}.\\

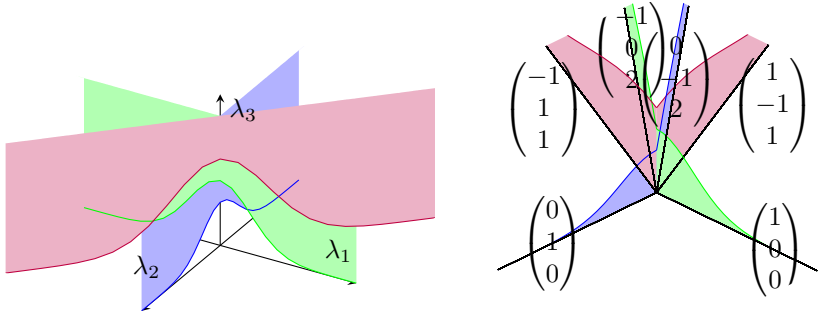
\begin{figure}[h!] 
	
	\centering
	\begin{tikzpicture}
		\begin{axis}
			[scale=0.6,
			axis x line=center,
			axis y line=center,
			axis z line=middle,
			zmin=0.0,
			zmax=3.5,
			ticks=none,
			enlargelimits=false,
			view={-30}{-45},
			xlabel = \(\lambda_1\),
			ylabel = {\(\lambda_2\)},
			zlabel = {\(\lambda_3\)}
			]
			\addplot3[name path=A,color=blue,domain=-10:10,samples y=1] (0,x, {exp(-x^2/10)});
			\addplot3[name path=C,color=green,domain=-10:10,samples y=1] (x,0, {1.5*exp(-x^2/10)});
			\addplot3[name path=D,color=purple,domain=-10:10,samples y=1] (x,-x, {2*exp(-x^2/10)});
			\addplot3[name path=B,variable=t,smooth,draw=none,domain=-10:10]  (0,t, 3);
			\addplot3[name path=E,variable=t,smooth,draw=none,domain=-10:10]  (t,0, 3);
			\addplot3[name path=F,variable=t,smooth,draw=none,domain=-10:10]  (t,-t, 3);
			\addplot3 [blue!30,fill opacity=0.5] fill between [of=A and B];
			\addplot3 [green!30,fill opacity=0.5] fill between [of=C and E];
			\addplot3 [purple!30,fill opacity=0.5] fill between [of=D and F];
			
		\end{axis}
		
	\end{tikzpicture}
	\begin{tikzpicture}
		\begin{axis}
			[scale=0.6,
			hide axis,
			zmin=0.0,
			zmax=3.5,
			ticks=none,
			enlargelimits=false,
			view={-45}{-45}
			]
			\addplot3[name path=A,color=blue,domain=0:10,samples y=1] (0,x, {exp(-x^2/10)});
			\addplot3[name path=B,color=green,domain=0:7,samples y=1] (-x,0, {1.5*exp(-x^2/10)+2*x});
			\addplot3[name path=C,color=purple,domain=0:7,samples y=1] (x,-x, {2*exp(-x^2/10)+x});
			\addplot3[name path=D,variable=t,domain=0:10]  (0,t, 0) node [pos=0.095] {$ \begin{pmatrix}
					0\\1\\0
				\end{pmatrix} $};
			\addplot3[name path=E,variable=t,domain=0:7]  (-t,0, 2*t) node [pos=0.0045] {$ \begin{pmatrix}
					-1\\0\\2
				\end{pmatrix} $};
			\addplot3[name path=F,variable=t,domain=0:7]  (t,-t, t) node [right, pos=0.006] {$ \begin{pmatrix}
					1\\-1\\1
				\end{pmatrix} $};
			\addplot3 [blue!30,fill opacity=0.5] fill between [of=A and D];
			\addplot3 [green!30,fill opacity=0.5] fill between [of=B and E];
			\addplot3 [purple!30,fill opacity=0.5] fill between [of=C and F];
			
			\addplot3[name path=A,color=purple,domain=0:10,samples y=1] (-x,x, {2*exp(-x^2/10)+x});
			\addplot3[name path=B,color=blue,domain=0:10,samples y=1] (0,-x, {exp(-x^2/10)}+2*x);
			\addplot3[name path=C,color=green,domain=0:10,samples y=1] (x,0, {1.5*exp(-x^2/10)});
			\addplot3[name path=D,variable=t,domain=0:10]  (-t,t, t) node [left, pos=0.004] {$ \begin{pmatrix}
					-1\\1\\1
				\end{pmatrix} $};
			\addplot3[name path=E,variable=t,domain=0:10]  (0,-t, 2*t) node [pos=0.0025] {$ \begin{pmatrix}
					0\\-1\\2
				\end{pmatrix} $};
			\addplot3[name path=F,variable=t,domain=0:10]  (t,0, 0) node [pos=0.015] {$ \begin{pmatrix}
					1\\0\\0
				\end{pmatrix} $};
			\addplot3 [purple!30,fill opacity=0.5] fill between [of=A and D];
			\addplot3 [blue!30,fill opacity=0.5] fill between [of=B and E];
			\addplot3 [green!30,fill opacity=0.5] fill between [of=C and F];
			
		\end{axis}
	\end{tikzpicture}
	\caption{Convex base diagrams corresponding to the restricted Lagrangian fibration related with the second decomposition of $Q_6$}
	\label{fig:cbdQ62d}
\end{figure}

As in the first decomposition, we can perform the transferring the cuts operations. The resulting convex base diagram is presented on the right of Figure \ref{fig:cbdQ62d}.\\

We get the same cone as in Section \ref{firstdecomp}. The corresponding potential function is:

\begin{equation} \label{potentialQ62d}
	\mathfrak{PO}(z_1,z_2,z_3)= z_3(1+z_1)(1+z_2)(1+z_1z_2).
\end{equation}

Using Lemma \ref{critpoint}, we conclude that this potential has critical points for $K^*= \Lambda^*, \mathbb{C}^*,$ or $ \mathrm{U}(1)$, where $z_3$ is free to be any element on $K^*$, and $(z_1,z_2)\in \{(-1,-1), (-1,1),(1,-1) \}$.
\begin{rmk}
	These examples were already known to Jonathan Evans and Renato Vianna.
\end{rmk}  

\subsection{Symplectization of unit cosphere bundles of 3-dimensional lens spaces}

As it was pointed out in \cite{abreu2018contact}, the symplectization of $ S^*L^3_p(q) $ is the toric symplectic cone determined by the cone $ C \subset \mathbb{R}^3 $ with normals.
\[ v_1=(q+1,p,1),v_2=(0,0,1),v_3=(1,0,1) \text{, and } v_4=(q,p,1). \]

If we let $ Q:= \text{Conv}\{ (0,0),(1,0),(q,p),(q+1,p) \} \subset \mathbb{R}^2 $, and $ \sigma=C(Q) \subset \mathbb{R}^3 $, then we have that $ C=\sigma^\vee $. For $ Q $, we have the following  Minkowski decomposition:
\[ \vcenter{\hbox{\begin{tikzpicture}
			\draw[fill=gray!25] (0, 0)  -- (1,0) -- (2,2) --(1,2) -- cycle;
			\path (1,1) -- (0,0) node[pos=1.3]{$ (0,0) $};
			\path (1,1) -- (1,0) node[pos=1.3]{$ (1,0) $};
			\path (1,1) -- (1,2) node[pos=1.5]{$ (q,p) $};
			\path (1,1) -- (2,2) node[pos=1.5]{$ (q+1,p) $};
\end{tikzpicture}}} = \vcenter{\hbox{\begin{tikzpicture}
			\draw[-] (0, 0)  -- (1,0);
			\path (1,1) -- (0,0) node[pos=1.3]{$ (0,0) $};
			\path (1,1) -- (1,0) node[pos=1.3]{$ (1,0) $};
\end{tikzpicture}}} + \vcenter{\hbox{\begin{tikzpicture}
			\draw[-] (0, 0)  -- (1,2);
			\path (1,1) -- (0,0) node[pos=1.3]{$ (0,0) $};
			\path (0.5,0.5) -- (1,2) node[pos=1.3]{$ (q,p) $};
\end{tikzpicture}}}\]

\begin{rmk}
	For $p=q=1$, we get the singular quadric studied in \cite{chan2013lagrangian}, which is equivalent to the cone on the unitary square.
\end{rmk}

We write $ Q=M_1 +M_2$, where $ M_1 =\text{Conv}\{ (0,0),(1,0)\} $ and $ M_2 = \text{Conv}\{ (0,0),(q,p) \}  $, so we can use our construction to understand the convex base diagram and the potential function of the monotone fibers in the symplectization of $S^*L^3_p(q)$..\\

Recalling our restricted Lagrangian fibration construction, the walls are in $r=|\epsilon_i-1|$. The collapsing class corresponding to the walls $ r=1, |1-\epsilon| $ are $ (1,0,0) $, $ (q,p,0) $ in the base $ (\theta_1,\theta_2,\theta_3) $, as in \eqref{circles},  respectively. Therefore, we obtain the convex base diagram at the left of Figure \ref{fig:cbdlens}.
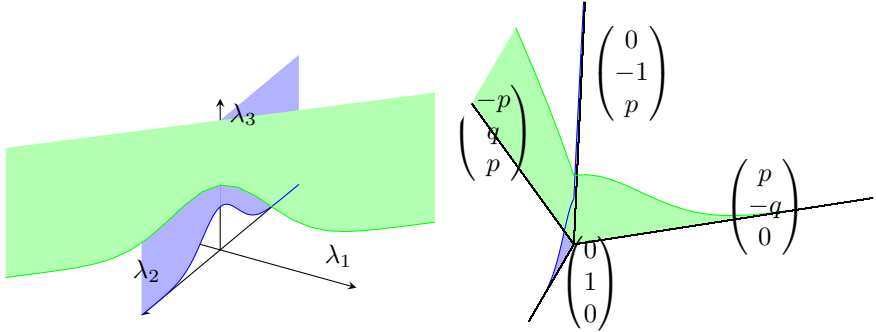
\begin{figure}[h!] 
	
	\centering
	\begin{tikzpicture}
		\begin{axis}
			[scale=0.6,
			axis x line=center,
			axis y line=center,
			axis z line=middle,
			zmin=0.0,
			zmax=3.5,
			ticks=none,
			enlargelimits=false,
			view={-30}{-45},
			xlabel = \(\lambda_1\),
			ylabel = {\(\lambda_2\)},
			zlabel = {\(\lambda_3\)}
			]
			\addplot3[name path=A,color=blue,domain=0:10,samples y=1] (0,x, {exp(-x^2/10)});
			\addplot3[name path=D,color=blue,domain=0:10,samples y=1] (0,-x, {exp(-x^2/10)});
			\addplot3[name path=B,variable=t,smooth,draw=none,domain=0:10]  (0,t, 3);
			\addplot3[name path=F,variable=t,smooth,draw=none,domain=0:10]  (0,-t, 3);
			\addplot3 [blue!30,fill opacity=0.5] fill between [of=A and B];
			\addplot3 [blue!30,fill opacity=0.5] fill between [of=D and F];
			
			\addplot3[name path=G,color=green,domain=-10:10,samples y=1] (-2*x,x, {1.5*exp(-x^2/10)});
			\addplot3[name path=H,variable=t,smooth,draw=none,domain=-10:10]  (2*t,-t, 3);
			\addplot3 [green!30,fill opacity=0.5] fill between [of=G and H];
			
		\end{axis}
	\end{tikzpicture}
	\begin{tikzpicture}
		\begin{axis}
			[scale=0.6,
			hide axis,
			zmin=0.0,
			zmax=3.5,
			ticks=none,
			enlargelimits=false,
			view={-15}{-45}
			]
			\addplot3[name path=C,color=blue,domain=0:10,samples y=1] (0,x, {exp(-x^2/10)});
			\addplot3[name path=D,color=blue,domain=0:7,samples y=1] (0,-x, {exp(-x^2/10)+2*x});
			\addplot3[name path=E,variable=t,domain=0:10]  (0,t, 0) node [right, pos=0.990] {$ \begin{pmatrix}
					0\\1\\0
				\end{pmatrix} $};
			\addplot3[name path=F,variable=t,domain=0:7]  (0,-t, 2*t) node [right, pos=0.005] {$ \begin{pmatrix}
					0\\-1\\p
				\end{pmatrix} $};
			\addplot3 [blue!30,fill opacity=0.5] fill between [of=C and E];
			\addplot3 [blue!30,fill opacity=0.5] fill between [of=D and F];
			
			\addplot3[name path=G,color=green,domain=0:10,samples y=1] (-2*x,x, {1.5*exp(-x^2/10)+2*x});
			\addplot3[name path=J,color=green,domain=0:10,samples y=1] (2*x,-x, {1.5*exp(-x^2/10)});
			\addplot3[name path=H,variable=t,domain=0:10]  (2*t,-t, 0) node [left, pos=0.996] {$ \begin{pmatrix}
					p\\-q\\0
				\end{pmatrix} $};
			\addplot3[name path=I,variable=t,domain=0:10]  (-2*t,t, t) node [pos=0.0055] {$ \begin{pmatrix}
					-p\\q\\p
				\end{pmatrix} $};
			\addplot3 [green!30,fill opacity=0.5] fill between [of=G and I];
			\addplot3 [green!30,fill opacity=0.5] fill between [of=J and H];
			
		\end{axis}
	\end{tikzpicture}
	
	\caption{Convex base diagrams corresponding to the restricted Lagrangian fibration related with $ S^*L^3_p(q)) $}
	\label{fig:cbdlens}
\end{figure}

With the same notations as in Section \ref{affine}, let $S_1$ be the image of the singular fibres corresponding to the case when $0\in\gamma(r)$, and $S_2$ be the image of the singular fibres corresponding to the case when $\epsilon \in \gamma(r)$. The cones $C_i$ for $i=1,2$, are defined with $v_1=v_2=(0,0,1)$. Let
\[ D_{1,0} = \{ (\lambda_1,\lambda_2,\lambda_3) \in \text{Im}(\pi) | \lambda_1>0 \}, \]   
\[ D_{1,1} = \{ (\lambda_1,\lambda_2,\lambda_3) \in \text{Im}(\pi) | \lambda_1<0\}, \] 
\[ D_{2,0} = \{ (\lambda_1,\lambda_2,\lambda_3) \in \text{Im}(\pi) | q\lambda_1+p\lambda_2>0 \}, \] 
\[ D_{2,1} = \{ (\lambda_1,\lambda_2,\lambda_3) \in \text{Im}(\pi) | q\lambda_1+p\lambda_2<0\}, \] 
and define $\mathfrak{b}_{i,j}$ as in Section \ref{affine}.\\

The topological monodromies $ M_{\mathfrak{b}_{i,j}}$ around the circles $ \mathfrak{b}_{i,j} $ are:

\[ M_{\mathfrak{b}_{1,1}}=\begin{pmatrix}
	1 & 0 & 1 \\
	0 & 1 & 0 \\
	0 & 0 & 1
\end{pmatrix},   M_{\mathfrak{b}_{2,1}} = \begin{pmatrix}
	1 & 0 & q \\
	0 & 1 & p \\
	0 & 0 & 1
\end{pmatrix}. \]
respectively.\\

The corresponding affine monodromies are:
\[ M_{\mathfrak{b}_{1,1}}^{\text{af}}= \begin{pmatrix}
			1 & 0 & 0 \\
			0 & 1 & 0 \\
			-1 & 0 & 1
		\end{pmatrix}, M_{\mathfrak{b}_{2,1}}^{\text{af}}= \begin{pmatrix}
			1 & 0 & 0 \\
			0 & 1 & 0 \\
			-q & -p & 1
		\end{pmatrix}.\]

The result of applying two transferring the cut operations (Definition \ref{tcut}) is obtained by following the steps:

\begin{enumerate}
	\item Apply the matrix $  M_{\mathfrak{b}_{1,1}}^{\text{af}} $ to $ \overline{D_{1,1}} $. 
	\item Apply the matrix $  M_{\mathfrak{b}_{2,1}}^{\text{af}} $ to the image of $ \overline{D_{2,1}} $ after step 1.
\end{enumerate}
The resulting almost toric base diagram is presented at the right of Figure \ref{fig:cbdlens}.

The corresponding potential function is:

\[ \mathfrak{PO}(z_1,z_2,z_3)= z_3(1+z_1)(1+z_1^qz_2^p). \]

Using Lemma \ref{critpoint}, we conclude that this potential has critical points for $K^*= \Lambda^*, \mathbb{C}^*,$ or $ \mathrm{U}(1)$, where $z_1=-1$, $z_2$ equals to one of the $p$ roots of $(-1)^{q+1}$, and $z_3$ is free in $K^*$.

\begin{rmk}
	We note that a Lagrangian Lens space $L^3_p(q)$ appears as the union of the $T^2$-orbits over a segment connecting the singularities at $0$ and $\epsilon$, for which $\lambda_1=\lambda_2=0$. The smoothing $Y_{\widetilde{\sigma},\epsilon} \cong T^*L^3_p(q)$.
\end{rmk}

\subsection{Cone over the cubic}\label{conecubic}

Consider $Q_3:= \text{Conv} \{(0,0),(3,0),(0,3)\} \subseteq \mathbb{Z}^2$ and its Minkowski decomposition:

\[ \vcenter{\hbox{\begin{tikzpicture}
			\draw[fill=gray!25] (0, 0)  -- (3,0) -- (0,3) -- cycle;
			\path (1,1) -- (0,0) node[pos=1.2]{$ (0,0) $};
			\path (1,1) -- (3,0) node[pos=1.2]{$ (3,0) $};
			\path (1,1) -- (0,3) node[pos=1.2]{$ (0,3) $};
\end{tikzpicture}}} = \vcenter{\hbox{\begin{tikzpicture}
			\draw[fill=gray!25] (0, 0)  -- (1,0) -- (0,1) -- cycle;
			\path (0,1) -- (0,0) node[pos=1.3]{$ (0,0) $};
			\path (1,1) -- (1,0) node[pos=1.3]{$ (1,0) $};
			\path (0,0) -- (0,1) node[pos=1.2]{$ (0,1) $};
\end{tikzpicture}}} + \vcenter{\hbox{\begin{tikzpicture}
			\draw[fill=gray!25] (0, 0)  -- (1,0) -- (0,1) -- cycle;
			\path (0,1) -- (0,0) node[pos=1.3]{$ (0,0) $};
			\path (1,1) -- (1,0) node[pos=1.3]{$ (1,0) $};
			\path (0,0) -- (0,1) node[pos=1.2]{$ (0,1) $};
\end{tikzpicture}}} + \vcenter{\hbox{\begin{tikzpicture}
			\draw[fill=gray!25] (0, 0)  -- (1,0) -- (0,1) -- cycle;
			\path (0,1) -- (0,0) node[pos=1.3]{$ (0,0) $};
			\path (1,1) -- (1,0) node[pos=1.3]{$ (1,0) $};
			\path (0,0) -- (0,1) node[pos=1.2]{$ (0,1) $};
\end{tikzpicture}}}\]
We write $Q_3=M_1+M_2+M_3$, where $M_1=M_2=M_3=\text{Conv}\{(0,0),(1,0),(0,1)\}$. This Minkowski decomposition is admissible. Let $ Y_{\sigma} $ the affine variety related to the cone $ \sigma:=C(Q_3) $ and $ Y_{\widetilde{\sigma},\epsilon} $ be the smoothing of $ Y_{\sigma} $. The ideal $ I(Y_{\widetilde{\sigma}},\epsilon)\subseteq \mathbb{C}[x_{1},x_{2},y,t] $ is generated by:

\[ x_{1}x_{2}y=t(t-\epsilon_1)(t-\epsilon_2). \]

   We obtain a complex fibration, given by the projection to the $ t $ character, with three singular fibres modeled as $ x_{1}x_{2}y=0 \subset \mathbb{C}^3_{(x_{1},x_{2},y)}  $ over $ t=0,\epsilon_1,\epsilon_2 $. We build our Lagrangian fibration with a convex base diagram as in the left picture of Figure \ref{fig:cbdQ3}. After transferring the cut operations, we get the right picture of Figure \ref{fig:cbdQ3}.

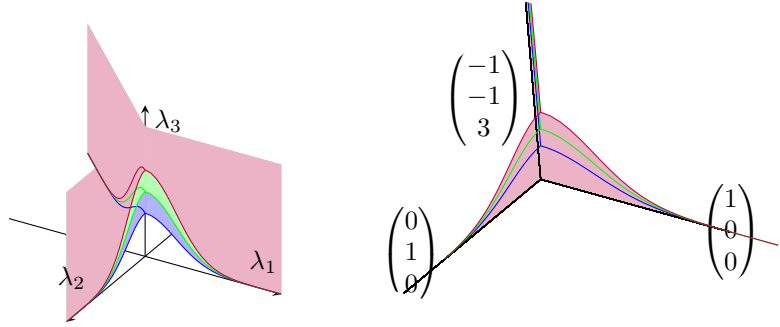
\begin{figure}[h!] 
	
	\centering
	\begin{tikzpicture}
		\begin{axis}
			[scale=0.6,
			axis x line=center,
			axis y line=center,
			axis z line=middle,
			zmin=0.0,
			zmax=3.5,
			ticks=none,
			enlargelimits=false,
			view={-30}{-45},
			xlabel = \(\lambda_1\),
			ylabel = {\(\lambda_2\)},
			zlabel = {\(\lambda_3\)}
			]
			\addplot3[name path=A,color=blue,domain=0:10,samples y=1] (-x,-x, {exp(-x^2/10)});
			\addplot3[name path=C,color=blue,domain=0:10,samples y=1] (0,x, {exp(-x^2/10)});
			\addplot3[name path=D,color=blue,domain=0:10,samples y=1] (x,0, {exp(-x^2/10)});
			\addplot3[name path=B,variable=t,smooth,draw=none,domain=0:10]  (-t,-t, 3);
			\addplot3[name path=E,variable=t,smooth,draw=none,domain=0:10]  (0,t, 3);
			\addplot3[name path=F,variable=t,smooth,draw=none,domain=0:10]  (t,0, 3);
			\addplot3 [blue!30,fill opacity=0.5] fill between [of=A and B];
			\addplot3 [blue!30,fill opacity=0.5] fill between [of=C and E];
			\addplot3 [blue!30,fill opacity=0.5] fill between [of=D and F];

			\addplot3[name path=A,color=green,domain=0:10,samples y=1] (-x,-x, {1.5*exp(-x^2/10)});
			\addplot3[name path=C,color=green,domain=0:10,samples y=1] (0,x, {1.5*exp(-x^2/10)});
			\addplot3[name path=D,color=green,domain=0:10,samples y=1] (x,0, {1.5*exp(-x^2/10)});
			\addplot3[name path=B,variable=t,smooth,draw=none,domain=0:10]  (-t,-t, 3);
			\addplot3[name path=E,variable=t,smooth,draw=none,domain=0:10]  (0,t, 3);
			\addplot3[name path=F,variable=t,smooth,draw=none,domain=0:10]  (t,0, 3);
			
			\addplot3 [green!30,fill opacity=0.5] fill between [of=A and B];
			\addplot3 [green!30,fill opacity=0.5] fill between [of=C and E];
			\addplot3 [green!30,fill opacity=0.5] fill between [of=D and F];
			
			\addplot3[name path=A,color=purple,domain=0:10,samples y=1] (-x,-x, {2*exp(-x^2/10)});
			\addplot3[name path=C,color=purple,domain=0:10,samples y=1] (0,x, {2*exp(-x^2/10)});
			\addplot3[name path=D,color=purple,domain=0:10,samples y=1] (x,0, {2*exp(-x^2/10)});
			\addplot3[name path=B,variable=t,smooth,draw=none,domain=0:10]  (-t,-t, 3);
			\addplot3[name path=E,variable=t,smooth,draw=none,domain=0:10]  (0,t, 3);
			\addplot3[name path=F,variable=t,smooth,draw=none,domain=0:10]  (t,0, 3);
			
			\addplot3 [purple!30,fill opacity=0.5] fill between [of=A and B];
			\addplot3 [purple!30,fill opacity=0.5] fill between [of=C and E];
			\addplot3 [purple!30,fill opacity=0.5] fill between [of=D and F];

		\end{axis}
		
	\end{tikzpicture}
	\begin{tikzpicture}
		\begin{axis}
			[scale=0.6,
			hide axis,
			zmin=0.0,
			zmax=4.5,
			ticks=none,
			enlargelimits=false,
			view={-30}{-45}
			]
			\addplot3[name path=A,color=blue,domain=0:10,samples y=1] (0,x, {exp(-x^2/10)});
			\addplot3[name path=B,color=blue,domain=0:7,samples y=1] (x,0, {exp(-x^2/10)});
			\addplot3[name path=C,color=blue,domain=0:7,samples y=1] (-x,-x, {exp(-x^2/10)+3*x});
			\addplot3[name path=D,variable=t,domain=0:10]  (0,t, 0) node [left,pos=0.095] {$ \begin{pmatrix}
					0\\1\\0
				\end{pmatrix} $};
			\addplot3[name path=E,variable=t,domain=0:8]  (t,0, 0) node [pos=1] {$ \begin{pmatrix}
					1\\0\\0
				\end{pmatrix} $};
			\addplot3[name path=F,variable=t,domain=0:7]  (-t,-t, 3*t) node [left, pos=0.002] {$ \begin{pmatrix}
					-1\\-1\\3
				\end{pmatrix} $};
			\addplot3 [blue!30,fill opacity=0.5] fill between [of=A and D];
			\addplot3 [blue!30,fill opacity=0.5] fill between [of=B and E];
			\addplot3 [blue!30,fill opacity=0.5] fill between [of=C and F];
			
			\addplot3[name path=G,color=green,domain=0:10,samples y=1] (x,0, {1.5*exp(-x^2/10)});
			\addplot3[name path=H,color=green,domain=0:7,samples y=1] (0,x, {1.5*exp(-x^2/10)});
			\addplot3[name path=I,color=green,domain=0:7,samples y=1] (-x,-x, {1.5*exp(-x^2/10)+3*x});
			
			\addplot3[name path=D,variable=t,smooth,draw=none,domain=0:10]  (t,0, 0);
			\addplot3[name path=E,variable=t,smooth,draw=none,domain=0:10]  (0,t, 0);
			\addplot3[name path=F,variable=t,smooth,draw=none,domain=0:10]  (-t,-t, 3*t);
			
			\addplot3 [green!30,fill opacity=0.5] fill between [of=G and D];
			\addplot3 [green!30,fill opacity=0.5] fill between [of=H and E];
			\addplot3 [green!30,fill opacity=0.5] fill between [of=I and F];
			
			\addplot3[name path=J,color=purple,domain=0:10,samples y=1] (x,0, {2*exp(-x^2/10)});
			\addplot3[name path=K,color=purple,domain=0:7,samples y=1] (0,x, {2*exp(-x^2/10)});
			\addplot3[name path=L,color=purple,domain=0:7,samples y=1] (-x,-x, {2*exp(-x^2/10)+3*x});
			
			\addplot3 [purple!30,fill opacity=0.5] fill between [of=J and D];
			\addplot3 [purple!30,fill opacity=0.5] fill between [of=K and E];
			\addplot3 [purple!30,fill opacity=0.5] fill between [of=L and F];
			
		\end{axis}
	\end{tikzpicture}
	
	\caption{Convex base diagrams corresponding to the restricted Lagrangian fibration related to Minkowski decomposition of $Q_3$.}
	\label{fig:cbdQ3}
\end{figure}

As we proved in Section \ref{thm}, this cone coincides with:
\[ \sigma^\vee =\text{Cone}\left\{  \begin{pmatrix}
	-1\\-1\\3
\end{pmatrix},
\begin{pmatrix}
	0\\0\\1
\end{pmatrix},
\begin{pmatrix}
	0\\1\\0
\end{pmatrix},
\begin{pmatrix}
	1\\0\\0
\end{pmatrix} \right\}. \]
After applying the matrix 
\[ \mathscr{M}=\begin{pmatrix}
	1 & 0 & 0 \\
	0 & 1 & 0 \\
	1 & 1 & 1
	
\end{pmatrix} \in SL(3,\mathbb{Z})\]
to $ \sigma^\vee $, we obtain:
\[ \mathscr{M}\sigma^\vee =\text{Cone}\left\{  \begin{pmatrix}
	-1\\-1\\1
\end{pmatrix},
\begin{pmatrix}
	0\\0\\1
\end{pmatrix},
\begin{pmatrix}
	0\\1\\1
\end{pmatrix},
\begin{pmatrix}
	1\\0\\1
\end{pmatrix} \right\}. \]
In this new cone, we have $ Q_3^\vee $ at height 1. 

\begin{figure}[h!]
	\centering
	\begin{tikzpicture}
		\draw[fill=gray!25] (-1,-1) --(0,1)-- (1,0) -- cycle;
		\path (0,0) -- (-1,-1) node[pos=1.3]{$ (-1,-1) $};
		\path (0,0) -- (0,1) node[pos=1.3]{$ (0,1) $};
		\path (0,0) -- (1,0) node[pos=1.5]{$ (1,0) $};
	\end{tikzpicture}
	\caption{$Q_3^\vee$}
\end{figure}

By \eqref{potential}, the potential function is given by: 

\begin{equation}\label{cubic}
	\mathfrak{PO}=z_{3}(1+z_1+z_2)^3,
\end{equation}
and using Lemma \ref{critpoint}, we get that the critical locus has $\dim_{\mathbb{C}}=2$ when $K^*= \Lambda^*$ or $\mathbb{C}^*$, and a 2-$\dim_{\mathbb{R}}$ family of critical points, when $K^*=\mathrm{U}(1)$.

\subsection{Example where $\mathfrak{PO}$ doesn't have critical points}

\[ \vcenter{\hbox{\begin{tikzpicture}
			\draw[fill=gray!25] (0, 0)  -- (0,1) -- (1,1) -- (2,0)-- cycle;
			\path (1,1) -- (0,0) node[pos=1.2]{$ (0,0) $};
			\path (1,0) -- (0,1) node[pos=1.2]{$ (0,1) $};
			\path (0,0) -- (1,1) node[pos=1.2]{$ (1,1) $};
			\path (1,1) -- (2,0) node[pos=1.2]{$ (2,0) $};
\end{tikzpicture}}} = \vcenter{\hbox{\begin{tikzpicture}
			\draw[fill=gray!25] (0, 0)  -- (1,0) -- (0,1) -- cycle;
			\path (1,1) -- (0,0) node[pos=1.2]{$ (0,0) $};
			\path (0,1) -- (1,0) node[pos=1.2]{$ (1,0) $};
			\path (1,0) -- (0,1) node[pos=1.2]{$ (0,1) $};
\end{tikzpicture}}} + \vcenter{\hbox{\begin{tikzpicture}
			\draw[-] (0, 0)  -- (1,0);
			\path (1,1) -- (0,0) node[pos=1.3]{$ (0,0) $};
			\path (1,1) -- (1,0) node[pos=1.3]{$ (1,0) $};
\end{tikzpicture}}} \]

Consider $M_1= \text{Conv} \{(0,0),(0,1),(1,0)\}$, $M_2=\text{Conv} \{ (0,0),(1,0) \}$, and $Q=M_1+M_2$. This Minkowski decomposition is admissible. Let $ Y_{\sigma} $ the affine variety related to the cone $ \sigma:=C(Q) $ and $ Y_{\widetilde{\sigma},\epsilon} $ be the smoothing of $ Y_{\sigma} $. The ideal $ I(Y_{\widetilde{\sigma}},\epsilon)\subseteq \mathbb{C}[x_{1,1},x_{1,2},y_1,y_2,z,t] $ is generated by:

\begin{equation*} \label{eqexmpnot}
	\begin{aligned}
		&x_{1,1}y_1- tz  &\qquad &x_{1,1}y_2-t(t-\epsilon)  &\qquad  &x_{1,2}z  - (t-\epsilon)\\
		&x_{1,2}y_1-y_2  &\qquad  &y_1(t-\epsilon)-zy_2 &\qquad  &
	\end{aligned}
\end{equation*}

   We obtain a complex fibration, given by the projection to the $ t $ character, with two singular fibres modeled as $ x_{1,1}x_{1,2}y_1=0 \subset \mathbb{C}^3_{(x_{1,1},x_{1,2},y_1)}  $ over $ t=0 $, and $ x_{1,1}y_2=0 \subset \mathbb{C}^2_{(x_{1,1},y_2)} \times \mathbb{C}^*_{(x_{1,2})}  $ over $ t=\epsilon $. Out of that, we build our Lagrangian fibration with a convex base diagram as in the left picture of Figure \ref{fig:nondisplaceable} (analogous to the one depicted in Figure \ref{fig:cbdQ5} for our toy example). After transferring the cut operations, we get the right picture of Figure \ref{fig:nondisplaceable}.

\begin{figure}[h!]
	
	\centering
	\begin{tikzpicture}
		\begin{axis}
			[scale=0.6,
			axis x line=center,
			axis y line=center,
			axis z line=middle,
			zmin=0.0,
			zmax=3.5,
			ticks=none,
			enlargelimits=false,
			view={-30}{-45},
			xlabel = \(\lambda_1\),
			ylabel = {\(\lambda_2\)},
			zlabel = {\(\lambda_3\)}
			]
			\addplot3[name path=A,color=blue,domain=0:10,samples y=1] (x,0, {exp(-x^2/10)});
			\addplot3[name path=C,color=blue,domain=0:10,samples y=1] (0,x, {exp(-x^2/10)});
			\addplot3[name path=D,color=blue,domain=0:10,samples y=1] (-x,-x, {exp(-x^2/10)});
			\addplot3[name path=B,variable=t,smooth,draw=none,domain=0:10]  (t,0, 3);
			\addplot3[name path=E,variable=t,smooth,draw=none,domain=0:10]  (0,t, 3);
			\addplot3[name path=F,variable=t,smooth,draw=none,domain=0:10]  (-t,-t, 3);
			\addplot3 [blue!30,fill opacity=0.5] fill between [of=A and B];
			\addplot3 [blue!30,fill opacity=0.5] fill between [of=C and E];
			\addplot3 [blue!30,fill opacity=0.5] fill between [of=D and F];
			
			\addplot3[name path=G,color=green,domain=-10:10,samples y=1] (0,x, {1.5*exp(-x^2/10)});
			\addplot3[name path=H,variable=t,smooth,draw=none,domain=-10:10]  (0,t, 3);
			\addplot3 [green!30,fill opacity=0.5] fill between [of=G and H];
			
		\end{axis}
		
	\end{tikzpicture}
	\begin{tikzpicture}
		\begin{axis}
			[scale=0.6,
			hide axis,
			zmin=0.0,
			zmax=4.5,
			ticks=none,
			enlargelimits=false,
			view={-30}{-45}
			]
			\addplot3[name path=A,color=blue,domain=0:10,samples y=1] (x,0, {exp(-x^2/10)});
			\addplot3[name path=C,color=blue,domain=0:10,samples y=1] (0,x, {exp(-x^2/10)});
			\addplot3[name path=D,color=blue,domain=0:7,samples y=1] (-x,-x, {exp(-x^2/10)+2*x});
			\addplot3[name path=B,variable=t,domain=0:10]  (t,0, 0) node [right, pos=0.990] {$ \begin{pmatrix}
					1\\0\\0
				\end{pmatrix} $};
			\addplot3[name path=E,variable=t,domain=0:10]  (0,t, 0) node [left, pos=0.990] {$ \begin{pmatrix}
					0\\1\\0
				\end{pmatrix} $};
			\addplot3[name path=F,variable=t,domain=0:7]  (-t,-t, 2*t) node [left, pos=0.003] {$ \begin{pmatrix}
					-1\\-1\\2
				\end{pmatrix} $};
			\addplot3 [blue!30,fill opacity=0.5] fill between [of=A and B];
			\addplot3 [blue!30,fill opacity=0.5] fill between [of=C and E];
			\addplot3 [blue!30,fill opacity=0.5] fill between [of=D and F];
			
			\addplot3[name path=G,color=green,domain=0:10,samples y=1] (0,x, {1.5*exp(-x^2/10)});
			\addplot3[name path=J,color=green,domain=0:10,samples y=1] (0,-x, {1.5*exp(-x^2/10)+x});
			\addplot3[name path=H,variable=t,domain=0:10]  (0, t, 0);
			\addplot3[name path=I,variable=t,domain=0:10]  (0,-t, t) node [right, pos=0.0055] {$ \begin{pmatrix}
					0\\-1\\1
				\end{pmatrix} $};
			\addplot3 [green!30,fill opacity=0.5] fill between [of=G and I];
			\addplot3 [green!30,fill opacity=0.5] fill between [of=J and H];
			
		\end{axis}
	\end{tikzpicture}
	\caption{Convex base diagrams corresponding to the restricted Lagrangian fibration related to the polytope given by $\text{Conv}\{(0,0),(0,1),(1,1),(2,0)\} $. }
	\label{fig:nondisplaceable}
\end{figure}
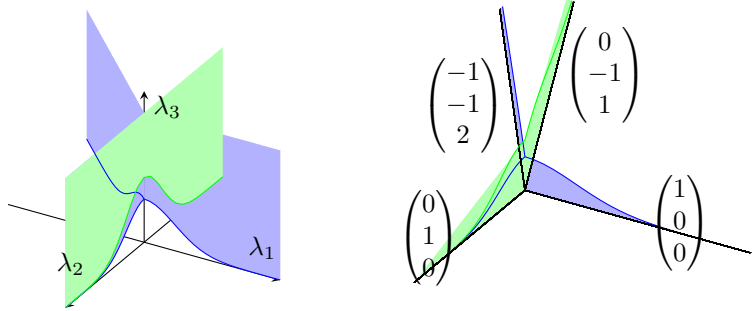

By \eqref{potential}, the potential function is given by: 

\begin{equation}\label{displa}
\mathfrak{PO}=z_{3}(1+z_1+z_2)(1+z_1),
\end{equation}

and using Lemma \ref{critpoint}, we get that \eqref{displa} does not have critical points. 
\begin{rmk}
	In this case, we can not use McDuff's method of probes \cite{mcduff2011displacing} to displace the monotone fibres over the line $\{(0,0,l) | l \in \mathbb{R}_{\geq 0} \}$. So it remains unclear whether the monotone fibres in this example are displaceable.
\end{rmk}

\section{Compactification} \label{secompact}
In this section, we will discuss the existence of a compactification of $Y_{\widetilde{\sigma},\epsilon}$ that is compatible with our construction. Under the assumption of $Q$ giving rise to a smooth toric projective Fano variety $F_Q$, i.e., being the convex hull of the generators of its fan, we get a compactification of $Y_{\widetilde{\sigma},\epsilon}$ as a complete intersection of hyperplanes on a projective variety $X_{\mathscr{A}_{\mathscr{H}}}$. Moreover, this is well adapted with the complex fibration constructed in Theorem \ref{fibration}, in the sense that we get a `pencil' in $\overline{Y_{\widetilde{\sigma},\epsilon}}$, well defined in $\overline{Y_{\widetilde{\sigma},\epsilon}} \setminus D \to \mathbb{P}^1$ where $Y_{\widetilde{\sigma},\epsilon}= \overline{Y_{\widetilde{\sigma},\epsilon}} \setminus \overline{F_{\infty}}$, $\overline{F_{\infty}}=F_{\infty}\cup D$ is the compactification of the fibre at infinity and is the toric projective manifold associated with the polytope $Q$. This pencil extends the fibration $Y_{\widetilde{\sigma},\epsilon}\to \mathbb{C}$ constructed in Theorem \ref{fibration}.\\

Symplectically, this corresponds to an infinite rescaling of the base diagram where the image is shrunk to height one, and the divisor $Q^\vee$ appears at height one. In particular, the singular fibres on $Y_{\widetilde{\sigma},\epsilon}$ are asymptotic to height one.\\

\begin{exmp}
	
	Continuing with the cone over the cubic presented in Section \ref{conecubic}, we discuss its compactification. In Figure \ref{fig:compactQ3}, we show the base diagram of the algebraic compactification. 
	
	\begin{figure}[h!] 
	\centering
	\begin{tikzpicture}
		\begin{axis}
			[scale=0.6,
			hide axis,
			zmin=0.0,
			ticks=none,
			enlargelimits=false,
			view={10}{25}
			]
			\addplot3[name path=A,color=blue,domain=0:7,samples y=1] (0,x, {exp(-x^2/10)+x});
			\addplot3[name path=B,color=blue,domain=0:7,samples y=1] (x,0, {exp(-x^2/10)+x});
			\addplot3[name path=C,color=blue,domain=0:7,samples y=1] (-x,-x, {exp(-x^2/10)+x});
			\addplot3[name path=D,variable=t,domain=0:7]  (0,t, t);
			\addplot3[name path=E,variable=t,domain=0:7]  (t,0, t);
			\addplot3[name path=F,variable=t,domain=0:7]  (-t,-t, t);
			\addplot3 [blue!30,fill opacity=0.5] fill between [of=A and D];
			\addplot3 [blue!30,fill opacity=0.5] fill between [of=B and E];
			\addplot3 [blue!30,fill opacity=0.5] fill between [of=C and F];
			
			\addplot3[name path=G,color=green,domain=0:7,samples y=1] (x,0, {1.5*exp(-x^2/10)+x});
			\addplot3[name path=H,color=green,domain=0:7,samples y=1] (0,x, {1.5*exp(-x^2/10)+x});
			\addplot3[name path=I,color=green,domain=0:7,samples y=1] (-x,-x, {1.5*exp(-x^2/10)+x});
			
			\addplot3[name path=D,variable=t,smooth,draw=none,domain=0:7]  (t,0, t);
			\addplot3[name path=E,variable=t,smooth,draw=none,domain=0:7]  (0,t, t);
			\addplot3[name path=F,variable=t,smooth,draw=none,domain=0:7]  (-t,-t, t);
			
			\addplot3 [green!30,fill opacity=0.5] fill between [of=G and D];
			\addplot3 [green!30,fill opacity=0.5] fill between [of=H and E];
			\addplot3 [green!30,fill opacity=0.5] fill between [of=I and F];
			
			\addplot3[name path=J,color=purple,domain=0:7,samples y=1] (x,0, {2*exp(-x^2/10)+x});
			\addplot3[name path=K,color=purple,domain=0:7,samples y=1] (0,x, {2*exp(-x^2/10)+x});
			\addplot3[name path=L,color=purple,domain=0:7,samples y=1] (-x,-x, {2*exp(-x^2/10)+x});
			
			\addplot3 [purple!30,fill opacity=0.5] fill between [of=J and D];
			\addplot3 [purple!30,fill opacity=0.5] fill between [of=K and E];
			\addplot3 [purple!30,fill opacity=0.5] fill between [of=L and F];
			
			\addplot3[name path=D,variable=t,domain=0:7]  (t,7-t, 7);
			\addplot3[name path=E,variable=t,domain=0:7]  (7-2*t,-t, 7);
			\addplot3[name path=F,variable=t,domain=0:7]  (-7+t,-7+2*t, 7);
			
		\end{axis}
	\end{tikzpicture}
	
	\caption{Compactifications of the cubic.}
	\label{fig:compactQ3}
\end{figure}
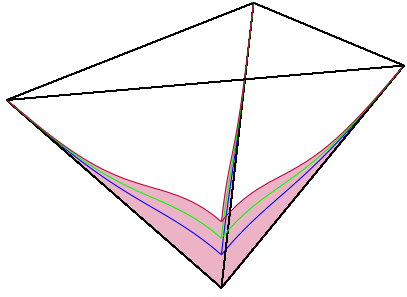
	
\end{exmp}

Let us now prove the following result that ensures the algebraic compactification when $Y_{\widetilde{\sigma}}=Y_{\mathscr{A}_{\mathscr{H}}} $ is the affine cone of the projective variety $ X_{\mathscr{A}_{\mathscr{H}}} $ (see Proposition \ref{hom}).

\begin{thm} \label{comp1}
	With the assumptions given in Theorem \ref{fibration}, suppose that $ Q=M_1+\dots+M_k $ with $ k\geq 2 $, and that $ Y_{\widetilde{\sigma}}=Y_{\mathscr{A}_{\mathscr{H}}} $ is the affine cone of the projective variety $ X_{\mathscr{A}_{\mathscr{H}}} $. Then, the compactification of $ Y_{\sigma,\epsilon} $ in $\mathbb{CP}^{|\mathscr{A}_{\mathscr{H}}|}$ is isomorphic to the intersection of $ k-2 $ different hyperplanes in $ X_{\mathscr{A}_{\mathscr{H}}} $. 
\end{thm}
\begin{proof}
	The compactification of  $ Y_{\widetilde{\sigma},\epsilon} $ is the projective variety $ \overline{Y_{\widetilde{\sigma},\epsilon}} = V(I^h(Y_{\widetilde{\sigma},\epsilon})) $, where $ I^h(Y_{\widetilde{\sigma},\epsilon}) $ is the homogenization of the ideal of $ Y_{\widetilde{\sigma},\epsilon} $. Recall that
	\[ Y_{\widetilde{\sigma},\epsilon}= V(I(Y_{\mathscr{A}_{\mathscr{H}}})) \cap \left( \bigcap_{i=1}^{k-1} V(t_1-t_{i+1}-\epsilon_i) \right) \subseteq \mathbb{C}^{|\mathscr{A}_{\mathscr{H}}|}_{(\dots,t_1,\dots,t_k)} ,\] 
	then 
	\[ \overline{Y_{\widetilde{\sigma},\epsilon}} = V(I(Y_{\mathscr{A}_{\mathscr{H}}})) \cap \left( \bigcap_{i=1}^{k-1} V(t_1-t_{i+1}-\epsilon_iz) \right) \subseteq \mathbb{CP}^{|\mathscr{A}_{\mathscr{H}}|}_{(\dots,t_1,\dots,t_k,z)}.  \] since by Proposition \ref{hom} the ideal $ I(Y_{\mathscr{A}_{\mathscr{H}}}) $ is homogeneous, and $ t_1-t_{i+1}-\epsilon_iz $ is the homogenization of $ t_1-t_{i+1}-\epsilon_i$.\\
	
	Finally, we have the isomorphism: 
	\begin{alignat*}{3}
		F:X_{\mathscr{A}_{\mathscr{H}}} \cap \left( \bigcap_{i=2}^{k-1} V\left(t_1-t_{i+1}-\epsilon_i\left(\frac{t_1-t_2}{\epsilon_1}\right)\right) \right)&\longrightarrow \overline{Y_{\widetilde{\sigma},\epsilon}}\\
		[\dots:t_1:\dots:t_k]&\longmapsto\left[\dots:t_1:\dots:t_k:\frac{t_1-t_2}{\epsilon_1}\right]
	\end{alignat*}.
	
\end{proof}

In this setting, we can consider the complex fibration  
\begin{alignat*}{3}
	f:\overline{Y_{\widetilde{\sigma},\epsilon}}&\longrightarrow \mathbb{CP}^1 \\
	[\dots:t_1:\dots:t_k:z]&\longmapsto [t_1:z].
\end{alignat*}

Note that for $ z\neq 0 $, it coincides with the complex fibration we constructed before, and that $ f^{-1}([1:0])= X_{\sigma} \setminus V(t_1) $.\\

Our toy example in section \ref{Q5} satisfies the conditions of Theorem \ref{comp1}. By examining the equations, we conclude that $ \overline{Y_{\widetilde{\sigma},\epsilon}} $ is equal to the blow-up of $ \mathbb{CP}^3 $ in one point.\\

In the example discussed in Section \ref{firstdecomp}, $Y_{\epsilon,\widetilde{\sigma}}$ compactifies to $\mathbb{P}^1\times \mathbb{P}^1 \times \mathbb{P}^1$ (See Figure \ref{fig:compactQ61d}). Indeed, one recovers $Y_{\epsilon,\widetilde{\sigma}}$ by deleting a smooth fibre of the "pencil" 
\begin{alignat*}{3}
	(\mathbb{P}^1)^3 & \to \mathbb{P}^1 \\
	([x_0:x_1],[y_0:y_1],[z_0,z_1]) & \mapsto [x_1y_1z_1:x_0y_0z_0].
\end{alignat*}  
Here, the singular fibres are at $0$ and $\infty$ instead of $0$ and $\epsilon$. We see in \eqref{potentialQ62d} that
\[\mathfrak{PO}=z_3(1+z_1+z_2+z_1^2z_2+z_1z_2^2+z_1^2z_2^2+3z_1z_2)\]
where $3z_1z_2$ is the term corresponding to the center of the polytope $Q$. According to \cite{diogoprep}, the Lagrangians $\mathcal{L}_a$ viewed in the compactification $Y_{\epsilon,\widetilde{\sigma}} \subseteq (\mathbb{P}^1)^3$ are lifts of the monotone Lagrangian $\mathcal{L}'$ on the compactification of the divisor $F_{\infty}$, viewed as the added fibre at $\infty$ on the 'pencil' over the compactification, which is a toric manifold with moment polytope $Q^\vee$, i.e., $\text{Bl}_3 \mathbb{C}P^2$. Its potential is then given by $T^b\mathfrak{PO} +T^a(z_1z_2z_3)^{-1} $; where $(z_1z_2z_3)^{-1}$ corresponds to the fibre disk passing through the compactification of the fibre $F_\infty$, with area $a$. $(\mathfrak{PO}-3z_1z_2z_3)$ correspond the terms in the potential of $\mathcal{L}'$ and $3z_1z_2z_3$ is the extra term, where $3$ is reinterpreted in \cite{diogoprep} as the number of $c_1=1$ rational curves in $(\mathbb{P}^1)^3$ passing transversally through a point in the divisor $F_{\infty}$.

\begin{figure}[h!] 
	\centering
	\begin{tikzpicture}
		\begin{axis}
			[scale=0.6,
			hide axis,
			zmin=0.0,
			ticks=none,
			enlargelimits=false,
			view={10}{20}
			]
			\addplot3[name path=A,color=blue,domain=0:7,samples y=1] (0,x, {exp(-x^2/10)+x});
			\addplot3[name path=B,color=blue,domain=0:7,samples y=1] (-x,0, {exp(-x^2/10)+x});
			\addplot3[name path=C,color=blue,domain=0:7,samples y=1] (x,-x, {exp(-x^2/10)+x});
			
			\addplot3[name path=D,variable=t,domain=0:7, draw=none]  (0,t, t);
			\addplot3[dashed] coordinates {(0,0,0) (0,7,7)};
			
			\addplot3[name path=E,variable=t,domain=0:7]  (-t,0, t);
			\addplot3[name path=F,variable=t,domain=0:7]  (t,-t, t);
			\addplot3 [blue!30,fill opacity=0.5] fill between [of=A and D];
			\addplot3 [blue!30,fill opacity=0.5] fill between [of=B and E];
			\addplot3 [blue!30,fill opacity=0.5] fill between [of=C and F];
			
			\addplot3[name path=A,color=green,domain=0:7,samples y=1] (-x,x, {1.5*exp(-x^2/10)+x});
			\addplot3[name path=B,color=green,domain=0:7,samples y=1] (0,-x, {1.5*exp(-x^2/10)+x});
			\addplot3[name path=C,color=green,domain=0:7,samples y=1] (x,0, {1.5*exp(-x^2/10)+x});
			\addplot3[name path=D,variable=t,domain=0:7, draw=none]  (-t,t, t);
			\addplot3[dashed] coordinates {(0,0,0) (-7,7,7)};
			
			\addplot3[name path=E,variable=t,domain=0:7]  (0,-t, t);
			\addplot3[name path=F,variable=t,domain=0:7]  (t,0, t);
			\addplot3 [green!30,fill opacity=0.5] fill between [of=A and D];
			\addplot3 [green!30,fill opacity=0.5] fill between [of=B and E];
			\addplot3 [green!30,fill opacity=0.5] fill between [of=C and F];
			
			\addplot3[variable=t, domain=0:7]  (t,7-t, 7);
			\addplot3[variable=t, domain=0:7]  (-t,7, 7);
			\addplot3[variable=t, domain=0:7]  (-7,7-t, 7);
			\addplot3[variable=t, domain=0:7]  (-7+t,-t, 7);
			\addplot3[variable=t, domain=0:7]  (t,-7, 7);
			\addplot3[variable=t, domain=0:7]  (7,-7+t, 7);
			
		\end{axis}
	\end{tikzpicture}
	\begin{tikzpicture}
		\begin{axis}
			[scale=0.6,
			hide axis,
			zmin=0.0,
			ticks=none,
			enlargelimits=false,
			view={-50}{30}
			]
			\addplot3[name path=A,color=blue,domain=0:7,samples y=1] ({exp(-x^2/10)+x},{-exp(-x^2/10)}, {exp(-x^2/10)});
			\addplot3[name path=B,color=blue,domain=0:7,samples y=1] ({exp(-x^2/10)},{-exp(-x^2/10)-x}, {exp(-x^2/10)});
			\addplot3[name path=C,color=blue,domain=0:7,samples y=1] ({exp(-x^2/10)},{-exp(-x^2/10)}, {exp(-x^2/10)+x});
			
			\addplot3[name path=G, color=blue, domain=0:7, samples y=1, draw=none] ({exp(-x^2/10)},{-exp(-x^2/10)}, {exp(-x^2/10)});
			
			\addplot3[name path=D, variable=t, domain=0:7, draw=none]  (t,0, 0);
			\addplot3[dashed] coordinates {(0,0,0) (7,0,0)};
			\addplot3[name path=E, variable=t, domain=0:7]  (0,-t, 0);
			\addplot3[name path=F, variable=t, domain=0:7]  (0,0, t);
			\addplot3 [blue!30,fill opacity=0.5] fill between [of=A and D];
			\addplot3 [blue!30,fill opacity=0.5] fill between [of=B and E];
			\addplot3 [blue!30,fill opacity=0.5] fill between [of=C and G];
			
			\addplot3[name path=A, color=green, domain=0:7, samples y=1] ({7-exp(-x^2/10)-x},{exp(-x^2/10)-7}, {-exp(-x^2/10)+7});
			\addplot3[name path=B,color=green,domain=0:7,samples y=1] ({7-exp(-x^2/10)},{exp(-x^2/10)-7+x}, {-exp(-x^2/10)+7});
			\addplot3[name path=C,color=green,domain=0:7,samples y=1] ({7-exp(-x^2/10)},{exp(-x^2/10)-7}, {-exp(-x^2/10)+7-x});
			
			\addplot3[name path=G, domain=0:7, samples y=1, draw=none] ({7-exp(-x^2/10)},{exp(-x^2/10)-7}, {-exp(-x^2/10)+7});
						
			\addplot3[name path=D,variable=t,domain=0:7]  (7-t,-7,7);
			\addplot3[name path=E,variable=t,domain=0:7]  (7,-7+t,7);
			\addplot3[name path=F,variable=t,domain=0:7]  (7,-7,7- t);
			\addplot3 [green!30,fill opacity=0.5] fill between [of=A and D];
			\addplot3 [green!30,fill opacity=0.5] fill between [of=B and E];
			\addplot3 [green!30,fill opacity=0.5] fill between [of=C and G];
			
			\addplot3[dashed] coordinates {(7,0,0) (7,0,7)};
			\addplot3[dashed] coordinates {(7,0,0) (7,-7,0)};
			\addplot3[] coordinates {(0,-7,0) (0,-7,7)};
			\addplot3[] coordinates {(0,0,7) (7,0,7)};
			\addplot3[] coordinates {(0,0,7) (7,0,7)};
			\addplot3[] coordinates {(0,0,7) (0,-7,7)};
			\addplot3[] coordinates {(7,-7,0) (0,-7,0)};
			
		\end{axis}
	\end{tikzpicture}
	
	\caption{Transferring one set of trivalent cuts from the compactified cone over the hexagon, one gets the polytope of $(\mathbb{P}^1)^3$}
	\label{fig:compactQ61d}
\end{figure}
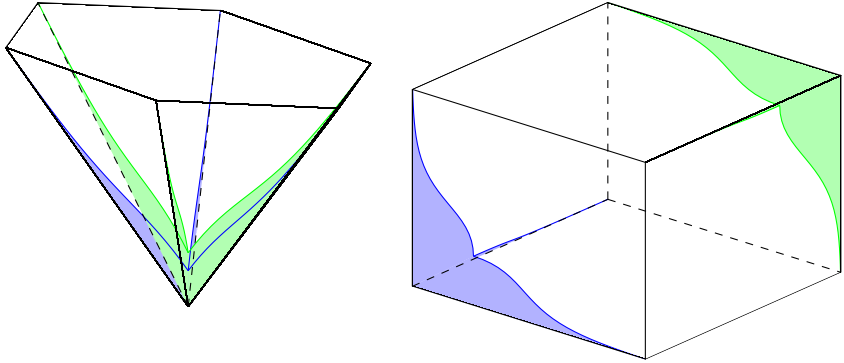

In the example discussed in Section \ref{seconddecompQ6}, it can be seen that $Y_{\widetilde{\sigma},\epsilon}$ compactifies to $Fl(3)$ in this case. One can identify the constructed Lagrangian fibration as a modification of the standard Gelfand-Cetlin polytope in the cases of $Fl(3)$ as indicated in Figure \ref{fig:compactQ62d}.
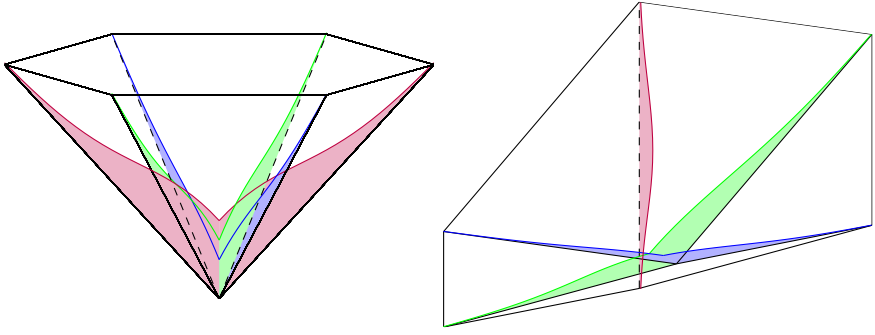
\begin{figure}
\begin{tikzpicture}
		\begin{axis}
			[scale=0.6,
			hide axis,
			zmin=0.0,
			ticks=none,
			enlargelimits=false,
			view={-45}{20}
			]
			\addplot3[name path=A,color=blue,domain=0:6,samples y=1] (0,x, {exp(-x^2/10)+x});
			\addplot3[name path=B,color=green,domain=0:6,samples y=1] (-x,0, {1.5*exp(-x^2/10)+x});
			\addplot3[name path=C,color=purple,domain=0:6,samples y=1] (x,-x, {2*exp(-x^2/10)+x});
			\addplot3[name path=D,variable=t,domain=0:6, draw=none]  (0,t, t);
			\addplot3[dashed] coordinates {(0,0,0) (0,6,6)};
			\addplot3[name path=E,variable=t,domain=0:6]  (-t,0, t);
			\addplot3[name path=F,variable=t,domain=0:6]  (t,-t, t);
			\addplot3 [blue!30,fill opacity=0.5] fill between [of=A and D];
			\addplot3 [green!30,fill opacity=0.5] fill between [of=B and E];
			\addplot3 [purple!30,fill opacity=0.5] fill between [of=C and F];
			
			\addplot3[name path=A,color=purple,domain=0:6,samples y=1] (-x,x, {2*exp(-x^2/10)+x});
			\addplot3[name path=B,color=blue,domain=0:6,samples y=1] (0,-x, {exp(-x^2/10)+x});
			\addplot3[name path=C,color=green,domain=0:6,samples y=1] (x,0, {1.5*exp(-x^2/10)+x});
			\addplot3[name path=D,variable=t,domain=0:6]  (-t,t, t);
			\addplot3[name path=E,variable=t,domain=0:6]  (0,-t, t);
			\addplot3[name path=F,variable=t,domain=0:6, draw=none]  (t,0, t);
			\addplot3[dashed] coordinates {(0,0,0) (6,0,6)};
			\addplot3 [purple!30,fill opacity=0.5] fill between [of=A and D];
			\addplot3 [blue!30,fill opacity=0.5] fill between [of=B and E];
			\addplot3 [green!30,fill opacity=0.5] fill between [of=C and F];
			
			\addplot3[variable=t, domain=0:6]  (t,6-t, 6);
			\addplot3[variable=t, domain=0:6]  (-t,6, 6);
			\addplot3[variable=t, domain=0:6]  (-6,6-t, 6);
			\addplot3[variable=t, domain=0:6]  (-6+t,-t, 6);
			\addplot3[variable=t, domain=0:6]  (t,-6, 6);
			\addplot3[variable=t, domain=0:6]  (6,-6+t, 6);
			
		\end{axis}
	\end{tikzpicture}
	\begin{tikzpicture}
		\begin{axis}
			[scale=0.6,
			hide axis,
			zmin=0.0,
			ticks=none,
			enlargelimits=false,
			view={-50}{10}
			]
			
			\addplot3[dashed, name path=purple1] coordinates {(0,0,0) (0,0,6)};
			\addplot3[] coordinates {(-2,0,0) (-2,0,2)};
			\addplot3[] coordinates {(-2,0,0) (0,0,0)};
			\addplot3[] coordinates {(0,0,6) (-2,0,2)};
			\addplot3[name path=blue1] coordinates {(-2,-1,2) (-2,0,2)};
			\addplot3[name path=green2] coordinates {(-2,-1,2) (-2,0,0)};
			\addplot3[] coordinates {(0,-1,6) (0,0,6)};
			\addplot3[name path=green1] coordinates {(0,-1,6) (-2,-1,2)};
			
			\addplot3[] coordinates {(0,-1,6) (0,-1,2)};
			\addplot3[] coordinates {(0,0,0) (0,-1,2)};
			\addplot3[name path=blue2] coordinates {(-2,-1,2) (0,-1,2)};
			
			\addplot3[name path=A, color=green, domain=0:1, samples y=1] ({-2+2*x+0.2*exp(-(2*x)^2)},{-1+0.2*exp(-(2*x)^2)}, {2+4*x});
			
			\addplot3[name path=B,color=green,domain=0:1,samples y=1] (-2+0.2*exp(-(2*x)^2),{-1+x+0.2*exp(-(2*x)^2)}, {2-2*x});
			
			\addplot3[name path=C, color=blue, domain=0:1,samples y=1] (-2+0.2*exp(-(2*x)^2),{-1+x+0.14*exp(-(2*x)^2)}, {2});
			
			\addplot3[name path=D, color=blue, domain=0:1,samples y=1] (-2+0.2*exp(-(2*x)^2)+2*x,{-1+0.14*exp(-(2*x)^2)}, {2});
			
			\addplot3[name path=E, color=purple, domain=0:1,samples y=1] ({-0.1*exp(-(3*(x-0.5))^2}, {-0.1*exp(-(3*(x-0.5))^2}, 6*x);
			
			\addplot3 [green!30,fill opacity=0.5] fill between [of=A and green1];
			\addplot3 [green!30,fill opacity=0.5] fill between [of=B and green2];
			\addplot3 [blue!30,fill opacity=0.5] fill between [of=C and blue1];
			\addplot3 [blue!30,fill opacity=0.5] fill between [of=D and blue2];
			\addplot3 [purple!30,fill opacity=0.5] fill between [of=E and purple1];
			
		\end{axis}
	\end{tikzpicture}

	\caption{Transferring one set of cuts from the compactified cone over the hexagon, one gets a modification of the standard Gelfand-Cetlin polytope in the cases of $Fl(3)$}
	\label{fig:compactQ62d}
\end{figure}

Also for $Y_{\widetilde{\sigma},\epsilon}$ we get that the potential \eqref{potentialQ62d} is:
\[\mathfrak{PO}=z_3(1+z_1+z_2+z_1^2z_2+z_1z_2^2+z_1^2z_2^2+2z_1z_2)\]
and for the corresponding Lagrangian $\mathcal{L}_a$ viewed in the compactification $\overline{Y_{\widetilde{\sigma},\epsilon}} \subseteq Fl(3)$ as a lift of the moment Lagrangian $\mathcal{L}'$ on the compactifying divisor $F_\infty$, the potential is 
$T^b\mathfrak{PO}+T^a(z_1z_2z_3)^{-1}$ where $(z_1z_2z_3)^{-1}$ corresponds to the fibre disk, passing through the compactifying fibre $F_{\infty}$, with area $a$. Similarly to the previous example, $(\mathfrak{PO}-2z_1z_2z_3)$ corresponds the term in the potential of $\mathcal{L}'$ and $2z_1z_2z_3$ is the additional term, where 2 is reinterpreted in \cite{diogoprep} as the number of $c_1=1$ rational curves in $Fl(3)$ passing transversally through a point in the divisor $F_{\infty}$.

\section{Future Developments}\label{future}

In this section, we mention some developments that arise naturally from this work.
\begin{enumerate}
	\item Lagrangian skeleton of $Y_{\widetilde{\sigma},\epsilon}$ and its relation to Gelfand-Cetlin fibrations.\\

The affine smooth quadric $Y_{\widetilde{\sigma},\epsilon} =T^*S^3$ arise as the smoothing of the singular quadric $Y_\sigma$, where $\sigma$ is the cone over the square with sides of length 1. The cone $\sigma^\vee$ describes a symplectic torus fibration of the singular quadric, but also the same cone describes a Gelfand-Cetlin fibration (in the sense of \cite{shelukhin2018geometry}) of    $Y_{\widetilde{\sigma},\epsilon}$ with a Lagrangian $S^3$ in the vertex. This fibration can be regarded as a limit of a family of restricted almost toric fibrations constructed in Section \ref{singlagfib}. We expect to generalize this example to obtain Gelfand-Cetlin fibrations (in the sense of \cite{shelukhin2018geometry}) as limits of the fibrations in Section \ref{singlagfib}. These Gelfand-Cetlin fibrations have the cone $\sigma^\vee$ as a convex base diagram, and the singular non-toric fibres are in the codimension $\geq 2$ faces. These varieties are affine, so they are Stein and Weinstein, and they have a Lagrangian skeleton such that $Y_{\widetilde{\sigma},\epsilon}$ deforms to it. We expect to describe the Lagrangian skeleton in terms of the Minkowski decomposition of $Q$, to describe a Weinstein structure of $Y_{\widetilde{\sigma},\epsilon}$ from the Gelfand-Cetlin fibration, and to compute symplectic homology and the wrapped Fukaya category for these spaces.
	\item Homological mirror symmetry for $Y_{\widetilde{\sigma},\epsilon}$.\\

We are interested in proving homological mirror symmetry results for $Y_\epsilon$ and $Y_\epsilon \setminus \pi^{-1}(1)$. This project can be divided into the following steps. First, construct the mirror following the SYZ approach in \cite{auroux2007mirror, auroux2009special}, or following the construction in \cite{abouzaid2016lagrangian} (see also \cite{chan2016lagrangian, chan2013homological, chan2012syz}). Second, study the Fukaya category and wrapped Fukaya category of $Y_\epsilon$. We expect to use our complex fibration $\pi$ in Theorem \ref{fibration} and results about Lefschetz fibrations \cite{seidel2008fukaya} to achieve this goal (see also \cite{abouzaid2020monotone}). Third, relate Lagrangians of the Fukaya category with holomorphic line bundles on the mirror using the results in \cite{arinkin1998fukaya, leung}.
\end{enumerate}

\bibliographystyle{amsalpha}
\bibliography{biblio}

\end{document}